\documentclass{amsart}
\usepackage[utf8]{inputenc}
\usepackage{multirow}
\usepackage[margin=1.2 in]{geometry}
\usepackage{amssymb,bm,amsfonts, stmaryrd, amsmath, amsthm, enumerate, mathtools,comment,quiver}
\usepackage[all]{xy}
\usepackage{csquotes}
\usepackage{colonequals}
\usepackage{enumitem}

\usepackage{verbatim}
\usepackage{wrapfig, tikz, tikz-cd}
\usepackage{caption}
\usepackage{subcaption}
\usepackage{float}
\usetikzlibrary{arrows}
\usepackage{blkarray}
\newcommand{\matlabel}[1]{ \makebox[15pt]{\text{\small{$#1$}}} }

\makeatletter
\@namedef{subjclassname@2020}{\textup{2020} Mathematics Subject Classification}
\makeatother
\usepackage[colorlinks]{hyperref}
\definecolor{chianti}{rgb}{0.6,0,0}
\definecolor{meretale}{rgb}{0,0,.6}
\definecolor{leaf}{rgb}{0,.35,0}

\definecolor{mypink}{RGB}{216, 27, 96}
\definecolor{myblue}{RGB}{30, 136, 229}
\definecolor{myyellow}{RGB}{169, 126, 0}
\definecolor{mygreen}{RGB}{0, 160, 64}
\hypersetup{
 colorlinks=true,
 linkcolor=chianti,
 filecolor=magenta, 
 urlcolor=leaf,
 citecolor=meretale,
}

\usepackage{wrapfig,tikz, tikz-cd}
\usetikzlibrary{arrows}
\usepackage[capitalise]{cleveref}
\crefformat{equation}{(#2#1#3)}
\crefrangeformat{equation}{(#3#1#4--#5#2#6)}
\crefformat{enumi}{(#2#1#3)}
\crefrangeformat{enumi}{(#3#1#4--#5#2#6)}
\newtheorem{theorem}{Theorem}[section]
\newtheorem{lemma}[theorem]{Lemma}
\newtheorem{proposition}[theorem]{Proposition}
\newtheorem{corollary}[theorem]{Corollary}
\newcounter{intro}
\newtheorem{questionx}{Question}
\newtheorem{introthm}[intro]{Theorem}

\theoremstyle{definition}
\newtheorem{definition}[theorem]{Definition}
\newtheorem{example}[theorem]{Example}
\newtheorem{remark}[theorem]{Remark}

\newtheorem{chunk}[theorem]{}

\newtheorem*{theorem*}{Theorem}
\newtheorem*{definition*}{Definition}
\newtheorem*{lemma*}{Lemma}
\newtheorem*{proposition*}{Proposition}
\newtheorem*{corollary*}{Corollary}
\newtheorem{question}[theorem]{Question}

\theoremstyle{definition}

\newcommand{\pdim}{{\operatorname{pdim}}}
\newcommand{\ges}{\geqslant}
\newcommand{\les}{\leqslant}

\newcommand{\Spec}{{\operatorname{Spec}}}

\newcommand{\Ext}{{\operatorname{Ext}}}

\renewcommand{\S}{\mathcal{S}}

\DeclareMathOperator{\id}{id}
\DeclareMathOperator{\DB}{DB}
\DeclareMathOperator{\WT}{WT}
\DeclareMathOperator{\spank}{span_k}
\DeclareMathOperator{\lcm}{lcm}
\newcommand{\supp}{{\operatorname{Supp}}}

\newcommand{\f}{\bm{f}}
\newcommand{\g}{\bm{g}}
\newcommand{\V}{{\mathcal{V}}}
\newcommand{\I}{{\mathcal{I}}}
\newcommand{\setbuild}[2]{\left\{ {#1} \;:\;{#2}\right\}}

\newcommand{\cV}{{\mathcal{V}}}

\DeclareMathOperator{\im}{im}

\DeclareMathOperator{\rank}{rank}

\newcommand{\h}{\mathtt{h}}
\newcommand{\e}{\mathtt{e}}
\newcommand{\dd}{\mathtt{d}}
\newcommand{\T}{\mathtt{T}}
\newcommand{\G}{\mathtt{G}}
\newcommand{\Aux}{\mathtt{A}}

\newcommand{\m}{\mathfrak{m}}

\newcommand{\p}{\mathfrak{p}}
\newcommand{\A}{\mathbb{A}}

\DeclareMathOperator{\sign}{sign}

\DeclareMathOperator{\Kos}{Kos}

\DeclareMathOperator{\Sym}{Sym}

\newcommand{\Q}{\mathbb{Q}}

\newcommand{\N}{\mathbb{N}}

\newcommand{\q}{\mathfrak q}

\newcommand{\ognum}[1]{}%\color{blue}{(#1)}\color{black}}

\newcommand*\circled[1]{\tikz[baseline=(char.base)]{\node[shape=circle,draw,inner sep=2pt] (char) {#1};}}
\def\VR{\kern-\arraycolsep\strut\vrule &\kern-\arraycolsep}
\def\vr{\kern-\arraycolsep & \kern-\arraycolsep}

\title{Cohomological support varieties for monomial ideals}
\author[K. Fagerstrom] {Kara Fagerstrom}
\address{University of Nebraska Lincon, NE 68588. U.S.A.}
\email{kfagerstrom2@huskers.unl.edu}
\urladdr{https://math.unl.edu/person/kara-fagerstrom/}

\author[J. Faur] {Julianne Faur}
\address{University of Nebraska-Lincoln, NE 68588. U.S.A.}
\email{juliannefaur@huskers.unl.edu}
\urladdr{https://juliannefaur.github.io}

\author[B. Katz]{Benjamin Katz}
\address{University of Nebraska-Lincoln, NE 68588. U.S.A.}
\email{bkatz2@huskers.unl.edu}
\urladdr{https://math.unl.edu/person/ben-katz/}

\author[K. Mohana Sundaram]{Kesavan Mohana Sundaram}
\address{University of Nebraska-Lincoln, NE 68588. U.S.A.}
\email{km2@huskers.unl.edu}
\urladdr{https://kesavan-ms.github.io/}

\author[S. Stern]{Stephen Stern}
\address{University of Nebraska-Lincoln, NE 68588. U.S.A.}
\email{sstern2@huskers.unl.edu}
\urladdr{https://sterns.github.io}

\author[R. Watson]{Ryan Watson}
\address{University of Nebraska-Lincoln, NE 68588. U.S.A.}
\email{rwatson9@huskers.unl.edu}
\urladdr{https://rawatson1997.github.io}


\begin{document}

\begin{abstract}
    Let $R$ be a local or positively graded ring with a regular presentation $R \cong Q/I$ where $I$ is a monomial ideal generated by $n$ elements on a regular sequence. In \cite{BGP} the authors classify the cohomological support varieties $\cV_R(R)$ for $n \les 5$. In this paper we extend their results to classify the varieties that can occur as $\cV_R(R)$ for $n=6$. 
    Moreover, we provide two families of rings, one realizing cohomological support varieties of unbounded codimension, the other realizing an unbounded number of components.
    Finally, we answer a question of \cite{gintz} about the varieties that occur as $\cV_R(R)$ where $I$ is given by the edge ideal of a cycle.
\end{abstract}
\maketitle
\vspace{-2em}
\section*{Introduction}
Inspired by Quillen's geometric approach to modular representation theory \cite{Quillen:1971}, Avramov imported the theory of support varieties to commutative algebra to study local complete intersection rings \cite{Avramov:1989}. 
This has now expanded to include all noetherian local rings through the work of various authors \cite{Jorgensen:2002,Burke/Walker:2015,Pollitz:2019, Pollitz:2021}. 
Most notably for this paper, Pollitz developed the theory of cohomological support varieties over Koszul complexes which generalizes the theory of Avramov.

Suppose that $R$ is a local or positively graded ring. Then $R$ is said to admit a minimal regular presentation, if one of the following holds:
    \begin{itemize}
        \item $\widehat{R} = Q/I$ with $(Q,\m, k)$ a regular local ring and $I \subset \m^2$ in the local case, or
        \item $R = Q/I$ where $Q$ is a positively graded polynomial algebra over a field $k$ and $I$ is an ideal generated by homogeneous forms of degree at least 2 in the graded case.
    \end{itemize}

Assume that $R$ has a minimal regular presentation given by $Q/I$, where $I$ is minimally generated by $n$ elements. Let $E$ be the Koszul complex on a minimal generating set of $I$. The ring of cohomological operators $\S = k[\chi_1,\dotsc, \chi_n]$ is a polynomial ring over $k$ where each $\chi_i$ is a variable of cohomological degree $2$. For any finitely generated $R$-module $M$, the ring $\S$ acts on $\Ext_E(M,k)$, making $\Ext_E(M,k)$ into a graded $\S$-module. There are many different ways to define this action, but one is explained as follows. The ring $\S$ is a graded $k$-subalgebra of $\Ext_E(k,k)$, and the left action of $\Ext_E(k,k)$ on $\Ext_E(M,k)$ defines an action of $\S$ on $\Ext_E(M,k)$ (see \cite[Section 2]{Avramov/Buchweitz:2000a}). The specific details of how $\S$ acts on this $\Ext_E(M,N)$ are not needed for this paper, so we refer the interested reader to \cite[3.2.6]{Pollitz:2019}. Due to our conditions on $Q$, it turns out that $\Ext_E(M,k)$ is a finitely generated graded $\S$ module \cite[3.2.5]{Pollitz:2019}.

The cohomological support variety of a finitely generated $R$ module $M$, is defined to be
\[
    \cV_R(M) := \supp_{\S}\left( \Ext_E(M,k) \right).
\]

As $\Ext_E(M,k)$ is a finitely generated graded $\S$ module, we can and will view $\cV_R(M)$ as a conical affine variety (i.e. an affine cone), living in $\A^n_k$ which we identify with $\m \Spec (\S)$. 
Geometric properties of these varieties encode homological information of both the ring and the module. For example, they see some of the structure of the lattice of thick subcategories in the derived category of $R$ \cite{Stevenson:2014, Pollitz:2019}. 

When the ring is a complete intersection, the theory of support varieties is generally well-understood and has been studied extensively \cite{Avramov:1989,Avramov/Buchweitz:2000b, Bergh:2007, Avramov/Iyengar:2007,Dao/Sanders:2017}. However, moving away from the complete intersection case, their behavior becomes much more mysterious. In \cite{Pollitz:2021}, the author asks the following question known as the \textit{realization problem} for cohomological support varieties:
\begin{questionx}
    Given a noetherian local ring $R$, what conical affine varieties in $\A^n_k$ can be realized as the support variety of some finitely generated $R$-module?
\end{questionx}
This problem is solved when the ring is a complete intersection \cite{Bergh:2007,Avramov/Iyengar:2007}, with the answer being that every conical affine variety is realizable. However, when the ring is not a complete intersection, much less is known. There are bounds on the dimensions of the varieties that can occur \cite{Briggs/Grifo/Pollitz:2024, BGP}, and certain varieties can be guaranteed to be realizable over certain rings in \cite{Pollitz:2021,Briggs/Grifo/Pollitz:2022,Watson:2025,gintz}. Beyond these results, little is known in general. In fact, it is unknown what varieties can arise as $\cV_R(R)$. Some results in this direction include the following. When $R$ is a complete intersection, $\cV_R(R) =\{0\}$ \cite{Avramov:1989}, and when $R$ is a Golod ring of codepth at least two, $\cV_R(R)$ is all of $\A^n_k$ \cite{Briggs/Grifo/Pollitz:2024}. Moreover, the possible support varieties arising from $\cV_R(R)$ when $R$ has small codepth is contained in \cite{Pollitz:2021}. In \cite{BGP} the authors classify all varieties that can be realized as $\cV_R(R)$ when $R$ is defined by a monomial ideal generated by five elements on a regular sequence of $Q$. These are exactly linear subspaces with dimension not 1 and the union of two hyperplanes. In \cite{gintz}, the author computes the support varieties of equigenerated monomial ideals with six generators on a regular sequence of $Q$ when the residue field is $\Q$. In this paper, we extend the results of \cite{BGP} and \cite{gintz} to all rings defined by six monomial relations with no restrictions on the residue field. 

\begin{introthm}\label{theorem_A}
    Let $R$ be a local or positively graded ring with minimal regular presentation given by $Q/I$ where $I$ is a monomial ideal generated by six elements on a regular sequence of $Q$. Then up to a reordering of the monomial generators of $I$, the varieties realized as $\cV_R(R)$ are listed below:
    \begin{itemize}
        \item a coordinate subspace of $\mathbb A^6_k$ with dimension not equal to $1$
        \item $V(\chi_1,\;\chi_4\chi_6)$ 
        \item $V(\chi_4\chi_6)$
        \item $V(\chi_4\chi_6,\; \chi_5\chi_6)$
        \item $V(\chi_2\chi_4\chi_6)$
        \item $V(\chi_1\chi_3\chi_5+\chi_2\chi_4\chi_6)$.
    \end{itemize}
\end{introthm}
When dealing with rings defined by monomial ideals, one has to go all the way to five generators to get a ring whose support variety is not a linear subspace \cite{BGP}. The only such variety is given by $V(\chi_1\chi_5)$ which is a union of hyperplanes which can be realized by $R = k\llbracket x,y,z,w \rrbracket/I$ where $I = (x^2,xy,yz,zw,w^2)$. Our result shows that one need only go up to six generators to attain a non-linear hypersurface, namely $V(\chi_1\chi_3\chi_5+\chi_2\chi_4\chi_6)$. 
Expanding beyond this variety, we obtain the following theorem which answers a question of \cite{gintz}, about the support variety of $R$ where $R$ is defined by the edge ideal of a cycle.

\begin{introthm}\label{theorem_B}
    Let $R=k[x_1,\dotsc,x_n]/(x_1x_2,x_2x_3, \dotsc, x_nx_1)$, the ring defined by the edge ideal of an $n$-cycle. Then
    \[\V_{R}(R) = \begin{cases}
    V(\chi_1\chi_3\ldots\chi_{n-1}+\chi_2\chi_4\ldots\chi_{n}) & \text{if} \ n\equiv2  \pmod {4}\\
     \A_k^n & otherwise.
    \end{cases}\]
\end{introthm}
Note that specializing to the case $n=6$ gives the non-linear hypersurface mentioned above.

The structure of the paper is as follows. In Section \ref{section preliminaries} we go over the necessary background information for computing support varieties over rings defined by monomial ideals. The general method of computing support varieties in this paper is taken from \cite{BGP} and uses the theory of GCD and Taylor graphs. However, in this paper, we employ many combinatorial techniques not present in \textit{loc.cit.}. In Section \ref{section 6 generated ideals}, we deal with support varieties of rings defined by monomial ideals with six generators and prove Theorem \ref{theorem_A} as Theorem \ref{theorem A.2}. In Sections \ref{section families} and \ref{section cycle varieties}, we consider different families of graphs which give rise to rings whose support varieties we can compute. The main results in Section \ref{section families} give three families of non complete intersection rings. The first two families give rings whose support varieties attain arbitrarily large codimension, and the third family gives rise to rings whose support varieties are a union of $n$ hyperplanes in $\A^{2n}_k$ for $n\ges 3$. Lastly, in Section \ref{section cycle varieties}, we compute the support varieties of rings defined by of ideals including edge ideals of cycles, proving Theorem \ref{theorem_B} as Theorem \ref{theorem cycle edge ideal supports}.

\section{Preliminaries}\label{section preliminaries}
%%%%%%%%%%%%%%%%%%%%%%%%%%%%%%%%%%%%

Throughout this paper, we assume that $R$ is either a local ring or a positively graded ring with residue field $k$, admitting a minimal regular presentation $Q/I$. Moreover, we assume the ideal $I$ is minimally generated by the ordered list of monomials $\f = f_1, \ldots, f_n$ on a regular sequence $x_1,\ldots, x_r$ of $Q$. That is, each $f_i$ is a product of powers of the $x_j$.
Our object of interest is the support variety $\V_{\f}:= \V_R(R)$. 
In this setting, \cite{BGP} furnishes a combinatorial construction of $\V_{\f}$ in terms of two graphs that are determined by $\f$, the GCD graph $\G_{\f}$ and the Taylor graph $\T_{\f}$. 
Moreover, for a fixed number of generators, only finitely many $\G_{\f}$ and $\T_{\f}$ can be realized, hence only finitely many support varieties can occur. 

\begin{definition} \label{definition gcd_graph}
The \emph{GCD graph} of $\f$ is the simple graph $\G_{\f}$ with vertex set $[n] = \{f_1 , \ldots , f_n\}$ such that there is an edge $\{f_i,f_j\}$ if and only if $\gcd(f_i, f_j)$ is not a unit. For a set of vertices $\sigma\subseteq [n]$ we denote the \emph{neighborhood} of $\sigma$ as $N_{\f}(\sigma) =\setbuild{f_i}{\{f_i,f_j\}\in \G_{\f},\; j\in \sigma}.$
\end{definition}

\begin{example}\label{example first gcd graphs}
Consider two monomial ideals generated by
    \[ \f_1 = (bce, bd,cd,ae) \ \text{and} \ \f_2 = (de, ae, be, cd).\] In both cases, $\G_{\f}$ is a triangle on $\{f_1,f_2,f_3\}$ along with the edge $\{f_1,f_4\}$. 
    Hence, $N_{\f}(f_1) = \{f_2,f_3,f_4\}$. 
    This also illustrates that a fixed $\G_{\f}$ can have multiple monomial ideals attaining it. 
\end{example}

A priori, for a given $\G_{\f}$ there are infinitely many monomial ideals to consider; our first step is to reduce to considering a finite number of ideals. The first part of this reduction is given by \cite[Remark 6.18]{BGP}, which argues that $\G_{\f}$ and $\T_{\f}$ (and hence $\V_{\f}$) are preserved by polarization. Thus we always assume that $\f$ is a list of square-free monomials. Notably $\V_{\f}$ depends only on $I$ and hence is not sensitive to the choice of the $Q$-regular sequence over which our monomials are defined. Thus, the second step of our reduction is to show that, up to a choice of regular sequence, there are only finitely many square-free monomial ideals with $n$ generators.
Moving forward, we adopt the convention that
$\f$ is an ordered list of $n$ square-free monomials, none of which divide another. 
Hence $[n] := \{f_1,\ldots,f_n\} $ is the set of minimal generators of a square-free monomial ideal.
    
\begin{definition} \label{definition type of rgular element}
    Given a fixed list of square-free monomials $\f$, for a regular element $x$, the \emph{type} $\tau(x)$ is the subset of minimal generators that $x$ divides. That is, we identify $\tau(x) = \setbuild{f_i}{x\mid f_i}\subseteq [n]$. 
\end{definition} 

The image of this function appears as the construction of certain hypergraphs in \cite{lin_hypergraphs_2012, kimura_arithmetical_2009} where it is used to compute the arithmetic rank and regularity of square-free monomial ideals. 
Towards the second step of our reduction, we will consider ideals for which $\tau$ is injective.
Suppose for some $i\neq j$ that $\tau(x_i) = \tau(x_j)$, then $\f$ can be realized on a regular sequence of length $r-1$ by removing both $x_i$ and $x_j$ along with adding in their product $x_ix_j$. Repeating until $\tau$ is injective, we have that $r\leq2^n$ as well as a more insightful naming of our regular sequence: $x_{\sigma} = \tau^{-1}(\sigma)$. 
Note that, when appropriate, we will abbreviate a set $\{f_{i_1},\ldots,f_{i_s}\}\subseteq[n]$ with the concatenation $i_1\cdots i_s$, for example $x_{123}$ refers to $x_{\{f_1,f_2,f_3\}}$. Note further that $\f$ may be determined by the image of $\tau$ by setting $f_i = \prod_{i\in\sigma\subseteq im(\tau)}x_\sigma$.
    
\begin{example} \label{example f given by image of tau}
    Consider the ideal generated by $\f = (abcd, abce, cde)$ then $\tau(a) = \tau(b) = \{f_1,f_2\}$, $\tau(c) = \{f_1,f_2,f_3\}$, $\tau(d) = \{f_1,f_3\}$, and $\tau(e) = \{f_2,f_
    3\}$. Then $\f$ can be realized on the regular sequence $x_{12}=ab,\;x_{123}=c,\;x_{13}=d,\;x_{23} = e$. Hence, $\f = (x_{12}x_{123}x_{13},\; x_{12}x_{123}x_{23},\; x_{123}x_{13}x_{23})$.  
\end{example}    
    
    The possible image sets of $\tau$ are the elements of $2^{2^{[n]}\smallsetminus \emptyset}$, allowing us to reduce to a finite number of square-free monomial ideals for fixed $n$. Moving forward, we assume the regular sequence consists only of regular elements that appear in $\f$ and that no distinct regular elements have the same type. Moreover, we assume that each regular element is is decorated as $x_{\sigma_1},\ldots, x_{\sigma_r}$ where $\sigma_i=\tau(x_{\sigma_i})$. 
    When our regular elements are labeled as such, we refer to them as \emph{variables}. 
    We have lost no generality in these assumptions as any list of square-free monomials $\f$ can be realized in this manner by adjusting the regular sequence appropriately. 
    Let us denote the set of variables used by $\f$ to be $X_{\f}$ which is in correspondence with $ \im(\tau)$. 
    Note that $X_{\f}$ and $\f$ determine one another, hence the square-free monomial ideals we consider are in correspondence with elements of $2^{2^{[n]}\smallsetminus \emptyset}$.
\begin{example}\label{example set of variables}
The two ideals of Example \ref{example f given by image of tau} are determined by the variables \[X_{\f_1} =\{x_{4},x_{12},x_{13}, x_{23}, x_{14}\}\text{ and }X_{\f_2}=\{x_{2}, x_{3},x_{4}, x_{14},x_{123}\},\] so that \[\f_1 = (bce, bd,cd,ae)=(x_{12}x_{13}x_{14},\;x_{12}x_{23},x_{13}x_{23},\;x_{4}x_{14})\] and \[\f_2 = (de, ae, be, cd)=(x_{14}x_{123},\;x_2x_{123},\;x_3x_{123},\;x_4x_{14}).\]
\end{example}

    \begin{definition}
    We call the variable $x_{\sigma}$ \emph{present} if $x_\sigma\in X_{\f}$ and \emph{absent} if $x_{\sigma}\notin X_{\f}$.
    \end{definition}
    
    \begin{remark}
   We can express conditions like $f_3\mid f_{24}$  (where $f_{24} = \lcm(f_{2},f_{4}$)) from \cite[Theorem 6.16]{BGP} in terms of which variables are present and absent. In the two graphs given in that theorem, $f_3\mid f_{24}$ is equivalent to the absence of the variable $x_3$. For larger graphs, we will classify support varieties based on the absence of several variables. 
    \end{remark}
    
    Let us now introduce some combinatorial objects that will be helpful to stratify the square-free monomial ideals we will consider. 

    \begin{definition} \label{definition abstract simplicial complex}
    An \emph{abstract simplicial complex} with vertex set $[n]$ is a nonempty subset $K\subseteq 2^{[n]}$ with the property that $\sigma\in K$ and $\theta\subset \sigma$ implies $\theta\in K$. 
    Elements of $K$ are called \emph{faces} and maximal faces are called \emph{facets}. 
    The dimension of a nonempty face is one less than its cardinality. 
    A \emph{simple graph} is a simplicial complex with facets of dimension at most 1, where the 0-dimensional faces are called vertices and the 1-dimensional faces are called edges. 
    The \emph{1-skeleton} of $K$ is the simple graph given by removing the faces of $K$ with dimension more than 1. 
    The GCD graph $\G_{\f}$ can be identified with 1-skeleton of the smallest simplicial complex containing the labels of the variables $X_{\f}$. 
    \end{definition}

    \begin{definition} \label{definition clique}
    A \emph{complete graph} is a simple graph with every possible edge. 
    If $G$ is a simple graph and $W$ is a subset of its vertices, then the \emph{induced subgraph} of $G$ on $W$ is the graph with vertex set $W$ and edges of $G$ that use vertices in $W$. 
    We call an induced, complete subgraph a \emph{clique}.
    \end{definition}

    \begin{remark}\label{remark clique complex}
    Note that every induced subgraph of a clique is also a clique. Hence, the set of all cliques of a simple graph forms a simplicial complex, which we call the \emph{clique complex}. We denote the clique complex of $\G_{\f}$ by $K_{\f}$. 
     Note that if $i, j \in\sigma$, then the presence of $x_{\sigma}$ witnesses the edge $\{i,j\}$ in $\G_{\f}$. 
     In particular, $\G_{\f}$ is the 1-skeleton of $K_{\f}.$
    Moreover, we always have that the labels of variables in $X_{\f} $ must be elements of $ K_{\f}$. 
    It follows that the $\f$ satisfying $\G = \G_{\f}$ arise from selecting certain subsets of $2^K$ if $K$ is the clique complex of $\G$. 
    \end{remark}

We now prepare some lemmas that relate $X_{\f}$ with the combinatorics of $\G_{\f}$, in particular we identify features of a GCD graph that force certain variables to be present.

    \begin{lemma}\label{lemma leafs and edges}
        If $\{f_i,f_j\}$ is an edge in $\G_{\f}$ that is not contained in a triangle, then $x_{ij}$ is present. Moreover, if $f_\ell$ is a vertex with degree 1 then $x_\ell$ is present. 
    \end{lemma}
    \begin{proof}
        Any edge of $\G_{\f}$ must be witnessed by some $x_{\sigma}$ where $\sigma$ contains the edge. Any clique larger than an edge containing $\{f_i,f_j\}$ would contain a triangle with $\{f_i,f_j\}$ as an edge. Thus, $x_{ij}$ is the only possible variable that can witness this edge in $\G_{\f}.$ If $f_\ell$ has degree 1 with neighbor $f_m$, the only variables that can appear in $f_\ell$ are $x_\ell$ and $x_{\ell m}$. By the first part of this lemma, it must be that  $x_{\ell m}$ is present. However, it cannot be that $f_\ell \vert f_m$ and only the presence of $x_\ell$ can prevent this. Thus $f_\ell = x_\ell x_{\ell m }.$
    \end{proof}

    \begin{corollary} \label{corollary triangle free fiber}
        If $\G$ is a triangle-free graph with $\ell$ vertices of degree 1, then up to a choice of regular sequence there are exactly $2^{n-\ell}$ monomial ideals $\f$ realizing $\G=\G_{\f}$.
    \end{corollary}

\begin{lemma} \label{lemma variables present for degree 2 vertex}
    Suppose $N_{\f}(f_\ell) = \{f_i, f_j\}$, for distinct $f_i, f_j, f_\ell \in [n]$. Then at least one of the following holds: 
     \begin{itemize}
        \item $x_{i\ell}$ and $ x_{j\ell}$ are present; 
        \item $x_\ell$ and  $x_{ij\ell}$ are present.
    \end{itemize}
\end{lemma}
\begin{proof}
    By symmetry, it suffices to show that if $x_{i\ell}$ is not present, then $x_\ell$ and $x_{ij\ell}$ are present. By assumption, the only other variable that could witness the edge $\{f_i,f_{\ell}\}$ in $\G_{\f}$ is $x_{ij\ell}$ which thus must be present. On the other hand, the only other variable that ensures $f_\ell  \nmid f_j$ is $x_{\ell}$, which must then be present.
\end{proof}

Let us now consider the Taylor graph which is instrumental in our computation of support varieties.
This graph encodes a $2^n\times 2^n$ matrix that we use to compute the stable homology of an eventually 2-periodic complex. 
In turn this homology is used to compute the support variety of $R$. 
We give a brief description of this process.

Recall that $R$ has a regular presentation $(Q,\m,k)/I$ where $I$ is minimally generated by $n$ monomials $f_1, \dotsc, f_n$. 
Let $E = \Kos^Q(f_1,\dotsc, f_n)$, and let $f = \sum_{i=1}^n a_i f_i$ for some $a_i \in k$. Then we have that $(a_1,\dotsc,a_n) \in \V_{\f} \iff \pdim_{Q/(f)}(R) = \infty$ \cite[6.2.4]{Pollitz:2021}.
One can resolve $R$ over $Q/(f)$ using the Taylor resolution of $R$ over $Q$ along with the dg $E$-module structure on the Taylor resolution that is dictated by its own differential graded algebra structure. 
This construction can be found in \cite[Section 2]{BGP}, and is also explained in higher generality in \cite[Chapter 4]{Watson:2026}. Let us call this resolution $G$. The complex $G$ eventually becomes $2$-periodic and so the finiteness of the projective dimension of $R$ over $Q/(f)$ can be detected by the differential in the $2$-periodic part of $G$. The edges of the Taylor graph encode the $2$-periodic differentials of $G\otimes k$ and so the Taylor graph can be used to detect exactly which points are in the support variety. 

%units that occur in the Eisenbud-Shamash construction. 
For $\sigma\subseteq [n]$ set $f_\sigma = \lcm(\sigma)$ and for $f_i\notin \sigma$ set $\sign(\sigma,i) = (-1)^p$ for $p = \lvert\setbuild{f_j\in \sigma}{j<i}\rvert$.

\begin{definition} \label{definition_Taylor_graph}
   The \emph{Taylor graph} of $\f$ is the weighted directed graph $\T_{\f}$ defined as follows. 

\begin{itemize}
        \item Vertices: one vertex $v_\sigma$ for each subset $\sigma \subseteq [n]$.
        \item Edges: for each $\sigma \subseteq [n]$ and $f_i \in [n]\smallsetminus \sigma$: 
\begin{itemize}
    \item if $f_i\;\vert\;f_\sigma$, there is a \emph{differential edge} $v_{\sigma\cup \{f_i\}} \xrightarrow[]{\dd_{\sigma,i}}v_{\sigma}$  with weight $\dd_{\sigma,i} = \sign(\sigma,i)$,
    \item if $f_i \notin N_{\f}(\sigma)$, there is a \emph{homotopy edge} $v_{\sigma} \xrightarrow[]{\h_{\sigma,i}}v_{\sigma\cup \{f_i\}}$  with weight $\h_{\sigma,i} = \sign(\sigma,i)\cdot \chi_i.$
    \end{itemize}
    \end{itemize}
\end{definition}

Note that homotopy edges depend only on $\G_{\f}$.
If the symbols $\dd_{\sigma,i}, \h_{\sigma,i}$ correspond to edges that do not exist in $\T_{\f}$ we take their value/weight to be 0. When referring to an edge that is unknown to be a homotopy edge or differential edge, we will use the symbol $\e_i$ to denote an edge that either adds or removes $f_i$.
Similarly, if $\sigma$ is unknown or obvious we will use the symbols $\dd_i$ and $\h_i$ accordingly.
We use $\e_i^s$ to refer to the source vertex of $\e_i$, and we denote its target vertex by $\e_i^t$. 

For any point $a=(a_1,\ldots,a_n)\in\A^n_{k}$, we consider a linear transformation $\T_{\f}(a): k^{2^n}\rightarrow k^{2^n}$ obtained by picking the vertices of $\T_{\f}$ as a basis, substituting $\chi_i=a_i$, and  sending 
\begin{equation} \label{differential in linear transformation}
  v_\sigma\mapsto \sum_{f_i \notin \sigma}\h_{\sigma,i}v_{\sigma\cup\{f_i\}} + \sum_{f_i \in \sigma}\dd_{\sigma\smallsetminus \{f_i\},i}v_{\sigma\smallsetminus \{f_i\}}.  
\end{equation}

This choice of basis realizes $\T_{\f}(a)$ as a sparse, $2^n\times2^n$ matrix with entries in $k$. We choose the convention that the columns of $\T_{\f}(a)$ correspond to sources (and rows to targets) of edges in $\T_{\f}$. We have that \cite[Remark 6.4]{BGP} implies that $\T_{f}(a)$ squares to 0,  forcing $\rank\T_{\f}(a) \leq2^{n-1}$, and moreover that  $a\in \V_{\f}$ exactly when this inequality is strict. In particular, $a\notin \V_{\f}$ is equivalent to $\ker(\T_{\f}(a)) = \im(\T_{\f}(a))$. We call the elements of $\ker(\T_{\f}(a))$ \emph{cycles} and elements of $\im(\T_{\f}(a))$ \emph{boundaries}. In the same vein, $\T_{\f}$ encodes an $\S$-module endomorphism of $\S^{2^n}$ that can be represented by the same matrix but without substituting values for the $\chi_i$; by \cite[2.5]{BGP} we are justified in the following definition for the ideal carving out $\V_{\f}$.
\begin{definition}\label{defenition support Ideal}
The \emph{support ideal} of $\f$ is $\I_{\f} = \sqrt{I_{2^{n-1}}(\T_{\f})}\subseteq \S$, the radical of the ideal of $2^{n-1}$-minors of the matrix that $\T_{\f}$ encodes.
\end{definition}

\begin{example}\label{example taylor graph definition}
    Let us consider the two ideals of Example \ref{example set of variables} given by 
    \[\f_1 =(x_{12}x_{13}x_{14},\;x_{12}x_{23},x_{13}x_{23},\;x_{4}x_{14})\text{ and }\f_2 = (x_{14}x_{123},\;x_2x_{123},\;x_3x_{123},\;x_4x_{14}).\]
\end{example}
% https://q.uiver.app/#q=WzAsMzQsWzEsNSwiXFxlbXB0eXNldCJdLFsxLDYsIjEiXSxbMiw1LCIyIl0sWzAsNSwiMyJdLFsxLDQsIjQiXSxbMCwwLCIxMiJdLFsyLDAsIjEzIl0sWzIsNiwiMTQiXSxbMSwxLCIyMyJdLFsyLDQsIjI0Il0sWzAsNCwiMzQiXSxbMSwwLCIxMjMiXSxbMiwzLCIxMjQiXSxbMCwzLCIxMzQiXSxbMSwyLCIyMzQiXSxbMSwzLCIxMjM0Il0sWzEsNywiXFxUX3tcXGZfMX0iXSxbNSw3LCJcXFRfe1xcZl8yfSJdLFs1LDUsIlxcZW1wdHlzZXQiXSxbNSw2LCIxIl0sWzYsNSwiMiJdLFs0LDUsIjMiXSxbNSw0LCI0Il0sWzQsMCwiMTIiXSxbNiwwLCIxMyJdLFs2LDYsIjE0Il0sWzUsMSwiMjMiXSxbNiw0LCIyNCJdLFs0LDQsIjM0Il0sWzUsMCwiMTIzIl0sWzYsMywiMTI0Il0sWzQsMywiMTM0Il0sWzUsMiwiMjM0Il0sWzUsMywiMTIzNCJdLFswLDEsIlxcaF8xIl0sWzAsMiwiXFxoXzIiXSxbMCwzLCJcXGhfMyIsMl0sWzAsNCwiXFxoXzQiLDJdLFs0LDksIlxcaF8yIl0sWzMsMTAsIlxcaF80Il0sWzQsMTAsIlxcaF8zIiwyXSxbMiw5LCJcXGhfNCIsMl0sWzE1LDEyLCJcXGRkXzMiLDJdLFsxNSwxMywiXFxkZF8yIl0sWzE1LDE0LCJcXGRkXzEiLDJdLFsxMSw1LCJcXGRkXzMiLDJdLFsxMSw2LCJcXGRkXzIiXSxbOCwxNCwiXFxoXzQiXSxbMTgsMTksIlxcaF8xIl0sWzE4LDIwLCJcXGhfMiJdLFsxOCwyMSwiXFxoXzMiLDJdLFsxOCwyMiwiXFxoXzQiLDJdLFsyMiwyNywiXFxoXzIiXSxbMjEsMjgsIlxcaF80Il0sWzIyLDI4LCJcXGhfMyIsMl0sWzIwLDI3LCJcXGhfNCIsMl0sWzMzLDMyLCJcXGRkXzEiLDJdLFsyNiwzMiwiXFxoXzQiXSxbMzEsMjgsIlxcZGRfMSIsMl0sWzMwLDI3LCJcXGRkXzEiXV0=
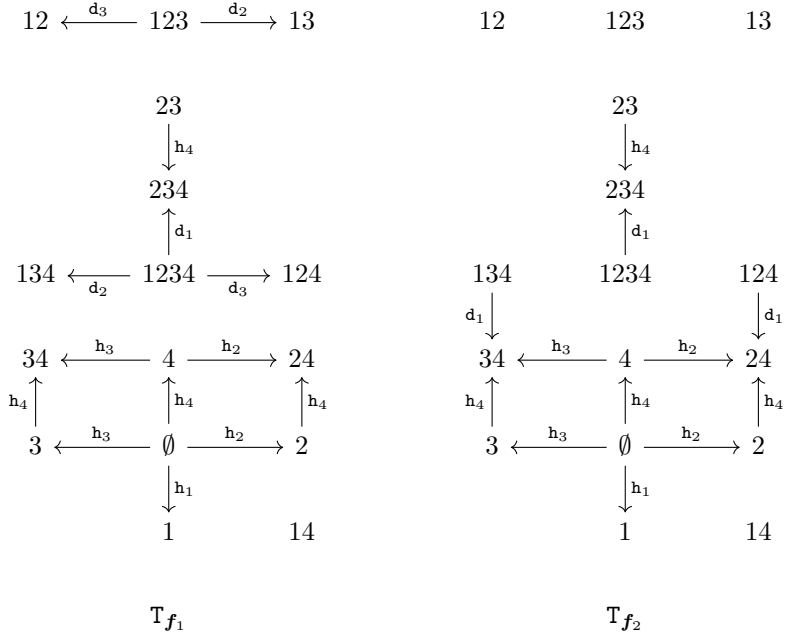
\[\begin{tikzcd}
	12 & 123 & 13 && 12 & 123 & 13 \\
	& 23 &&&& 23 \\
	& 234 &&&& 234 \\
	134 & 1234 & 124 && 134 & 1234 & 124 \\
	34 & 4 & 24 && 34 & 4 & 24 \\
	3 & \emptyset & 2 && 3 & \emptyset & 2 \\
	& 1 & 14 &&& 1 & 14 \\
	& {\T_{\f_1}} &&&& {\T_{\f_2}}
	\arrow["{\dd_3}"', from=1-2, to=1-1]
	\arrow["{\dd_2}", from=1-2, to=1-3]
	\arrow["{\h_4}", from=2-2, to=3-2]
	\arrow["{\h_4}", from=2-6, to=3-6]
	\arrow["{\dd_1}"', from=4-2, to=3-2]
	\arrow["{\dd_2}", from=4-2, to=4-1]
	\arrow["{\dd_3}"', from=4-2, to=4-3]
	\arrow["{\dd_1}"', from=4-5, to=5-5]
	\arrow["{\dd_1}"', from=4-6, to=3-6]
	\arrow["{\dd_1}", from=4-7, to=5-7]
	\arrow["{\h_3}"', from=5-2, to=5-1]
	\arrow["{\h_2}", from=5-2, to=5-3]
	\arrow["{\h_3}"', from=5-6, to=5-5]
	\arrow["{\h_2}", from=5-6, to=5-7]
	\arrow["{\h_4}", from=6-1, to=5-1]
	\arrow["{\h_4}"', from=6-2, to=5-2]
	\arrow["{\h_3}"', from=6-2, to=6-1]
	\arrow["{\h_2}", from=6-2, to=6-3]
	\arrow["{\h_1}", from=6-2, to=7-2]
	\arrow["{\h_4}"', from=6-3, to=5-3]
	\arrow["{\h_4}", from=6-5, to=5-5]
	\arrow["{\h_4}"', from=6-6, to=5-6]
	\arrow["{\h_3}"', from=6-6, to=6-5]
	\arrow["{\h_2}", from=6-6, to=6-7]
	\arrow["{\h_1}", from=6-6, to=7-6]
	\arrow["{\h_4}"', from=6-7, to=5-7]
\end{tikzcd}\]

%%%%%%%%%%%%%%%%%%%%%%%%%%%%%%%%%%%%%%%%%%%%%%

\begin{proposition} \label{proposition taylor edges down/upward closed}
    If $\dd_{\sigma,i}$ is an edge of  $\T_{\f}$, then every $\sigma'\supseteq \sigma$ with $f_i \notin \sigma'$ has an edge $\dd_{\sigma',i}$ in  $\T_{\f}$. Similarly, if $\h_{\theta,j}$ is an edge of $\T_{\f}$, then every $\theta'\subseteq\theta$ has an edge $\h_{\theta',j}$ in $\T_{\f}$. 
\end{proposition}

\begin{proof}
    By assumption we have that $f_\sigma\vert f_{\sigma'}$, hence $f_i\vert f_{\sigma'}$ and thus the desired differential edge exists. Similarly, we have $N_{\f}(\theta')\subseteq N_{\f}(\theta)$ hence $f_j$ is not in the (possibly smaller) neighborhood, thus the desired homotopy edge exists. 
\end{proof}

Shortly we will want to consider certain edges of $\T_{\f}$ en masse, so we have the following notation. 
Consider an edge $\e_{\sigma,i}$ in $\T_{\f}$ and a subset $\Lambda\subseteq 2^{[n]\smallsetminus(\sigma\cup \{f_i\})}$.
We set the notation \[\e_{\sigma,i}\otimes \Lambda = \setbuild{\e_{\sigma\cup \lambda,i}}{\lambda\in \Lambda}.\] 
When $\Lambda$ is a power set, we call the resulting set of edges a \emph{hypercube}. 
Let $\min(\Lambda)$ (respectively, $\max(\Lambda)$) denote the set of minimal (respectively, maximal) elements of $\Lambda$ under containment.
 
 \begin{lemma}\label{lemma existence of hypercubes}
If $\T_{\f}$ is a Taylor graph and $\h$ is a homotopy edge, then every edge in the set $\h\otimes \Lambda$ exists in $\T_{\f}$ if and only if every edge of $\h\otimes \max(\Lambda)$ exists in $\T_{\f}$. 
Similarly, if $\dd$ is a differential edge, then every edge in the set $\dd\otimes \Lambda$ exists in $\T_{\f}$ if and only if every edge of $\dd\otimes \min(\Lambda)$ exists in $\T_{\f}$.
\end{lemma}

\begin{proof}
    Apply Proposition \ref{proposition taylor edges down/upward closed}.
\end{proof}

This lemma is especially useful when $X$ is a hypercube, as the existence of one edge implies the existence of $2^{\lvert \Lambda\rvert}$ edges. In Example \ref{running example taylor graph}, the \color{mygreen} green edges \color{black}  (alternatively, the edges labeled $\dd_3$) comprise the hypercube $\dd_{24,3}\otimes 2^{\{1,5\}}$.

\begin{definition} \label{definition definition homotopically/differentially isolated} If a vertex $v_\sigma$ is not used by any homotopy edge, we say $v_\sigma$ is \emph{homotopically isolated}. Similarly, if no differential edge uses the vertex $v_{\sigma}$, then $v_{\sigma}$ is \emph{differentially isolated}.
\end{definition}

\begin{proposition} \label{proposition homotopically/differentially isolated}
    A vertex $v_{\sigma}$ is homotopically isolated if and only if $N_{\f}(\sigma) = [n]$. 
    If for all $f_i\in [n]$ we have $\lvert N_{\f}(f_i)\cap(\sigma\smallsetminus \{f_i\})\rvert < 2$, then $v_{\sigma}$ is differentially isolated. In particular, no differential edge exists between a two element set and a one element set.
\end{proposition}
\begin{proof}
    We first note that no homotopy edge leaves $v_\sigma$  when every $f_j\notin\sigma$ is a neighbor of some generator in $\sigma$, i.e., $[n] = \sigma\cup N_{\f}(\sigma)$. On the other hand, no homotopy edge has target $v_\sigma$ exactly when $\sigma \subseteq N_{\f}(\sigma)$. The second statement of the proposition follows from \cite[Remark 6.7]{BGP}, that if a differential edge $\dd_{\sigma,i}$ exists then $\lvert N_{\f}(f_i)\cap \sigma\rvert \geq 2$.
\end{proof}

\section{Computing Support Varieties}\label{section computing support varieties}

In the case that the support variety is $V(\chi_1\chi_5)$ for 5 generated monomial ideals, the authors of \cite{BGP} wrote out and argued about certain features of the images and kernels of two $16\times 16$ matrices. 
This exact strategy is not well suited to scale; in the case of 6 generators there would be more than just two $32 \times 32$ matrices to analyze.
The purpose of this section is to develop a combinatorial framework for computing support varieties from Taylor graphs. 
As a demonstration of this framework, we use the ideals realizing $V(\chi_1\chi_5)$ in \cite[Theorem 6.16]{BGP} as an example. 
Our argument proceeds via a two sided containment. 
Certain features of the Taylor graph give $V(\chi_1)\subseteq \V_{\f}$ and $V(\chi_5)\subseteq \V_{\f}$ and hence $V(\chi_1\chi_5)\subseteq \V_{\f}$.
For the other containment, we compute a minor to find that $\chi_1\chi_5\in \I_{\f}$ so that $\V_{\f}\subseteq V(\chi_1\chi_5)$.
We first provide this argument for just one ideal, then later show that the argument applies to each 5 generated ideal realizing $\V_{\f} = V(\chi_1\chi_5)$.

\begin{example}{\cite[Theorem 6.16]{BGP}}\label{running example taylor graph}

If $\f$ has that $\G_{\f}$ is either A or $P_5$ depicted below and $x_3\notin X_{\f}$ then $\V_{\f} = V(\chi_1\chi_5).$
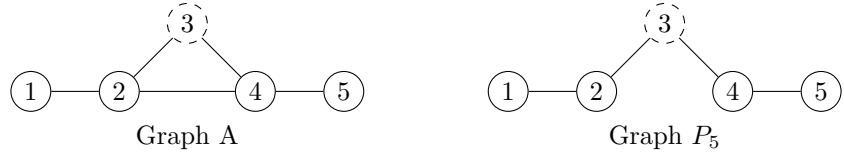
\begin{center}
\begin{minipage}{0.4\textwidth}
\centering
\begin{tikzpicture}[scale=.9, every node/.style={circle, inner sep=0pt, minimum size=15pt, draw}]

  % Nodes
  \node[dashed] (v1) at (0,1) {$3$};
  \node (v2) at (-1,0) {$2$};
  \node (v3) at (1,0) {$4$};
  \node (v4) at (-2.3,0) {$1$};
  \node (v5) at (2.3,0) {$5$};

  % Edges
  \draw (v1) -- (v2);
  \draw (v1) -- (v3);
  %\draw (v1) -- (v5);
  \draw (v2) -- (v3);
  \draw (v2) -- (v4);
  %\draw (v2) -- (v5);
  \draw (v3) -- (v5);
\end{tikzpicture}\\
Graph A
\end{minipage}
\begin{minipage}{0.4\textwidth}
\centering
\begin{tikzpicture}[scale=.9, every node/.style={circle, inner sep=0pt, minimum size=15pt, draw}]

  % Nodes
  \node[dashed] (v1) at (0,1) {$3$};
  \node (v2) at (-1,0) {$2$};
  \node (v3) at (1,0) {$4$};
  \node (v4) at (-2.3,0) {$1$};
  \node (v5) at (2.3,0) {$5$};

  % Edges
  \draw (v1) -- (v2);
  \draw (v1) -- (v3);
  %\draw (v1) -- (v5);
  %\draw (v2) -- (v3);
  \draw (v2) -- (v4);
  %\draw (v2) -- (v5);
  \draw (v3) -- (v5);
\end{tikzpicture}\\
Graph $P_5$
\end{minipage}
\\
\vspace{10pt} 
\end{center}  

For a running example, we consider $\f = (x_1x_{12},\, x_{12}x_{23},\, x_{23}x_{34},\, x_{34}x_{45},\,x_{45}x_5)$, so in particular $X_{\f} = \{x_1,x_5,x_{12},x_{23},x_{34},x_{45}\}$. 
Below is the Taylor graph $\T_{\f}$. 
The colors and dashed edges will be used later to highlight various features that we develop. 
For now, they can be ignored and every edge depicted is indeed an edge regardless of color or dashed-ness.

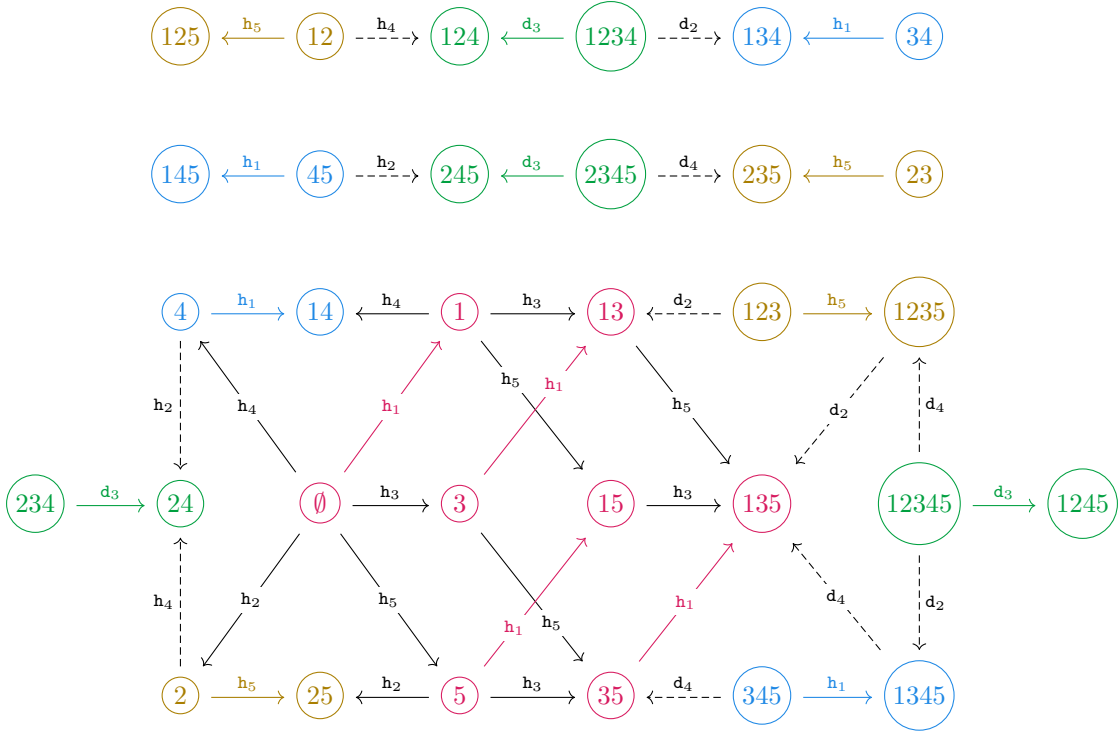
\[\begin{tikzcd}
	& \textcolor{myyellow}{{\circled{125}}} & \textcolor{myyellow}{{\circled{12}}} & \textcolor{mygreen}{{\circled{124}}} & \textcolor{mygreen}{{\circled{1234}}} & \textcolor{myblue}{{\circled{134}}} & \textcolor{myblue}{{\circled{34}}} & \\
	& \textcolor{myblue}{{\circled{145}}} & \textcolor{myblue}{{\circled{45}}} & \textcolor{mygreen}{{\circled{245}}} & \textcolor{mygreen}{{\circled{2345}}} & \textcolor{myyellow}{{\circled{235}}} & \textcolor{myyellow}{{\circled{23}}} \\
	& \textcolor{myblue}{{\circled{4}}} & \textcolor{myblue}{{\circled{14}}} & \textcolor{mypink}{{\circled{1}}} & \textcolor{mypink}{{\circled{13}}} & \textcolor{myyellow}{{\circled{123}}} & \textcolor{myyellow}{{\circled{1235}}} \\
	\\
	\textcolor{mygreen}{{\circled{234}}} & \textcolor{mygreen}{{\circled{24}}} & \textcolor{mypink}{{\circled{$\emptyset$}}} & \textcolor{mypink}{{\circled{3}}} & \textcolor{mypink}{{\circled{15}}} & \textcolor{mypink}{{\circled{135}}} & \textcolor{mygreen}{{\circled{12345}}} & \textcolor{mygreen}{{\circled{1245}}} \\
	\\
	& \textcolor{myyellow}{{\circled{2}}} & \textcolor{myyellow}{{\circled{25}}} & \textcolor{mypink}{{\circled{5}}} & \textcolor{mypink}{{\circled{35}}} & \textcolor{myblue}{{\circled{345}}} & \textcolor{myblue}{{\circled{1345}}}
	\arrow["{\h_5}"', color={myyellow}, from=1-3, to=1-2]
	\arrow["{\h_4}", dashed, from=1-3, to=1-4]
	\arrow["{\dd_3}"', color={mygreen}, from=1-5, to=1-4]
	\arrow["{\dd_2}", dashed, from=1-5, to=1-6]
	\arrow["{\h_1}"', color={myblue}, from=1-7, to=1-6]
	\arrow["{\h_1}"', color={myblue}, from=2-3, to=2-2]
	\arrow["{\h_2}", dashed, from=2-3, to=2-4]
	\arrow["{\dd_3}"', color={mygreen}, from=2-5, to=2-4]
	\arrow["{\dd_4}", dashed, from=2-5, to=2-6]
	\arrow["{\h_5}"', color={myyellow}, from=2-7, to=2-6]
	\arrow["{\h_1}", color={myblue}, from=3-2, to=3-3]
	\arrow["{\h_2}"', dashed, from=3-2, to=5-2]
	\arrow["{\h_4}"', from=3-4, to=3-3]
	\arrow["{\h_3}", from=3-4, to=3-5]
	\arrow["{\h_5}"{description, pos=0.3}, from=3-4, to=5-5]
	\arrow["{\h_5}"{description}, from=3-5, to=5-6]
	\arrow["{\dd_2}"', dashed, from=3-6, to=3-5]
	\arrow["{\h_5}", color={myyellow}, from=3-6, to=3-7]
	\arrow["{\dd_2}"{description}, dashed, from=3-7, to=5-6]
	\arrow["{\dd_3}", color={mygreen}, from=5-1, to=5-2]
	\arrow["{\h_4}"{description}, from=5-3, to=3-2]
	\arrow["{\h_1}"{description}, color={mypink}, from=5-3, to=3-4]
	\arrow["{\h_3}", from=5-3, to=5-4]
	\arrow["{\h_2}"{description}, from=5-3, to=7-2]
	\arrow["{\h_5}"{description}, from=5-3, to=7-4]
	\arrow["{\h_1}"{description, pos=0.7}, color={mypink}, from=5-4, to=3-5]
	\arrow["{\h_5}"{description, pos=0.7}, from=5-4, to=7-5]
	\arrow["{\h_3}", from=5-5, to=5-6]
	\arrow["{\dd_4}"', dashed, from=5-7, to=3-7]
	\arrow["{\dd_3}", color={mygreen}, from=5-7, to=5-8]
	\arrow["{\dd_2}", dashed, from=5-7, to=7-7]
	\arrow["{\h_4}", dashed, from=7-2, to=5-2]
	\arrow["{\h_5}", color={myyellow}, from=7-2, to=7-3]
	\arrow["{\h_1}"{description, pos=0.3}, color={mypink}, from=7-4, to=5-5]
	\arrow["{\h_2}"', from=7-4, to=7-3]
	\arrow["{\h_3}", from=7-4, to=7-5]
	\arrow["{\h_1}"{description}, color={mypink}, from=7-5, to=5-6]
	\arrow["{\dd_4}"', dashed, from=7-6, to=7-5]
	\arrow["{\h_1}", color={myblue}, from=7-6, to=7-7]
	\arrow["{\dd_4}"{description}, dashed, from=7-7, to=5-6]
\end{tikzcd}\]

\end{example}

We now develop some tools that will be fruitful in the cases that $\V_{\f}$ is a union of proper coordinate subspaces -- that is, when $\I_{\f}$ is a square-free monomial ideal. 
The first tool will allow us to identify certain coordinate subspaces that are subvarieties of $\V_{\f}$. 
Recall that in a directed graph, vertices with only outgoing edges are \textit{sources} and those with only incoming edges are \textit{sinks}.

\begin{definition}\label{homotopy sink/source}
Let $v_\sigma$ be a differentially isolated vertex such that any homotopy edge using $v_{\sigma}$ is one of $\h_{i_1},\ldots, \h_{i_s}$.
We say that $v_\sigma$ is a \emph{homotopy sink for $\chi_{i_1}, \ldots, \chi_{i_s}$} if $[n] = \sigma\cup N_{\f}(\sigma)$. 
Similarly, $v_\sigma$ is a \emph{homotopy source for $\chi_{i_1}, \ldots, \chi_{i_s}$} if $ \sigma \subseteq  N_{\f}(\sigma)$.
\end{definition}
Indeed, by the proof of Proposition \ref{proposition homotopically/differentially isolated} the conditions $[n] = \sigma\cup N_{\f}(\sigma)$ and $ \sigma \subseteq  N_{\f}(\sigma)$ force $v_{\sigma}$ to be a sink/source respectively. 

\begin{proposition}\label{proposition homotopy soure/sink gives subvariety}
    If $v$ is a homotopy source/sink for $\chi_{i_1}, \ldots, \chi_{i_s}$ in $\T_{\f}$, then $V(\chi_{i_1}, \ldots, \chi_{i_s})\subseteq \V_{\f}$.
\end{proposition}
\begin{proof}
    Take $a\in V(\chi_{i_1}, \ldots, \chi_{i_s})$. If $v$ is a source, then it is not in the image of $\T_{\f}(a)$ and every coefficient appearing in $\T_{\f}(a)(v)$ is 0, meaning that $v$ is a cycle that is not a boundary. A symmetric argument holds when $v$ is a sink.
\end{proof}

\begin{corollary}\label{corollary isolated vertex}
    \cite[Lemma 6.9]{BGP} If $v$ is an isolated vertex in $\T_{\f}$ then it is a homotopy sink for $\emptyset$ and hence $V(\emptyset) = \A^n_k \subseteq \V_{\f}$.
\end{corollary}

\begin{example}\label{running example subvarieties}
    Let us find all of the homotopy sources in sinks and Example \ref{running example taylor graph}. 
    The empty set is homotopy source for $\chi_1,\chi_2,\chi_3,\chi_4,\chi_5$. Moreover, for $\chi_1$, 145 is a homotopy sink and 34 is a homotopy source. For $\chi_5$, 125 is a homotopy sink and 23 is a homotopy source. So, by Proposition \ref{proposition homotopy soure/sink gives subvariety}, the hyperplanes $\chi_1 = 0$ and $\chi_5=0$ are subvarieties of $\V_{\f}$ for our running example.

\end{example}

\begin{remark} 
 It is not generally true that every point of $\V_{\f}$ is witnessed by a homotopy source or sink, otherwise all support ideals would be square-free monomial ideals. However, it turns out that every primary component of the square-free monomial support ideals that we consider is indeed witnessed by some homotopy source or sink. We now outline an approach to show that, in these cases, the support variety is no larger than the union of the subvarieties witnessed by homotopy sources and sinks. 
\end{remark}

\begin{definition} \label{perfect matching}
    Given $\T_{\f}$ and some $\sigma\subseteq [n]$, a $\sigma$\emph{-perfect matching} of $\T_{\f}$ is a set $M$ of $2^{n-1}$ edges in $\T_{\f}$ such that no two edges in $M$ share a vertex and every homotopy edge $\h_{i}\in M$ has $f_i\in \sigma$.
\end{definition}

\begin{example}\label{running example perfect matching}
    The four colors of Example \ref{running example taylor graph} give the $\{f_1,f_5\}$-perfect matching: \[M = 
    \color{mypink} \h_{\emptyset,1}\otimes 2^{\{3,5\}} \color{black} \cup 
    \color{myblue} \h_{\{4\},1}\otimes 2^{\{3,5\}} \color{black}\cup 
    \color{myyellow} \h_{\{2\},5}\otimes 2^{\{1,3\}}
    \color{black}\cup 
    \color{mygreen} \dd_{\{24\},3}\otimes 2^{\{1,5\}}.\]
\end{example}

\begin{remark}\label{remark matchings give minors}
We note that the data of a $\sigma$-perfect matching $M$ induces a partition of the vertices into sources and targets. 
We denote the $2^{n-1}\times 2^{n-1}$ submatrix of $\T_{\f}$ given by picking the sources' columns and targets' rows by $\T_{\f}^M$. 
Given some ordering of the columns of  $\T_{\f}^M$, the data of $M$ induces an ordering on the rows so that every entry of the diagonal is non-zero and the non-unit entries are $\pm\chi_i$ with $i\in \sigma$.
\end{remark}

\begin{example}\label{running example submatrix from matching}
    The matching $M$ of Example \ref{running example perfect matching} gives $\T_{\f}^M$ below.
   \renewcommand{\arraystretch}{1.25}
 \[\begin{blockarray}{ccccc|cccc|cccc|cccc}
    &\matlabel{\emptyset} & \matlabel{3}  & \matlabel{5}  & \matlabel{35} & \matlabel{4} & \matlabel{34} & \matlabel{45} &\matlabel{345} & \matlabel{2}  & \matlabel{12} & \matlabel{23} & \matlabel{123} & \matlabel{234}  &\matlabel{1234} & \matlabel{2345} &  \matlabel{12345}\\ 
    \begin{block}{c(cccc|cccc|cccc|cccc)}
    \matlabel{1}    &\h_1& \raisebox{13pt}{~}  &   &   &   &   &   &   &   &   &   &    &    &    &    &     \\
    \matlabel{13}   &   &\h_1&   &   &   &   &   &   &   &   &   &\dd_2&    &    &    &     \\
    \matlabel{15}   &   &   &\h_1&   &   &   &   &   &   &   &   &    &    &    &    &     \\
    \matlabel{135}  &   &   &   &\h_1&   &   &   &   &   &   &   &    &    &    &    &     \\
 \BAhline
    \matlabel{14}   &  &   &   &    &\h_1&   &   &   &   &   &   &    &    &    &    &     \cr
    \matlabel{134 } &  &   &   &    &   &\h_1&   &   &   &   &   &    &    &\dd_2&    &     \cr
    \matlabel{145}  &  &   &   &    &   &   &\h_1&   &   &   &   &    &    &    &    &     \cr
    \matlabel{1345} &  &   &   &    &   &   &   &\h_1&   &   &   &    &    &    &    &\dd_2 \cr
    \BAhline
   % \cr\hline\cr
    %
    \matlabel{25}   &  &   &\h_2&    &   &   &   &   &\h_5&   &   &    &    &    &    &     \cr
    \matlabel{125}  &  &   &   &    &   &   &   &   &   &\h_5&   &    &    &    &    &     \cr
    \matlabel{235}  &  &   &   &    &   &   &   &   &   &   &\h_5&    &    &    &\dd_4&     \cr
    \matlabel{1235} &  &   &   &    &   &   &   &   &   &   &   &\h_5 &    &    &    &\dd_4 \cr
     \BAhline
    %\cr\hline\cr
    %
    \matlabel{24}   &  &   &   &    &\h_2&   &   &   &\h_4&   &   &    &\dd_3&    &    &     \cr
    \matlabel{124}  &  &   &   &    &   &   &   &   &   &\h_4&   &    &    &\dd_3&    &     \cr
    \matlabel{245}  &  &   &   &    &   &   &\h_2&   &   &   &   &    &    &    &\dd_3&     \cr
    \matlabel{1245} &  &   &   &    &   &   &   &   &   &   &   &    &    &    &    &\dd_3 \raisebox{-10pt}{} \cr
    \end{block}
    \end{blockarray}
    \]
\end{example}

If there exists an ordering of the columns so that a submatrix $\T_{\f}^M$ is (upper) triangular, then $\chi^\sigma := \prod_{i\in\sigma}\chi_i\in \I_{\f}$. 
When there is such an ordering we say $M$ is \emph{triangular}. 
The existence of such an ordering can be detected by the nonexistence of directed cycles in another directed graph $\Aux_{\f}^M:=\Aux(\T_{\f}^M)$ which we call the auxiliary graph.

\begin{definition}\label{definition auxiliary graph}
    Given an $m\times m$ matrix $A = \{a_{ij}\}$ with $a_{ii}\neq 0$, the \emph{auxiliary graph} of $A$ is the directed graph $\Aux(A)$ with vertex set $C = \{c_1,\ldots, c_m\}$ and edges $c_i\rightarrow c_j$ when $i\neq j$ and $a_{ji}\neq0$. A \emph{cycle decomposition} of $\Aux(A)$ is a permutation $\rho\in \Sym(C)$ such that for every $c\in C$ either $c = \rho(c)$ or $c\rightarrow\rho(c)$ is an edge of $\Aux(A)$. We denote the set of all cycle decompositions of $\Aux(A)$ by $\mathcal{C}(A)$.
\end{definition}

\begin{proposition} \label{proposition det by cycle decompositions}
    In the setting of Definition \ref{definition auxiliary graph}, $\det(A) = \sum_{\rho\in \mathcal{C}(A)}\sign(\rho)\prod_ia_{i\rho(i)}$. Moreover, if $M$ is a $\sigma$-perfect matching of $\T_{\f}$ such that $ \Aux_{\f}^M$ has no directed cycles, then $M$ is triangular and $\chi^\sigma\in \I_{\f}$. 
\end{proposition}
\begin{proof}
The edges of $\Aux(A)$ encode when column $j$ can be replaced by column $i$ without putting $0$ in the diagonal. Thus cycle decompositions of $\Aux(A)$ correspond to the permutations $\rho\in\Sym(C)$ such that the product $\prod_i a_{i\rho(i)}$ is nonzero. The first statement follows by computing $\det(A)$ with the Leibniz formula, excluding the summands that are 0. The condition that $\Aux_{\f}^M$ has no directed cycles is equivalent to the statement $\mathcal{C}(\T_{\f}^M) = \{\id\}$. 
Up to a sign we have  $\det(\T_{\f}^M) = \prod_{i\in \sigma}\chi_i^{p_i}$. Since $\I_{\f}$ is the radical of the ideal of $2^{n-1}$-minors of $\T_{\f}$ it follows that $\chi^\sigma\in\I_{\f}$. To see that $M$ is triangular note that $\Aux_{\f}^M$ having no cycles gives a partial order on $C$ where $c_i\preceq c_j$ if there is a directed path from $c_i$ to $c_j$ in $\Aux_{\f}^M$. Any total order on $C$ that respects this partial order gives rise to an upper triangular representation of $\T_{\f}^M$ where the diagonal entries are given by the the edges in $M.$
\end{proof}

\begin{remark}\label{remark taylor weights give auxilliary weights}
    The edges of auxiliary graphs of the form $\Aux_{\f}^M$ inherit weights from $\T_{\f}$.  
    In particular, we formally identify the vertices of $\Aux_{\f}^M$ with the sources of the edges in $M$.
    Hence an edge $\e_i^s\rightarrow \e_{j}^s$ exists in $\Aux_{\f}^M$ exactly when there is an edge $\e_\ell = \e_i^s\rightarrow \e_{j}^t$ in $\T_{\f}$. 
    In particular, then there is a ``walk" from $\e_i^s$ to $\e_{j}^s$ in $\T_{\f}$ given by $\e_i^s\rightarrow \e_{j}^t\leftarrow \e_{j}^s$.
    Hence we assign the weight $\e_j^{-1}\e_\ell$ to the edge $\e_i^s\rightarrow \e_{j}^s$ in $\Aux_{\f}^M$. 
\end{remark}

\begin{example}\label{running example}
The 10 non-zero off-diagonal entries of $\T_{\f}^M$ in Example \ref{running example submatrix from matching} yield 10 edges in $\Aux_{\f}^M$ depicted below. 
Since this graph has no directed cycles, $M$ is triangular and thus $\chi_1\chi_5\in \I_{\f}$. 
Combining with the fact that $V(\chi_1\chi_5)\subseteq \V_{\f}$ from Example \ref{running example subvarieties} we have that $\V_{f} = V(\chi_1\chi_5)$.
% https://q.uiver.app/#q=WzAsMTYsWzYsMSwiXFxjaXJjbGVkezB9IixbMCw2MCw2MCwxXV0sWzYsMywiXFxjaXJjbGVkezN9IixbMCw2MCw2MCwxXV0sWzYsMCwiXFxjaXJjbGVkezV9IixbMCw2MCw2MCwxXV0sWzYsMiwiXFxjaXJjbGVkezM1fSIsWzAsNjAsNjAsMV1dLFswLDAsIlxcY2lyY2xlZHs0fSIsWzI0MCw2MCw2MCwxXV0sWzAsMSwiXFxjaXJjbGVkezM0fSIsWzI0MCw2MCw2MCwxXV0sWzAsMiwiXFxjaXJjbGVkezQ1fSIsWzI0MCw2MCw2MCwxXV0sWzAsMywiXFxjaXJjbGVkezM0NX0iLFsyNDAsNjAsNjAsMV1dLFs0LDAsIlxcY2lyY2xlZHsyfSIsWzYwLDYwLDYwLDFdXSxbNCwxLCJcXGNpcmNsZWR7MTJ9IixbNjAsNjAsNjAsMV1dLFs0LDIsIlxcY2lyY2xlZHsyM30iLFs2MCw2MCw2MCwxXV0sWzQsMywiXFxjaXJjbGVkezEyM30iLFs2MCw2MCw2MCwxXV0sWzIsMCwiXFxjaXJjbGVkezIzNH0iLFsxMjAsNjAsMzQsMV1dLFsyLDEsIlxcY2lyY2xlZHsxMjM0fSIsWzEyMCw2MCwzNCwxXV0sWzIsMiwiXFxjaXJjbGVkezIzNDV9IixbMTIwLDYwLDM0LDFdXSxbMiwzLCJcXGNpcmNsZWR7MTIzNDV9IixbMTIwLDYwLDM0LDFdXSxbMiw4LCJcXGhfNV57LTF9XFxoXzIiLDFdLFs0LDEyLCJcXGRkXzNeey0xfVxcaF8yIiwxXSxbNiwxNCwiXFxkZF97M31eey0xfVxcaF8yIiwxXSxbOCwxMiwiXFxkZF97M31eey0xfVxcaF80IiwxXSxbOSwxMywiXFxkZF97M31eey0xfVxcaF80IiwxXSxbMTEsMSwiXFxoX3sxfV57LTF9XFxkZF8yIiwxXSxbMTMsNSwiXFxoX3sxfV57LTF9XFxkZF8yIiwxXSxbMTUsNywiXFxoX3sxfV57LTF9XFxkZF8yIiwxXSxbMTQsMTAsIlxcaF97NX1eey0xfVxcZGRfNCIsMV0sWzE1LDExLCJcXGhfezV9XnstMX1cXGRkXzQiLDFdXQ==
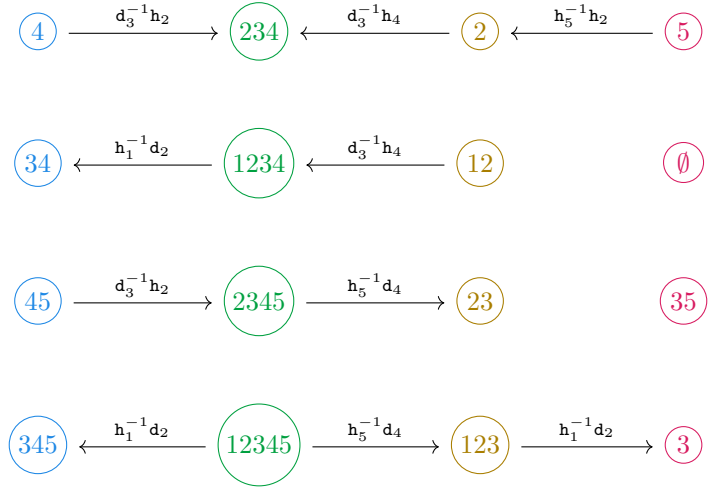
\[\begin{tikzcd}
	\textcolor{myblue}{{\circled{4}}} && \textcolor{mygreen}{{\circled{234}}} && \textcolor{myyellow}{{\circled{2}}} && \textcolor{mypink}{{\circled{5}}} \\
	\textcolor{myblue}{{\circled{34}}} && \textcolor{mygreen}{{\circled{1234}}} && \textcolor{myyellow}{{\circled{12}}} && \textcolor{mypink}{{\circled{$\emptyset$}}} \\
	\textcolor{myblue}{{\circled{45}}} && \textcolor{mygreen}{{\circled{2345}}} && \textcolor{myyellow}{{\circled{23}}} && \textcolor{mypink}{{\circled{35}}} \\
	\textcolor{myblue}{{\circled{345}}} && \textcolor{mygreen}{{\circled{12345}}} && \textcolor{myyellow}{{\circled{123}}} && \textcolor{mypink}{{\circled{3}}}
	\arrow["{\dd_3^{-1}\h_2}", from=1-1, to=1-3]
	\arrow["{\dd_{3}^{-1}\h_4}"', from=1-5, to=1-3]
	\arrow["{\h_5^{-1}\h_2}"', from=1-7, to=1-5]
	\arrow["{\h_{1}^{-1}\dd_2}"', from=2-3, to=2-1]
	\arrow["{\dd_{3}^{-1}\h_4}"', from=2-5, to=2-3]
	\arrow["{\dd_{3}^{-1}\h_2}", from=3-1, to=3-3]
	\arrow["{\h_{5}^{-1}\dd_4}", from=3-3, to=3-5]
	\arrow["{\h_{1}^{-1}\dd_2}"', from=4-3, to=4-1]
	\arrow["{\h_{5}^{-1}\dd_4}", from=4-3, to=4-5]
	\arrow["{\h_{1}^{-1}\dd_2}", from=4-5, to=4-7]
\end{tikzcd}\]
\end{example}

We now argue that this calculation can be leveraged to show that every $\f$ with $\G_{\f}\in\{A, P_5\}$ and $x_3\notin X_{\f}$ has $\V_{\f} = V(\chi_1\chi_5)$. 
Note that by Lemmas \ref{lemma leafs and edges} and \ref{lemma variables present for degree 2 vertex} along with Remark \ref{remark clique complex}, any such $\f$ must have \[X_{\min}:=\{x_1,x_5,x_{12},x_{23},x_{34},x_{45}\}\subseteq X_{\f}\subseteq \{x_1,x_2,x_4,x_5,x_{12},x_{23},x_{24},x_{34},x_{45}, x_{234}\}=:X_{\max}.\]
In Proposition \ref{proposition Taylor graphs are order respecting} we argue that adding variables to $X_{\f}$ has the effect of removing edges in $\T_{\f}$.
Intuitively, this is because more variables give larger neighborhoods and more disparate $\lcm$'s.
For $i\in \{2,4\}$, adding $x_{i}$ to $X_{\min}$ removes every $\dd_i$ edge in $\T_{\f}$. 
Moreover adding either of $x_{24}$ or $x_{234}$ causes $f_2$ and $f_4$ to be neighbors in $\G_{\f}$, removing any $\h_{2}$ and $\h_{4}$ with target containing $\{f_2,f_4\}$. 
Hence the Taylor graph for $X_{\max}$ is depicted by Example \ref{running example taylor graph}, where the dashed edges are removed. 
Notably, removing these edges preserves the homotopy sources and sinks for $\chi_1$ and $\chi_5$ along with the $\{f_1,f_5\}$-perfect matching $M$ of Example \ref{running example perfect matching}. 
Fewer edges in $\T_{\f}$ yields fewer edges of $\Aux_{\f}^M$ hence $\Aux_{\f}^M$ is still acyclic and thus $M$ remains triangular. 
It thus follows that for all $\f$ with $X_{\min}\subseteq X_{\f}\subseteq X_{\max}$, we have $\V_{\f} = V(\chi_1\chi_5)$.
This proves the portion of \cite[Theorem 6.16]{BGP} stated in Example \ref{running example taylor graph}.

\begin{proposition}\label{proposition Taylor graphs are order respecting}
    Suppose $\f$ and $\g$ minimally generate square-free monomial ideals with the same number of minimal generators.
    If $X_{\f} \subseteq X_{\g}$ then $\T_{\g}$ is a subgraph of $\T_{\f}$.  
\end{proposition}
\begin{proof}
    It suffices to show the case that $X_{\g} = X_{\f}\cup \{x_\sigma\}$.
    Let us consider how $\T_{\f}$ changes when $x_{\sigma}$ changes from absent to present.
    All edges of $\G_{\f}$ are also edges of $\G_{\g}$. Moreover, the 2 element subsets of $\sigma$ give rise to (possibly) more edges of $\G_{\g}$.
    Consequently, no homotopy edges are added and some homotopy edges may be removed.
    For differential edges note that for all $\theta,$ we have $f_\theta \mid g_\theta$ and $g_\theta / f_\theta\in\{1,x_\sigma\}$.
    It follows that no differential edges are added and some differential edges may be removed.
\end{proof}

We prove a quick proposition in the local setting.  Let $R$ be a local ring with minimal regular presentation given by $Q/I$. 
\begin{proposition}
    Let $\p$ be a prime ideal in $R$ and consider the local ring $R_{\p}$. Then the following inequality holds 
    \[
    \dim \cV_{R_{\p}}(R_{\p}) \les \dim \cV_R(R).
    \]
\end{proposition}

\begin{proof}
    The dimension of the support variety $\cV_R(R)$ gives the complexity of $R$ over $E$ (see \cite[5.2.8 and 5.2.9]{Pollitz:2021}). That is, $\dim \cV_R(R)$ gives the minimal degree of a polynomial that bounds the growth of the following sequence $\left (\dim_k \left(\Ext_E^i(R,k)\right) \right )_{i \in \mathbb N}$. Let $F$ be a minimal semifree $E$-resolution as in \cite[B.6]{Avramov/Iyengar/Nasseh/SatherWagstaff:2019}. Then $\dim_k(\Ext^i_E(R,k)) = \dim_k(F \otimes_E k)_i$. Localizing at $\q = Q \cap \p$, we have that $F_{\q}$ is an $E_{\q}$ resolution of $R_{\p}$. It follows that $\dim_{\kappa({\q})}(\Ext^i_{E_{\q}}(R_{\p},\kappa(\q))) \les  \dim_k \left(\Ext_E^i(R,k)\right)$. Hence $\dim \cV_{R_\p} (R_{\p} )\les \dim \cV_R(R) $.
\end{proof}

\begin{corollary}\label{corollary dimension is continuous}
    Suppose $\f$ and $\g$ minimally generate square-free monomial ideals with the same number of minimal generators.
    If $X_{\f} \subseteq X_{\g}$  then $\dim \V_{\f} \les \dim \V_{\g}$ 
\end{corollary}
\begin{proof}
    We remark that $\V_{\f}$ is only sensitive to the combinatorics of $\f$, not the setting, hence we may assume $Q = k\llbracket X_{\g}\rrbracket$ and $R = Q/(\g)$.
    Since none of the monomials in $\f$ are 1, we have that if $\p = (X_{\f})$ then $R_\p = Q_\p/(\g)_{\p} = Q_\p/(\f)$. 
\end{proof}

Proposition \ref{proposition Taylor graphs are order respecting} and Corollary \ref{corollary dimension is continuous} motivate further study of the relation $X_{\f}\mapsto \T_{\f}$ which is the content of the next part of this section.
Recall that in general $X_{\f}$ can be identified as an element of $2^{2^{[n]}}$, however we note that not all elements of $2^{2^{[n]}}$ arise as $X_{\f}$ for some $\f$.
The obstructions occur from the requirement that $f_i\nmid f_j$.
Recall that the \emph{support} of a monomial is the set of variables dividing it.

\begin{proposition}\label{proposition X realizing valid f}
Let $[n] = \{1,\ldots,n\}$ and $Q = k[\setbuild{x_\sigma}{\sigma \in 2^{[n]}}]$ and for $i\neq j$ we consider prime ideals  $P_{ij} = (\setbuild{x_\sigma}{i\in \sigma,\; j\notin\sigma })$.
An element $X\in 2^{2^{[n]}}$ has $X=X_{\f}$ for some $\f$ if and only if $X$ is the support of a monomial in \[J = \bigcap_{i\neq j}P_{ij}.\]
\end{proposition}
\begin{proof}
    Given $X$, checking that $X=X_{\f}$ for some $\f$ amounts to showing that the monomials given by $f_i=\prod_{i\in \sigma\in X}x_\sigma$ indeed form a minimal generating set. 
    Equivalently, for all $i\neq j$ we require that $f_i\nmid f_j$.
    Note that each of these conditions must be witnessed by some $x_\sigma$ with $i\in \sigma$ and $j\not \in \sigma$, or rather $x_\sigma\in P_{ij}$.
    It follows that a monomial being contained in every $P_{ij}$ is equivalent to its support containing a witness to every $f_i\nmid f_j$.
\end{proof}

In light of Propositions \ref{proposition Taylor graphs are order respecting} and \ref{proposition X realizing valid f}, the minimal generators of $J$ correspond to the largest possible Taylor graphs. 

\begin{example}\label{example largest Taylor graphs}
    Let us consider two extremal cases of minimal generators of $J$. 
    Take \[X_{\f_h} = \{x_{1}, x_2,\ldots,x_n\}\text{ and } X_{\f_d} = \{x_{[n]\smallsetminus \{f_1\}}, x_{[n]\smallsetminus \{f_2\}},\ldots, x_{[n]\smallsetminus \{f_n\}}\}.\] 
    In particular, $\f_h$ generates a monomial complete intersection and $\G_{\f_h}$ has no edges. 
    Thus $\T_{\f_h}$ has every homotopy edge possible.
    In general, a Taylor graph $\T_{\f}$ having no differential edges is equivalent to $X_{\f_h}\subseteq X_{\f}$.   
    On the other hand, $\G_{\f_d}$ is the complete graph.
    Consequently the only homotopy edges of $\T_{\f_d}$ have targets that are singletons, which always occur. 
    Moreover, every differential edge possible (i.e. with target cardinality at least 2) is present in $\T_{\f_d}$.
\end{example}

In general, the number of minimal generators of $J$ is hopelessly large to compute.
As \cite{BGP} shows, support varieties are heavily dependent on GCD graphs which serve as a nice stepping stone in classification.
We now construct an analogue of $J$ for a fixed GCD graph $\G$. 
Recall from Remark \ref{remark clique complex} that $K_\G$, the clique complex of $\G,$ is the set of all complete induced subgraphs.

\begin{proposition}\label{proposition X realizing given G}
    Fix a graph $\G$ and set $Q = k[K_{\G}]$. 
    For every edge $e = \{i,j\}\in\G $ set \[P_e = 
    (\setbuild{x_\varphi}{i\in \varphi,\; j \in \varphi})\cap(\setbuild{x_\sigma}{i\in \sigma,\; j\not \in \sigma})\cap
    (\setbuild{x_\theta}{j\in \theta,\; i\not \in \theta}).\] 
    The supports of square-free monomials in $J_{\G} = \bigcap_{e\in \G}P_e$ are exactly the elements $X_{\f} \in 2^{K_{\G}}$ with $\G = \G_{\f}$.
\end{proposition}
\begin{proof}
Note that restricting variables to $K_{\G}$ gives that no edge not in $\G$ could be witnessed by a variable. 
Hence it suffices to note that every edge of $\G$ is witnessed by some $x_{\varphi}$, whereas the $x_\sigma$ and $x_\theta$ force $X$ to yield a minimal generating set.
\end{proof}

\begin{example} \label{example X realizing given G}
    Let us compute $J_{\G}$ for the GCD graph in Example \ref{example first gcd graphs}. 
    \begin{align*}
         P_{12} &= (x_{12},x_{123})\cap(x_{1},x_{13},x_{14})\cap (x_{2},x_{23})  \label{e = \{f_1,f_2\}}\\ 
        P_{13}&=  (x_{13},x_{123})\cap(x_{1},x_{12},x_{14})\cap (x_3,x_{23}) \\
        P_{14}&=  (x_{14})\cap (x_{1},x_{12}, x_{13}, x_{123})\cap (x_4) \\
        P_{23}&= (x_{23}, x_{123})\cap (x_{2},x_{12}) \cap (x_{3},x_{13})
    \end{align*} 
    Hence \[J_{\G} = ( x_{4}x_{12}x_{13}x_{14}x_{23},\;x_{2}x_{3}x_{4}x_{14}x_{123},\; x_{2}x_{4}x_{13}x_{14}x_{23}x_{123},\; x_{3}x_{4}x_{12}x_{14}x_{23}x_{123}).\]

    Note then that in Example \ref{example set of variables}, $X_{\f_1}$ and $X_{\f_2}$ are the supports of the first and second listed generators of $J_{\G}$ respectively. 
    We observe that the resulting Taylor graphs in Example \ref{example taylor graph definition} are incomparable. 
    We remark that $J_{\G}$ is principle if and only if $\G$ is triangle-free.
    \end{example}
    \begin{remark}\label{remark computer good}
        Proposition \ref{proposition X realizing given G} invites a computational approach to collecting data. In particular, our arguments in Section \ref{section 6 generated ideals} stem from computing $J_{\G}$ for 41 graphs and analyzing the largest possible Taylor graphs corresponding to the minimal generators. 
    \end{remark}
We now prepare a lemma providing a sufficient condition for a matching to be triangular.
In general, arguing that $\Aux_{\f}^M$ is acyclic requires understanding the paths that can occur.
By Remark \ref{remark taylor weights give auxilliary weights}, directed paths in $\Aux_{\f}^M$ are given by certain structures in the Taylor graph.

\begin{definition}\label{definition walk on taylor graph}
We treat $\h_i$, $\dd_i$, $\h_i^{-1}$, and $\dd_i^{-1}$ as functions that are partially defined on the vertices of a Taylor graph.  
The function $\e_i$ sends $\e_i^s$ to $\e_i^t$, whereas $\e_i^{-1}$ sends $\e_i^t$ to $\e_i^s$. 
A \emph{walk} on a Taylor graph is a pair of vertices $v_s$ and $v_t$ along with a word $w=w_1w_2\cdots w_\ell$ in the alphabet $\{\h_1, \h_{1}^{-1},\dd_1, \dd_{1}^{-1},\ldots,\h_n, \h_{n}^{-1},\dd_n, \dd_{n}^{-1}\}$ such that $v_t = (w_1\circ w_2\circ \cdots \circ w_\ell )(v_s)$. We say a walk is \emph{balanced} if the letters of $w_1w_2\cdots w_\ell$ alternate between $\e_i$ and $\e_j^{-1}$.
In Example \ref{running example taylor graph}, $v_{34} = \h_1^{-1}\dd_{2}\dd_{3}^{-1}\h_4(v_{12})$ is a balanced walk whereas $v_{15} = \h_{3}^{-1}\dd_4\dd_2(v_{12345})$ is not a balanced walk.
Given a $\sigma$-perfect matching $M$, an \emph{$M$-walk} is a balanced walk such that $v_s$ and $v_t$ are sources of edges in $M$ and every $\e_i^{-1}$ in the walk has that $\e_i$ is an edge of $M$. 
Note that $M$-walks on $\T_{\f}$ correspond to directed paths in $\Aux_{\f}^M$.
The walk above from $v_{12}$ to $v_{34}$ is an $M$-walk and corresponds to the directed path in $\Aux_{\f}^M$ between these vertices.
\end{definition}

Recall that a proper coloring of a graph with vertex set $V$ and edge set $E$ is a function $c$ with domain $V$ such that an edge in $E$ using vertices $i$ and $j$ implies $c(i)\neq c(j)$. 
We will typically use two different colorings of $\Aux_{\f}^M$ to argue about possible cycles that could occur. 

For a $\sigma$-perfect matching, $M$, we set the notation $D_i(M) = \{\dd_{\theta,i}\in M\}$ for all $1\leq i\leq n$ as well as $H_j(M) = \{\h_{\theta,j}\in M\}$ for all $j\in \sigma$. 
Moreover, we denote the set of source vertices of edges in $D_i(M)$ by $D^s_i(M)$ and corresponding targets by $D^t_i(M)$. We also use the same convention for $H_j(M)$.

\begin{proposition}\label{proposition type coloring is a coloring}
    If $M$ is a $\sigma$-perfect matching then the function $\dd_{\theta,i}^s\mapsto D_i^s(M)$ and $\h_{\theta,j}^s\mapsto H_j^s(M)$ is a proper coloring of $\Aux_{\f}^M$.
\end{proposition}
\begin{proof}
    Suppose there is some edge from $\dd_{\theta,i}^s$ to $\dd_{\theta',i}^s$. 
    Note that $f_i\in \dd_{\theta,i}^s$ but $f_i\notin \dd_{\theta',i}^t$, so the edge of $\T_{\f}$ witnessing this edge of $\Aux_{\f}^M$ must also be $\dd_{\theta,i}$ which is preposterous. 
    An identical argument holds for edges contained in some $H_j(M)$.
\end{proof}

\begin{definition} \label{definition matching theta-determined}
    We call the coloring in Proposition \ref{proposition type coloring is a coloring} the \emph{type coloring}. 
    We now define a (possibly) finer coloring, under the assumption that $M$ satisfies the following combinatorial condition.
    Let $\theta\subseteq[n]$ and $M$ be a $\sigma$-perfect matching such that for all $i\in \theta,$ we have $D_i(M) = H_i(M)=\emptyset$. 
    For every \emph{coordinate} $\varphi \in  2^\theta$, set $M_\varphi = \setbuild{\e_{\psi,j}^s} {\e_{\psi,j}\in M,\; \varphi = \theta\cap\psi}$ i.e. the sources of edges in $M$, denoted $v_\psi$, with $\theta\cap\psi = \varphi$. 
    We say $M$ is \emph{$\theta$-determined} if every coordinate $\varphi$ has $M_\varphi\subseteq H_i^s(M)$ or $M_\varphi\subseteq D_i^s(M)$. 
    In this case the function $\e_{\psi,j}^s\mapsto M_{\theta\cap\psi}$ is the \emph{coordinate coloring}.
\end{definition}

\begin{example}\label{running example type and coordinate coloring}
    The $\{f_1, f_5\}$-perfect matching given in Example \ref{running example perfect matching} is $\{f_2, f_4\}$-determined. 
    Recall this matching was given as \(M = 
     \h_{\emptyset,1}\otimes 2^{\{3,5\}}  \cup 
     \h_{\{4\},1}\otimes 2^{\{3,5\}} \cup 
     \h_{\{2\},5}\otimes 2^{\{1,3\}}\cup 
     \dd_{\{24\},3}\otimes 2^{\{1,5\}}\) which we will observe coincides with the coordinate coloring. 
    In this case, the coordinate coloring is a 4-coloring which is finer than the 3 color type coloring. 
    In particular, $H_1^s = M_{\emptyset}\cup M_4$, which in turn consists of the sources of edges in $\h_{\emptyset,1}\otimes 2^{\{3,5\}}$ and $\h_{\{4\},1}\otimes 2^{\{3,5\}}$ respectively.
    Moreover, $H_{5}^s = M_2$ consists of the sources of edges in $\h_{\{2\},5}\otimes 2^{\{1,3\}}$.
    Similarly $D_1 = M_{24}$ consists of the sources of edges in $\dd_{\{24\},3}\otimes 2^{\{1,5\}}$.
\end{example}

\begin{lemma} \label{lemma 2 determined matchings are triangular}
If $M$ is a $\sigma$-perfect matching of $\T_{\f}$ that, up to relabeling, is $\{f_1,f_2\}$-determined then $M$ is triangular. 
\end{lemma}
\begin{proof}
 We have the coordinate coloring $M = M_{{\tt00}}\cup M_{{\tt01}}\cup M_{{\tt10}}\cup M_{{\tt11}}$ with $\e^s\in M_{b_1b_2}$ where $b_i$ is the indicator function of $f_i$ applied to $\e^s.$
 Note that the edges in $M$ with source in $M_{b_1b_2}$ have the same type (i.e. $\dd_i$ or $\h_j$ for whichever of $D_i^s$ or $H_j^s$ that $M_{b_1b_2}$ is a subset of) which is denoted by $\e_{b_1b_2}$. 
 Notably, no $\e_{b_1b_2}$ can be $\dd_i$ or $\h_i$ for $i\in \{1,2\}$. 
 Below are the coordinate color classes of $\Aux_{f}^M$ as well as possible edge weights between them.

\begin{center}
\begin{minipage}{\textwidth}
\centering
\begin{tikzpicture}[scale=1, fatnode/.style={circle, inner sep=0pt, minimum size=20pt, draw}]

  % Nodes
\node[fatnode] (v11) at (2,-2) {$M_{{\tt11}}$};
\node[fatnode] (v01) at (2,2) {$M_{{\tt01}}$};
\node[fatnode] (v10) at (-2,-2) {$M_{{\tt10}}$};
\node[fatnode] (v00) at (-2,2) {$M_{{\tt00}}$};

\draw[-{Stealth}, thick] (v00) edge[bend left = 10] node[midway, above, scale = .9, minimum size = 20pt]{$\e_{{\tt01}}^{-1}\h_2$}  (v01);
\draw[-{Stealth}, thick] (v01) edge[bend left = 10] node[midway, below, scale = .9, minimum size = 20pt]{$\e_{{\tt00}}^{-1}\dd_2$}  (v00);

\draw[-{Stealth}, thick] (v01) edge[bend left = 10] node[midway, right, scale = .9, minimum size = 20pt]{$\e_{{\tt11}}^{-1}\h_1$}  (v11);
\draw[-{Stealth}, thick] (v11) edge[bend left = 10] node[midway, left, scale = .9, minimum size = 20pt]{$\e_{{\tt01}}^{-1}\dd_1$}  (v01);

\draw[-{Stealth}, thick] (v10) edge[bend left = 10] node[midway, above, scale = .9, minimum size = 20pt]{$\e_{{\tt11}}^{-1}\h_2$}  (v11);
\draw[-{Stealth}, thick] (v11) edge[bend left = 10] node[midway, below, scale = .9, minimum size = 20pt]{$\e_{{\tt10}}^{-1}\dd_2$}  (v10);

\draw[-{Stealth}, thick] (v00) edge[bend left = 10] node[midway, right, scale = .9, minimum size = 20pt]{$\e_{{\tt10}}^{-1}\h_1$}  (v10);
\draw[-{Stealth}, thick] (v10) edge[bend left = 10] node[midway, left, scale = .9, minimum size = 20pt]{$\e_{{\tt00}}^{-1}\dd_1$}  (v00);
\end{tikzpicture}\\
\vspace{10pt}
Coordinate Coloring and Edge Labels of $\Aux_{\f}^M$
\end{minipage}

\end{center}

Any two-cycle would have to contain $\dd_i \e^{-1}\h_i$. This is impossible as $\dd_i^s$ would have to contain at least two neighbors of $f_i$ in $\G_{\f}$ whereas $\h_i^s$ must contain 0 neighbors.
Hence any cycle must be oriented (clockwise or counterclockwise) and 4-color periodic. Say a clockwise cycle exists. Starting at $M_{{\tt00}}$ and traversing  $\e_{{\tt00}}^{-1}\dd_1\cdot \e_{{\tt10}}^{-1}\dd_2\cdot \e_{{\tt11}}^{-1}\h_1\cdot \e_{{\tt01}}^{-1}\h_2$ (from right to left) must result in the identity. Being able to apply $\dd_2$ after traversing $\e_{{\tt11}}^{-1}\h_1\cdot \e_{{\tt01}}^{-1}\h_2$ requires that $\e_{{\tt11}}$ and $\e_{{\tt01}}$ are differential edges that remove neighbors of $f_2$. 
It cannot be that $\e_{{\tt01}}$ removes a neighbor of $f_1$, otherwise we could not have traversed $\h_1$ at $M_{{\tt01}}$. 
Thus upon arriving at $M_{{\tt10}}$, $\dd_1$ can only be traversed if $\e_{{\tt11}}$ and $\e_{{\tt10}}$ are differential edges removing neighbors of $f_1$. 
Hence to arrive at $M_{{\tt10}}$, the cardinality of the vertex must increase by 4 and since traversing $\e_{{\tt00}}^{-1}\dd_1$ can remove at most 2 elements from the vertex, no such cycle may exist.
\end{proof}

\section{Support Varieties of Six Generated Monomial Ideals }\label{section 6 generated ideals}

In this section we classify the support varieties $\V_{\f}$ that occur when $\f$ has six elements. We follow the same broad strategy as \cite{BGP}, using GCD graphs as a basis for classification. 

\begin{theorem}\label{theorem A.2}
    Let $R$ be a local or positively graded ring with minimal regular presentation given by $Q/I$ where $I$ is a monomial ideal minimally generated by six monomials listed by $\f$. 
    Then up to a reordering of $\f$, the varieties realized as $\V_{\f}$ are listed below:
    \begin{itemize}
        \item a coordinate subspace of $\mathbb A^6_k$ with dimension not equal to $1$
        \item $V(\chi_1,\;\chi_4\chi_6)$ 
        \item $V(\chi_4\chi_6)$
        \item $V(\chi_4\chi_6,\; \chi_5\chi_6)$
        \item $V(\chi_2\chi_4\chi_6)$
        \item $V(\chi_1\chi_3\chi_5+\chi_2\chi_4\chi_6)$.
    \end{itemize}
\end{theorem}

\begin{proof} We first note that \cite[Lemma 6.13]{BGP} shows that if the connected components of $\G_{\f}$ are $\G_{\f_1}, \ldots \G_{\f_s}$ then $\V_{\f}\cong \V_{\f_1}\times \cdots \times \V_{\f_s}$. 
Thus if $\G_{\f}$ is not connected, \cite[Theorem 6.14]{BGP} gives that if $\G_{\f_i}$ has 4 or fewer vertices then $\V_{\f_i}$ is a coordinate subspace not of dimension 1.
Such varieties are closed under products.
The only other case when $\G_{\f}$ is not connected is that $\G_{\f_1}$ is a singleton and $\G_{\f_2}$ is a connected graph with 5 vertices. 
Note that $\V_{\f_1}$ is a point, hence $\V_{\f}$ is an embedding of $\V_{\f_2}$ in $\A_k^{6}.$ 
Applying \cite[Theorem 6.16]{BGP} yields either a coordinate subspace or a variety isomorphic to $V(\chi_1,\; \chi_{4}\chi_6)$.

We now consider connected graphs.
By \cite[Remark 6.2]{BGP} any $\G_{\f}$ with a vertex with degree 5 has full support, i.e.  $\V_{\f} = \A_k^6$.
Using SageMath \cite{sagemath} to access the database \cite{grout_graph_database_nodate}, we find there are 78 connected graphs with 6 vertices and no vertices with degree 5. 
Of these graphs, 37 satisfy the hypothesis of \cite[Lemma 6.12]{BGP} and hence yield full support varieties. 
The remaining 41 graphs appear throughout the remainder of the paper, named Graph 1 to Graph 41. 
Moreover, we group these graphs by types A, B, C, and F. 
This classification is based on the support varieties that can be realized by ideals with each GCD graph. 

Type F stands for full as these graphs always give rise to full support varieties. 
We further assign subtypes F$_1$ to F$_4$ of type F based on how we prove they yield full support varieties. 
In particular, we introduce several new lemmas that can be used to determine when certain GCD graphs always yield full support.

The other types can realize ideals with interesting support (i.e. not full). The two type A graphs fall into an infinite family of graphs considered in Section \ref{section families}. 
These families form a generalization of the two graphs on 5 vertices realizing interesting support, i.e. A and $P_5$ of Example \ref{running example taylor graph}. 
Every type B graph is a subgraph of Graph 27, but this is not a defining feature as the type A graphs also are subgraphs of Graph 27. 
Type C stands for cycle, as the graphs contain the 6-cycle as a subgraph. 
The table below has one caveat* given by the 6-cycle, which we consider in great detail in Section \ref{section cycle varieties}. 
Specifically, graphs 38-40 realize exactly the support varieties listed in the table, but graph 41 realizes both of these along with two additional support varieties: $V(\chi_1\chi_3\chi_5)$ and $V(\chi_1\chi_3\chi_5 + \chi_2\chi_4\chi_6)$.

\begin{table}[H]
\renewcommand{\arraystretch}{1.25}
\begin{tabular}{|c|c|c|cc|c|}
\hline
Type  & Count & Graph Numbers & \multicolumn{2}{c|}{Support Varieties Realized}               & Proof                                               
\\ \hline
F$_1$ & 6     & 1-6           & \multicolumn{2}{c|}{$\A_k^{6}$}                                   & Lemma \ref{lemma isolated in degree 3}              \\ \hline
F$_2$ & 2     & 7-8           & \multicolumn{2}{c|}{$\A_k^{6}$}                                   & Lemma \ref{lemma pair of edges with no common edge} \\ \hline
F$_3$ & 11    & 9-19          & \multicolumn{2}{c|}{$\A_k^{6}$}                                   & Lemma \ref{lemma edges and triangles}               \\ \hline
F$_4$ & 7     & 20-26         & \multicolumn{2}{c|}{$\A_k^{6}$}                                   & Theorem \ref{theorem F for full support}  
\\ \hline
A     & 2     & 36-37         & \multicolumn{1}{c|}{$\A_k^{6}$} & $V(\chi_4\chi_5, \chi_4\chi_6)$ & Theorem \ref{theorem DB WT varieties}               \\ \hline
B     & 9     & 27-35         & \multicolumn{1}{c|}{$\hspace{5pt}\A_k^{6}\hspace{5pt}$} & $V(\chi_4\chi_6)$               & Theorem \ref{theorem interesting supports 6 gen}    \\ \hline
C     & 4*    & 38-41         & \multicolumn{1}{c|}{$\A_k^{6}$} & $V(\chi_2\chi_4\chi_6)$         & Corollary \ref{corollary gcd fiber over hexagon}     \\ \hline
\end{tabular}
\end{table}
\end{proof}

We start with two lemmas whose arguments are based on counting dimensions of certain subspaces of cycles and boundaries of $\T_{\f}(a)$ for all $a$. These arguments motivate the argument we later provide for Lemma \ref{lemma odd length path in taylor graph}.

\begin{lemma}\label{lemma more sources than neighbors}
    Let $S$ be a set of sources in the Taylor graph $\T_{\f}$. Let $N$ be the neighbors of $S$ in $\T_{\f}$. 
    If $|S|>|N|$ then $\V_{\f}$ is full.
\end{lemma}

\begin{proof}
    Consider the submatrix of $\T_{\f}$ with columns corresponding to $S$ and rows corresponding to $N$. Since there are more columns than rows, there is some non-trivial solution to this submatrix. We claim this solution is a cycle of $\T_{\f}$ as the only rows with entries in columns given by $S$ belong to $N$. Moreover, this cycle is not a boundary as $S$ consists of sources.
\end{proof}

\begin{lemma}\label{lemma more sinks than sources}
    Let $S$ be a set of sinks in the Taylor graph $\T_{\f}$. Let $N$ be the set of vertices that are sources of edges in $\T_{\f}$ with targets in $S$. If $|S|>|N|$ then $\V_{\f}$ is full.
\end{lemma}
\begin{proof}
    We will explicitly find a cycle that is not a boundary of the linear transformation $\T_{\f}(a)$ for all $a \in \A^n_k$.  Let us consider $Z = \spank(S)$, by assumption all $\sigma\in S$ are sinks so $\T_f(a)(v_\sigma)=0$ hence $Z\subseteq \ker (\T_{\f}(a))$.   
    Set $B = \spank(\setbuild{\T_{f}(a)(v_\theta)}{v_\theta \in N})$ and observe that $\im(\T_{\f}(a))\cap Z \subseteq B$. However, by assumption $\dim_k(Z)>\dim_k(B)\geq \dim_k(\im(\T_{\f}(a))\cap Z )$ so $Z$ contains cycles that are not boundaries. Thus $\V_{\f}$ is full. 
\end{proof}

The following graphs constitute type F$_1$, each satisfying the hypothesis of Lemma \ref{lemma isolated in degree 3} by taking $\sigma = \{3,4,5\}$.

\begin{center}
    \normalfont
    \noindent \begin{minipage}{0.16\textwidth}
\centering 
\begin{tikzpicture}[x=.3\textwidth,y=.3\textwidth,scale=1, every node/.style={circle, inner sep=0pt, minimum size=13pt, draw}]

  % Nodes
  \node (v1) at (1,0) {1};
  \node (v2) at (1,2) {2};
  \node (v3) at (-1,2) {3};
  \node (v4) at (-1,0) {4};
  \node (v5) at (-.55,1) {5};
  \node (v6) at (.55,1) {6};

  % Edges
  \draw (v1) -- (v4);
  \draw (v4) -- (v5);
  \draw (v5) -- (v6);
  \draw (v5) -- (v3);
  \draw (v3) -- (v2);
 
\end{tikzpicture}
Graph 1 \ognum{6}
\end{minipage}
%%%%%%%%%%%%%%%%%%%%%%%%%%%%%%%%%%%%%%%%%%%%%%
%graph 7 in Stephen's 
\begin{minipage}{0.16\textwidth}
\centering

\begin{tikzpicture}[x=.3\textwidth,y=.3\textwidth,scale=1, every node/.style={circle, inner sep=0pt, minimum size=13pt, draw}]

  % Nodes
  \node (v1) at (1,0) {1};
  \node (v2) at (1,2) {2};
  \node (v3) at (-1,2) {3};
  \node (v4) at (-1,0) {4};
  \node (v5) at (-.55,1) {5};
  \node (v6) at (.55,1) {6};

  % Edges
  \draw (v1) -- (v4);
  \draw (v4) -- (v5);
  \draw (v5) -- (v3);
  \draw (v3) -- (v4);
  \draw (v3) -- (v2);
  \draw (v5) -- (v6);
 
\end{tikzpicture}
Graph 2 \ognum{7}
\end{minipage}
%%%%%%%%%%%%%%%%%%%%%%%%%%%%%%%%%%%%%%%%%%%%%%
%graph 22 in Stephen's 
\begin{minipage}{0.16\textwidth}
\centering

\begin{tikzpicture}[x=.3\textwidth,y=.3\textwidth,scale=1, every node/.style={circle, inner sep=0pt, minimum size=13pt, draw}]

  % Nodes
  
  \node (v1) at (1,0) {1};
  \node (v2) at (1,2) {2};
  \node (v3) at (-1,2) {3};
  \node (v4) at (-1,0) {4};
  \node (v5) at (-.55,1) {5};
  \node (v6) at (.55,1) {6};
  
  % Edges
  \draw (v5) -- (v6);
  \draw (v5) -- (v3);
  \draw (v3) -- (v2);
  \draw (v2) -- (v1);
  \draw (v1) -- (v4);
  \draw (v4) -- (v5);
 
\end{tikzpicture}
Graph 3 \ognum{22}
\end{minipage}
%%%%%%%%%%%%%%%%%%%%%%%%%%%%%%%%%%%%%%%%%%%%%%
%graph 23 in Stephen's 
\begin{minipage}{0.16\textwidth}
\centering
\begin{tikzpicture}[x=.3\textwidth,y=.3\textwidth,scale=1, every node/.style={circle, inner sep=0pt, minimum size=13pt, draw}]

  % Nodes
  \node (v1) at (1,0) {1};
  \node (v2) at (1,2) {2};
  \node (v3) at (-1,2) {3};
  \node (v4) at (-1,0) {4};
  \node (v5) at (-.55,1) {5};
  \node (v6) at (.55,1) {6};

  % Edges
  \draw (v5) -- (v6);
  \draw (v5) -- (v4);
  \draw (v4) -- (v1);
  \draw (v1) -- (v2);
  \draw (v2) -- (v3);
  \draw (v3) -- (v5);
  \draw (v3) -- (v4);
 
\end{tikzpicture}
Graph 4 \ognum{23}
\end{minipage}
%%%%%%%%%%%%%%%%%%%%%%%%%%%%%%%%%%%%%%%%%%%%%%
%graph 25 in Stephen's 
\begin{minipage}{0.16\textwidth}
\centering
\begin{tikzpicture}[x=.3\textwidth,y=.3\textwidth,scale=1, every node/.style={circle, inner sep=0pt, minimum size=13pt, draw}]

  % Nodes
  \node (v1) at (1,0) {1};
  \node (v2) at (1,2) {2};
  \node (v3) at (-1,2) {3};
  \node (v4) at (-1,0) {4};
  \node (v5) at (-.55,1) {5};
  \node (v6) at (.55,1) {6};

  % Edges
  \draw (v5) -- (v6);
  \draw (v4) -- (v1);
  \draw (v1) -- (v2);
  \draw (v2) -- (v3);
  \draw (v3) -- (v5);
  \draw (v3) -- (v4);
 
\end{tikzpicture}
Graph 5 \ognum{25}
\end{minipage}
%%%%%%%%%%%%%%%%%%%%%%%%%%%%%%%%%%%%%%%%%%%%%%
%graph 38 in Stephen's 
\begin{minipage}{0.16\textwidth}
\centering
\begin{tikzpicture}[x=.3\textwidth,y=.3\textwidth,scale=1, every node/.style={circle, inner sep=0pt, minimum size=13pt, draw}]

  % Nodes
  \node (v1) at (1,0) {1};
  \node (v2) at (1,2) {2};
  \node (v3) at (-1,2) {3};
  \node (v4) at (-1,0) {4};
  \node (v5) at (-.55,1) {5};
  \node (v6) at (.55,1) {6};

  % Edges
  \draw (v1) -- (v6);
  \draw (v6) -- (v2);
  \draw (v2) -- (v3);
  \draw (v3) -- (v4);
  \draw (v4) -- (v1);
  \draw (v4) -- (v5);
  \draw (v5) -- (v6);
 
\end{tikzpicture}
Graph 6 \ognum{38}
\end{minipage}
\\
\vspace{10pt} 
Type F$_1$ Graphs 
\end{center}

\begin{lemma}\label{lemma isolated in degree 3}
    Suppose that there exists a set $\sigma \subseteq [n]$ with the following properties: 
    \begin{itemize}
        \item $N_{\f}(\sigma)=[n]$
        \item for all $f_i \in  \sigma$, we have $\lvert N_{\f}(f_i) \cap [n]\smallsetminus\sigma \rvert \geq 1$
        \item for all $f_j \in [n]\smallsetminus \sigma$, we have $\lvert N_{\f}(f_j) \cap \sigma \rvert = 1$ 
    \end{itemize}
then $\V_{\f}$ is full.
\end{lemma}
\begin{proof}
We will show that $v_\sigma$ is an isolated vertex in the Taylor graph $\T_{\f}$, which by Corollary \ref{corollary isolated vertex} implies that $\V_{\f}$ is full. Since $v_\sigma$ is homotopically isolated it suffices to show no differential edges use $v_\sigma$. The third condition, combined with Proposition \ref{proposition homotopically/differentially isolated} implies there is no differential edge $\dd_{\sigma,j}$. For each $f_i\in \sigma$ the second condition gives some neighbor $f_j\notin\sigma$. The third condition gives $N_{\f}(f_j)\cap \sigma = \{f_i\}$. Hence any variable $x_\theta$ witnessing the edge $\{f_i,f_j\}$ of $\G_{\f}$ must have $\theta\cap \sigma = \{f_i\}$. In particular, there is no edge $\dd_{\sigma\smallsetminus\{f_i\},i}$ because $x_{\theta}\mid f_{\sigma}$ and $x_{\theta}\nmid f_{\sigma\smallsetminus{\{f_i\}}}$.
\end{proof}

The following graphs constitute type F$_2$. These satisfy the hypothesis of Lemma \ref{lemma pair of edges with no common edge} by taking edges $\sigma_1 = \{1,2\}$ and $\sigma_2 = \{5,6\}$.

\begin{center}
    \normalfont
    \noindent %graph 8 in Stephen's
\begin{minipage}{0.19\textwidth}
\centering
\begin{tikzpicture}[x=.3\textwidth,y=.3\textwidth,scale=1, every node/.style={circle, inner sep=0pt, minimum size=13pt, draw}]

  % Nodes
  \node (v1) at (-1,1.155) {1};
  \node (v3) at (1,1.155) {3};
  \node (v4) at (1,0) {4};
  \node (v2) at (0,1.73) {2};
  \node (v5) at (0,-.577) {5};
  \node (v6) at (-1,0) {6};

  % Edges
  \draw[line width = 1.5pt] (v6) -- (v5);
  \draw (v5) -- (v4);
  \draw (v3) -- (v2);
  \draw[line width = 1.5pt] (v2) -- (v1);
  \draw (v4) -- (v3);

\end{tikzpicture}
Graph 7 \ognum{8}
\end{minipage}
%graph 20 in Stephen's
\begin{minipage}{0.19\textwidth}
\centering
\begin{tikzpicture}[x=.3\textwidth,y=.3\textwidth,scale=1, every node/.style={circle, inner sep=0pt, minimum size=13pt, draw}]

  % Nodes
  \node (v1) at (-1,1.155) {1};
  \node (v2) at (0,1.73) {2};
  \node (v3) at (1,1.155) {3};
  \node (v4) at (1,0) {4};
  \node (v5) at (0,-.577) {5};
  \node (v6) at (-1,0) {6};

  % Edges
  \draw[line width = 1.5pt] (v6) -- (v5);
  \draw (v5) -- (v4);
  \draw (v4) -- (v3);
  \draw (v3) -- (v2);
  \draw[line width = 1.5pt] (v2) -- (v1);
  \draw (v3) -- (v1);

\end{tikzpicture}
Graph 8 \ognum{20}
\end{minipage}
\\
\vspace{10pt} 
Type F$_2$ Graphs 
\end{center}

\begin{lemma} \label{lemma pair of edges with no common edge}
    Suppose $\{\sigma_1,\ldots,\sigma_s\}$ is a set of edges in $\G_{\f}$ and $\theta = \bigcup_i\sigma_i$ such that:
    \begin{itemize}
        \item $N_{\f}(\theta) = [n]$
        \item for all $i\neq j$ we have $N_{\f}(\sigma_i)\cap N_{\f}(\sigma_j) =\emptyset$
        \item  for every $\sigma_i$ and every $f_\ell\in N_{\f}(\sigma_i)\smallsetminus \sigma_i$ we have $N_{\f}(f_\ell)\not\subseteq N_{\f}(\sigma_i)$.
    \end{itemize}
    Then $\V_{\f}$ is full. 
       
\end{lemma}
\begin{proof}
    To show $\V_{\f}$ is full we will show that $v_\theta$ is an isolated vertex in $\T_{\f}$. The first condition imposes that this vertex is homotopically isolated, so we only need to consider differential edges. The second condition implies that every $f_i\in \theta$ has $\lvert N_{\f}(f_i)\cap\theta\rvert = 1$. 
    Hence $v_\theta$ is not the source of any differential edge by Proposition \ref{proposition homotopically/differentially isolated}. 
    Together, all of these conditions imply that every $f_\ell\in [n]\smallsetminus \theta$ has a neighbor $f_j$ also not in $\theta$. 
    Since $N_{\f}(f_\ell)\cap \theta \subseteq \sigma_i$ for some unique $i$, $\dd_{\theta,\ell}$ exists if and only if $\dd_{\sigma_i,\ell}$ exists. By assumption, $f_j\notin N_{\f}(\sigma_i)$. 
    It follows that the variable witnessing the edge $\{f_\ell, f_i\}$ of $\G_{\f}$ prevents the differential edge $\dd_{\sigma_i,\ell}$.
\end{proof}

The following graphs constitute type F$_3$. These satisfy the hypothesis of Lemma \ref{lemma edges and triangles}, which requires counting edges $\sigma$ in $\G_{\f}$  that satisfy $N_{\f}(\sigma)=[n]$. For convenience, we display these edges as ``bold".

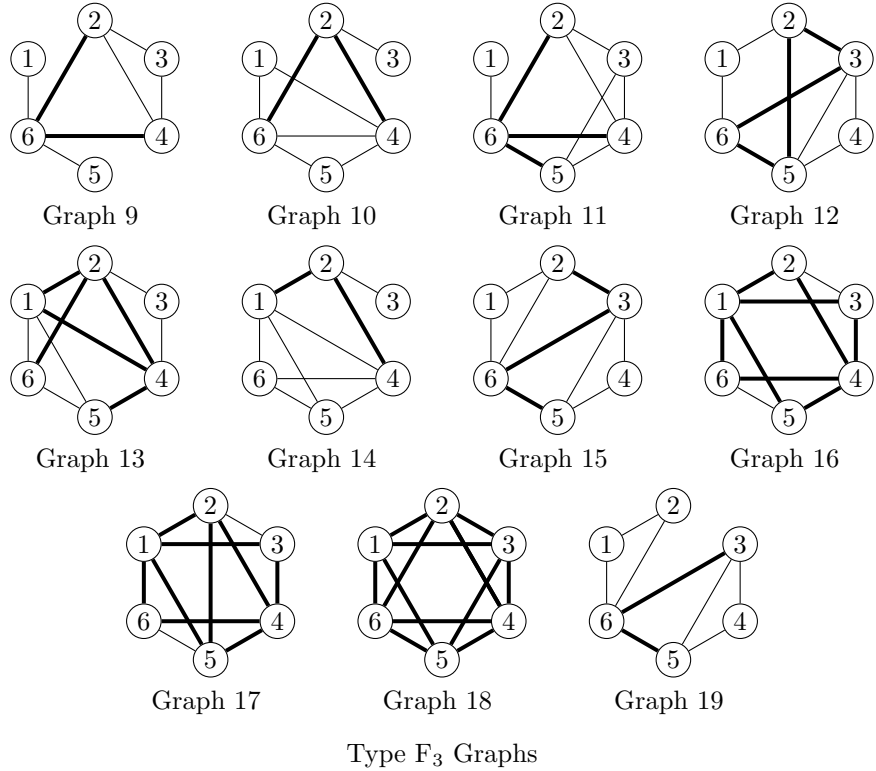
\begin{center}
    \normalfont
    \noindent %graph 8 in Stephen's
\begin{minipage}{0.19\textwidth}
\centering

\begin{tikzpicture}[x=.3\textwidth,y=.3\textwidth,scale=1, every node/.style={circle, inner sep=0pt, minimum size=13pt, draw}]

  % Nodes
  \node (v1) at (-1,1.155)  {1};
  \node (v2) at (1,1.155) {3};
  \node (v3) at (0,1.73) {2};
  \node (v4) at (1,0) {4};
  \node (v5) at (-1,0) {6};
  \node (v6) at (0,-.577) {5};

  % Edges
  \draw (v1) -- (v5);
  \draw (v2) -- (v3);
  \draw (v2) -- (v4);
  \draw (v3) -- (v4);
  \draw[line width = 1.5pt] (v3) -- (v5);
  \draw[line width = 1.5pt] (v4) -- (v5);
  \draw (v5) -- (v6);
 
\end{tikzpicture}
Graph 9 \ognum{5}
\end{minipage}
\vspace{5pt}
%graph 18 in Stephen's
\begin{minipage}{0.19\textwidth}
\centering
\begin{tikzpicture}[x=.3\textwidth,y=.3\textwidth,scale=1, every node/.style={circle, inner sep=0pt, minimum size=13pt, draw}]

  % Nodes
  \node (v1) at (-1,1.155) {1};
  \node (v2) at (0,-.577) {5};
  \node (v3) at (0,1.73) {2};
  \node (v4) at (1,0) {4};
  \node (v5) at (-1,0) {6};
  \node (v6) at (1,1.155) {3};

  % Edges
  \draw (v1) -- (v4);
  \draw (v1) -- (v5);
  \draw (v2) -- (v4);
  \draw (v2) -- (v5);
  \draw[line width = 1.5pt] (v3) -- (v4);
  \draw[line width = 1.5pt] (v3) -- (v5);
  \draw (v3) -- (v6);
  \draw (v4) -- (v5);
 
\end{tikzpicture}
Graph 10 \ognum{18}
\end{minipage}
%graph 24 in Stephen's
\begin{minipage}{0.19\textwidth}
\centering

\begin{tikzpicture}[x=.3\textwidth,y=.3\textwidth,scale=1, every node/.style={circle, inner sep=0pt, minimum size=13pt, draw}]

  % Nodes
  \node (v1) at (1,1.155) {3};
  \node (v2) at (0,-.577) {5};
  \node (v3) at (0,1.73) {2};
  \node (v4) at (1,0) {4};
  \node (v5) at (-1,0) {6};
  \node (v6) at (-1,1.155) {1};

  % Edges
  \draw (v1) -- (v3);
  \draw (v1) -- (v4);
  \draw (v1) -- (v2);
  \draw (v2) -- (v4);
  \draw[line width = 1.5pt] (v2) -- (v5);
  \draw (v3) -- (v4);
  \draw[line width = 1.5pt] (v3) -- (v5);
  \draw[line width = 1.5pt] (v4) -- (v5);
  \draw (v5) -- (v6);
 
\end{tikzpicture}
Graph 11 \ognum{24}
\end{minipage}
%graph 27 in Stephen's
\begin{minipage}{0.19\textwidth}
\centering

\begin{tikzpicture}[x=.3\textwidth,y=.3\textwidth,scale=1, every node/.style={circle, inner sep=0pt, minimum size=13pt, draw}]

  % Nodes
  \node (v1) at (-1,1.155) {1};
  \node (v2) at (-1,0) {6};
  \node (v3) at (0,1.73) {2};
  \node (v4) at (1,1.155) {3};
  \node (v5) at (0,-.577) {5};
  \node (v6) at (1,0) {4};

  % Edges
  \draw (v1) -- (v2);
  \draw (v1) -- (v3);
  \draw[line width = 1.5pt] (v2) -- (v4);
  \draw[line width = 1.5pt] (v2) -- (v5);
  \draw[line width = 1.5pt] (v3) -- (v4);
  \draw[line width = 1.5pt] (v3) -- (v5);
  \draw (v4) -- (v5);
  \draw (v4) -- (v6);
  \draw (v5) -- (v6);
 
\end{tikzpicture}
Graph 12 \ognum{27}
\end{minipage}
%graph 30 in Stephen's

\begin{minipage}{0.19\textwidth}
\centering

\begin{tikzpicture}[x=.3\textwidth,y=.3\textwidth,scale=1, every node/.style={circle, inner sep=0pt, minimum size=13pt, draw}]

  % Nodes
  \node (v1) at (-1,0) {6};
  \node (v2) at (0,-.577) {5};
  \node (v3) at (-1,1.155) {1};
  \node (v4) at (1,0) {4};
   \node (v5) at (0,1.73) {2};
  \node (v6) at (1,1.155) {3};

  % Edges
  \draw (v1) -- (v2);
  \draw (v1) -- (v3);
  \draw[line width = 1.5pt] (v1) -- (v5);
  \draw (v2) -- (v3);
  \draw[line width = 1.5pt] (v2) -- (v4);
  \draw[line width = 1.5pt] (v3) -- (v4);
  \draw[line width = 1.5pt] (v3) -- (v5);
  \draw[line width = 1.5pt] (v4) -- (v5);
  \draw (v4) -- (v6);
  \draw (v5) -- (v6);
 
\end{tikzpicture}
Graph 13 \ognum{30}
\end{minipage}
%graph 31 in Stephen's
\begin{minipage}{0.19\textwidth}
\centering

\begin{tikzpicture}[x=.3\textwidth,y=.3\textwidth,scale=1, every node/.style={circle, inner sep=0pt, minimum size=13pt, draw}]

  % Nodes
  \node (v1) at (-1,0) {6};
  \node (v2) at (0,-.577) {5};
  \node (v3) at (0,1.73) {2};
  \node (v4) at (1,0) {4};
  \node (v5) at (-1,1.155) {1};
  \node (v6) at (1,1.155) {3};

  % Edges
  \draw (v1) -- (v2);
  \draw (v1) -- (v4);
  \draw (v1) -- (v5);
  \draw (v2) -- (v4);
  \draw (v2) -- (v5);
  \draw[line width = 1.5pt] (v3) -- (v4);
  \draw[line width = 1.5pt] (v3) -- (v5);
  \draw (v3) -- (v6);
  \draw (v4) -- (v5);
 
\end{tikzpicture}
Graph 14 \ognum{31}
\end{minipage}
\vspace{5pt}
%graph 32 in Stephen's
\begin{minipage}{0.19\textwidth}
\centering
\begin{tikzpicture}[x=.3\textwidth,y=.3\textwidth,scale=1, every node/.style={circle, inner sep=0pt, minimum size=13pt, draw}]

  % Nodes
  \node (v1) at (-1,1.155) {1};
  \node (v2) at (0,1.73) {2};
  \node (v3) at (0,-.577) {5};
  \node (v4) at (1,1.155) {3};
  \node (v5) at (-1,0) {6};
  \node (v6) at (1,0) {4};

  % Edges
  \draw (v1) -- (v2);
  \draw (v1) -- (v5);
  \draw[line width = 1.5pt] (v2) -- (v4);
  \draw (v2) -- (v5);
  \draw (v3) -- (v4);
  \draw[line width = 1.5pt] (v3) -- (v5);
  \draw (v3) -- (v6);
  \draw[line width = 1.5pt] (v4) -- (v5);
  \draw (v4) -- (v6);
 
\end{tikzpicture}
Graph 15 \ognum{32}
\end{minipage}
%graph 33 in Stephen's
\begin{minipage}{0.19\textwidth}
\centering

\begin{tikzpicture}[x=.3\textwidth,y=.3\textwidth,scale=1, every node/.style={circle, inner sep=0pt, minimum size=13pt, draw}]

  % Nodes
  \node (v1) at (-1,0) {6};
  \node (v2) at (0,-.577) {5};
  \node (v3) at (0,1.73) {2};
  \node (v4) at (1,0) {4};
  \node (v5) at (-1,1.155) {1};
  \node (v6) at (1,1.155) {3};

  % Edges
  \draw (v1) -- (v2);
  \draw[line width = 1.5pt] (v1) -- (v4);
  \draw[line width = 1.5pt] (v1) -- (v5);
  \draw[line width = 1.5pt] (v2) -- (v4);
  \draw[line width = 1.5pt] (v2) -- (v5);
  \draw[line width = 1.5pt] (v3) -- (v4);
  \draw[line width = 1.5pt] (v3) -- (v5);
  \draw (v3) -- (v6);
  \draw[line width = 1.5pt] (v4) -- (v6);
  \draw[line width = 1.5pt] (v5) -- (v6);

\end{tikzpicture}
Graph 16 \ognum{33}
\end{minipage}

%graph 35 in Stephen's
\begin{minipage}{0.19\textwidth}
\centering

\begin{tikzpicture}[x=.3\textwidth,y=.3\textwidth,scale=1, every node/.style={circle, inner sep=0pt, minimum size=13pt, draw}]

  % Nodes
  \node (v1) at (-1,0) {6};
  \node (v2) at (0,-.577) {5};
  \node (v3) at (0,1.73) {2};
  \node (v4) at (1,0) {4};
  \node (v5) at (-1,1.155) {1};
  \node (v6) at (1,1.155) {3};

  % Edges
  \draw (v1) -- (v2);
  \draw[line width = 1.5pt] (v1) -- (v4);
  \draw[line width = 1.5pt] (v1) -- (v5);
  \draw[line width = 1.5pt] (v2) -- (v3);
  \draw[line width = 1.5pt] (v2) -- (v4);
  \draw[line width = 1.5pt] (v2) -- (v5);
  \draw[line width = 1.5pt] (v3) -- (v4);
  \draw[line width = 1.5pt] (v3) -- (v5);
  \draw (v3) -- (v6);
  \draw[line width = 1.5pt] (v4) -- (v6);
  \draw[line width = 1.5pt] (v5) -- (v6);
 
\end{tikzpicture}
Graph 17 \ognum{35}
\end{minipage}
%graph 37 in Stephen's
\begin{minipage}{0.19\textwidth}
\centering
\begin{tikzpicture}[x=.3\textwidth,y=.3\textwidth,scale=1, every node/.style={circle, inner sep=0pt, minimum size=13pt, draw}]

  % Nodes
  \node (v1) at (-1,0) {6};
  \node (v2) at (0,-.577) {5};
  \node (v3) at (0,1.73) {2};
  \node (v4) at (1,0) {4};
  \node (v5) at (-1,1.155) {1};
  \node (v6) at (1,1.155) {3};

  % Edges
  \draw[line width = 1.5pt] (v1) -- (v2);
  \draw[line width = 1.5pt] (v1) -- (v3);
  \draw[line width = 1.5pt] (v1) -- (v4);
  \draw[line width = 1.5pt] (v1) -- (v5);
  \draw[line width = 1.5pt] (v2) -- (v4);
  \draw[line width = 1.5pt] (v2) -- (v5);
  \draw[line width = 1.5pt] (v2) -- (v6);
  \draw[line width = 1.5pt] (v3) -- (v4);
  \draw[line width = 1.5pt] (v3) -- (v4);
  \draw[line width = 1.5pt] (v3) -- (v5);
  \draw[line width = 1.5pt] (v3) -- (v6);
  \draw[line width = 1.5pt] (v4) -- (v6);
  \draw[line width = 1.5pt] (v5) -- (v6);

\end{tikzpicture}
Graph 18 \ognum{37}
\end{minipage}
%graph 39 in Stephen's
\begin{minipage}{0.19\textwidth}
\centering
\begin{tikzpicture}[x=.3\textwidth,y=.3\textwidth,scale=1, every node/.style={circle, inner sep=0pt, minimum size=13pt, draw}]

  % Nodes
  \node (v1) at (-1,1.155) {1};
  \node (v2) at (1,0) {4};
  \node (v3) at (0,-.577) {5};
  \node (v4) at (1,1.155) {3};
  \node (v5) at (-1,0) {6};
  \node (v6) at (0,1.73) {2};

  % Edges
  \draw (v1) -- (v5);
  \draw (v1) -- (v6);
  \draw (v2) -- (v3);
  \draw (v2) -- (v4);
  \draw (v3) -- (v4);
  \draw[line width = 1.5pt] (v3) -- (v5);
  \draw[line width = 1.5pt] (v4) -- (v5);
  \draw (v5) -- (v6);
 
\end{tikzpicture}
Graph 19 \ognum{39}
\end{minipage}
\\
\vspace{10pt} 
Type F$_3$ Graphs 
\end{center}
\begin{lemma} \label{lemma edges and triangles}
    Let us consider two disjoint subsets of vertices of $\T_{\f}$ given by
    \begin{itemize}
        \item $E = \setbuild{v_\sigma}{\sigma \ \text{is an edge in} \ \G_{\f} \ \text{and } v_\sigma \text{ is homotopically isolated} }$
        \item $T = \setbuild{v_\theta}{\theta \ \text{is a triangle in} \ \G_{\f} \ \text{and } v_\theta \text{ is homotopically isolated} }.$
        
    \end{itemize}
    If $|E| > |T|$, then $\V_{\f}$ is full. 
\end{lemma}
\begin{proof}
 
    For all $v_\sigma \in E$,  Proposition \ref{proposition homotopically/differentially isolated}, along with being homotopically isolated gives that $v_\sigma$ is either a sink or an isolated vertex. If $v_\sigma$ is isolated then $\V_{\f}$ is full, so we will consider the case when every element of $E$ is a sink with some incoming differential edge. As a consequence of Proposition \ref{proposition homotopically/differentially isolated} it must be that $\sigma\cup\{f_i\}$ is a triangle in $\G_{\f}$ so we can apply Lemma \ref{lemma more sinks than sources} noting that $N\subseteq T$.
    
\end{proof}

The following graphs constitute type F$_4$. The proofs we provide are not encompassed by any lemmas. We give a specific argument for each of these graphs yielding a full support variety in Theorem \ref{theorem F for full support}
\begin{center}
\normalfont
\begin{minipage}{0.19\textwidth}
\centering

\begin{tikzpicture}[x=.3\textwidth,y=.3\textwidth,scale=1, every node/.style={circle, inner sep=0pt, minimum size=13pt, draw}]

  % Nodes
  \node (v1) at (-1,1.155) {1};
  \node (v2) at (1,1.155) {2};
  \node (v3) at (1,0) {3};
  \node (v4) at (0,1.73) {4};
  \node (v5) at (-1,0) {5};
  \node (v6) at (0,-.577) {6};

  % Edges
  \draw (v1) -- (v5);
  \draw (v1) -- (v4);
  \draw (v2) -- (v3);
  \draw (v2) -- (v4);
  \draw (v3) -- (v4);
  \draw (v3) -- (v5);
  \draw (v4) -- (v5);
  \draw (v5) -- (v6);
 
\end{tikzpicture}
Graph 20 \ognum{11}
\end{minipage}
%graph 21 in Stephen's
%graph 13 in Stephen's
\begin{minipage}{0.19\textwidth}
\centering

\begin{tikzpicture}[x=.3\textwidth,y=.3\textwidth,scale=1, every node/.style={circle, inner sep=0pt, minimum size=13pt, draw}]

  % Nodes
  \node (v1) at (1,1.155) {1};
  \node (v2) at (0,-.577) {2};
  \node (v3) at (0,1.73) {3};
  \node (v4) at (1,0) {4};
  \node (v5) at (-1,0) {5};
  \node (v6) at (-1,1.155) {6};

  % Edges
  \draw (v1) -- (v3);
  \draw (v1) -- (v4);
  \draw (v2) -- (v3);
  \draw (v2) -- (v4);
  \draw (v2) -- (v5);
  \draw (v3) -- (v5);
  \draw (v4) -- (v5);
  \draw (v5) -- (v6);
 
\end{tikzpicture}
Graph 21 \ognum{13}
\end{minipage}
\begin{minipage}{0.19\textwidth}
\centering

\begin{tikzpicture}[x=.3\textwidth,y=.3\textwidth,scale=1, every node/.style={circle, inner sep=0pt, minimum size=13pt, draw}]

  % Nodes
  \node (v1) at (-1,1.155) {1};
  \node (v2) at (0,-.577) {2};
  \node (v3) at (0,1.73) {3};
  \node (v4) at (1,0) {4};
  \node (v5) at (-1,0) {5};
  \node (v6) at (1,1.155) {6};
  
  % Edges
  \draw (v1) -- (v2);
  \draw (v1) -- (v5);
  \draw (v2) -- (v5);
  \draw (v3) -- (v4);
  \draw (v3) -- (v5);
  \draw (v4) -- (v5);
  \draw (v4) -- (v6);
 
\end{tikzpicture}
Graph 22 \ognum{21}
\end{minipage}
%%%%%%%%%%%%%%%%%%%%%%%%%%%%%%%
%{Cycle from source}
%graph 29 in Stephen's
\begin{minipage}{0.19\textwidth}
\centering
\begin{tikzpicture}[x=.3\textwidth,y=.3\textwidth,scale=1, every node/.style={circle, inner sep=0pt, minimum size=13pt, draw}]

  % Nodes
  \node (v1) at (-1,0) {1};
  \node (v2) at (0,-.577) {2};
  \node (v3) at (-1,1.155) {3};
  \node (v4) at (1,0) {4};
  \node (v5) at (0,1.73) {5};
  \node (v6) at (1,1.155) {6};

  % Edges
  \draw (v6) -- (v5);
  \draw (v5) -- (v3);
  \draw (v3) -- (v1);
  \draw (v1) -- (v2);
  \draw (v2) -- (v4);
  \draw (v4) -- (v6);
  \draw (v4) -- (v5);
  \draw (v2) -- (v3);
  \draw (v4) -- (v3);
  \draw (v2) -- (v5);
  
\end{tikzpicture}
Graph 23 \ognum{29}
\end{minipage}

%graph 28 in Stephen's
\begin{minipage}{0.19\textwidth}
\centering
\begin{tikzpicture}[x=.3\textwidth,y=.3\textwidth,scale=1, every node/.style={circle, inner sep=0pt, minimum size=13pt, draw}]

  % Nodes
  \node (v1) at (1,1.155) {1};
  \node (v2) at (0,-.577) {2};
  \node (v3) at (0,1.73) {3};
  \node (v4) at (1,0) {4};
  \node (v5) at (-1,0) {5};
  \node (v6) at (-1,1.155) {6};

  % Edges
  \draw (v6) -- (v5);
  \draw (v5) -- (v2);
  \draw (v2) -- (v1);
  \draw (v1) -- (v3);
  \draw (v3) -- (v5);
  \draw (v5) -- (v4);
  \draw (v4) -- (v1);
  \draw (v2) -- (v4);
  \draw (v2) -- (v3);
  \draw (v4) -- (v3);
 
\end{tikzpicture}
Graph 24 \ognum{28}
\end{minipage}
%graph 34 in Stephen's
\begin{minipage}{0.19\textwidth}
\centering
\begin{tikzpicture}[x=.3\textwidth,y=.3\textwidth,scale=1, every node/.style={circle, inner sep=0pt, minimum size=13pt, draw}]

  % Nodes
  \node (v1) at (0,-.577) {1};
  \node (v2) at (-1,1.155) {2};
  \node (v3) at (0,1.73) {3};
  \node (v4) at (1,0) {4};
  \node (v5) at (-1,0) {5};
  \node (v6) at (1,1.155) {6};

  % Edges
  \draw (v1) -- (v2);
  \draw (v2) -- (v3);
  \draw (v3) -- (v6);
  \draw (v6) -- (v4);
  \draw (v4) -- (v1);
  \draw (v1) -- (v5);
  \draw (v5) -- (v3);
  \draw (v2) -- (v5);
  \draw (v5) -- (v4);
 
\end{tikzpicture}
Graph 25 \ognum{34}
\end{minipage}
%graph 40 in Stephen's
\begin{minipage}{0.19\textwidth}
\centering
\begin{tikzpicture}[x=.3\textwidth,y=.3\textwidth,scale=1, every node/.style={circle, inner sep=0pt, minimum size=13pt, draw}]

  % Nodes
  \node (v1) at (1,1.155) {1};
  \node (v2) at (0,-.577) {2};
  \node (v3) at (0,1.73) {3};
  \node (v4) at (1,0) {4};
  \node (v5) at (-1,0) {5};
  \node (v6) at (-1,1.155) {6};

  % Edges
  \draw (v6) -- (v5);
  \draw (v5) -- (v2);
  \draw (v2) -- (v4);
  \draw (v4) -- (v1);
  \draw (v1) -- (v6);
  \draw (v5) -- (v3);
  \draw (v3) -- (v4);
  \draw (v2) -- (v3);
 
\end{tikzpicture}
Graph 26 \ognum{40}
\end{minipage}
\\
\vspace{10pt} 
Type F$_4$ Graphs
\end{center}

\begin{theorem} \label{theorem F for full support}
    If $\f$ generates a monomial ideal such that $\G_{\f}$ is type F, then $\V_{\f} = \A_k^{6}$.  
\end{theorem}

\begin{proof}
With Lemmas \ref{lemma isolated in degree 3}, \ref{lemma pair of edges with no common edge}, and \ref{lemma edges and triangles} in hand, it suffices to show type F$_4$ graphs yield full support varieties.

\noindent \underline{\textbf{Graph 20}:} Note that vertices $v_{35}$ and $v_{45}$ are homotopically isolated and are not the sources of differential edges. The possible differential edges into these sinks are $\dd_{35,4}$, $\dd_{45,1}$, and $\dd_{45,3}$. If $\dd_{35,4}$ exists then $x_{14}$ is not present, so by Lemma \ref{lemma variables present for degree 2 vertex} $x_1$ must be present. 
Consequently, $\dd_{45,1}$ does not exist. 
It follows that we may apply Lemma \ref{lemma more sinks than sources} to yield full support. On the other hand, if $\dd_{35,4}$ does not exist then $v_{35}$ is isolated.

\noindent \underline{\textbf{Graph 21}:} 
We again have that $v_{35}$ and $v_{45}$ are  homotopically isolated sinks. 
The only possible differential edges with these as targets are $\dd_{35,2}$ and $\dd_{45,2}$. 
If both edges exist then it must be $f_2\mid x_{25}x_{235}x_{245}$, but then $f_2\mid f_5$. 
It must be that at least one of
$v_{35}$ and $v_{45}$ is an isolated vertex.

\noindent \underline{\textbf{Graph 22}:} We consider the homotopically isolated vertices $v_{1246}$ and $v_{45}$. It follows from Proposition \ref{proposition homotopically/differentially isolated} that both vertices are sinks, with possible incoming edges $\dd_{1246,5}$ and $\dd_{45,3}$ respectively. Using Lemma \ref{lemma variables present for degree 2 vertex} at vertex $f_3$ precludes both from occurring simultaneously as at least one of $x_3$ and $x_{35}$ is present. Hence one of our sinks is an isolated vertex.

\noindent\underline{\textbf{Graph 23\ognum{29}}:}
We first observe that $Z = \spank\{v_{24}, v_{25}, v_{34},v_{35}\}$ is a 4-dimensional subspace of $\ker(\T_{\f}(a))$ as each of these vertices are homotopically isolated and hence are sinks. 
Set $W $ to be the $k$-span of $ \{v_{234}, v_{235},$ $ v_{245}, v_{345}\}$ and note by Proposition \ref{proposition homotopically/differentially isolated} we have that  $\T_{\f}(a)^{-1}(Z)\subseteq W$. Below we display all possible edges of $\T_{\f}$ that have targets in $Z$ as well as all edges of $\T_{\f}$ with sources in $W\cup\{v_{2345}\}.$
% https://q.uiver.app/#q=WzAsOSxbMiwwLCJ2X3syNH0iXSxbMiwyLCJ2X3syNX0iXSxbMCwyLCJ2X3szNX0iXSxbMCwwLCJ2X3szNH0iXSxbMSwwLCJ2X3syMzR9Il0sWzIsMSwidl97MjQ1fSJdLFsxLDIsInZfezIzNX0iXSxbMCwxLCJ2X3szNDV9Il0sWzEsMSwidl97MjM0NX0iXSxbOCw0XSxbOCw1XSxbOCw2XSxbOCw3XSxbNywzXSxbNCwzXSxbNywyXSxbNiwyXSxbNiwxXSxbNSwxXSxbNCwwXSxbNSwwXV0=
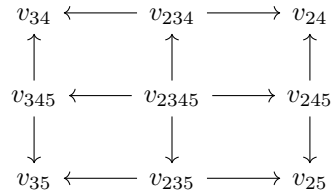
\[\begin{tikzcd}
	{v_{34}} & {v_{234}} & {v_{24}} \\
	{v_{345}} & {v_{2345}} & {v_{245}} \\
	{v_{35}} & {v_{235}} & {v_{25}}
	\arrow[from=1-2, to=1-1]
	\arrow[from=1-2, to=1-3]
	\arrow[from=2-1, to=1-1]
	\arrow[from=2-1, to=3-1]
	\arrow[from=2-2, to=1-2]
	\arrow[from=2-2, to=2-1]
	\arrow[from=2-2, to=2-3]
	\arrow[from=2-2, to=3-2]
	\arrow[from=2-3, to=1-3]
	\arrow[from=2-3, to=3-3]
	\arrow[from=3-2, to=3-1]
	\arrow[from=3-2, to=3-3]
\end{tikzcd}\]
Let $T$ be the matrix given by the restriction of $\T_{\f}(a)$ to columns in $W$. We will show that $T$ has rank at most 3, hence $Z$ contains a cycle that is not boundary. Observe that $w=\T_{\f}(a)(v_{2345})\in \ker(T)$. If $w \neq 0$ then we are done by the rank-nullity theorem. Otherwise it must be that none of the differential edges with source $v_{2345}$ exist, but then by Proposition \ref{proposition taylor edges down/upward closed} none of the edges in the diagram above exist. In this case then $T$ has rank 0.

\noindent \underline{\textbf{Graph 24\ognum{28}}:} We provide a similar argument as given for graph 23. Consider $Z = \spank\{v_{25},v_{35},v_{45}\}$, a 3-dimensional subspace of $\ker(\T_{\f}(a))$. Moreover, we take $W = \spank\{v_{235},v_{245},v_{345}\}$ and note just as before, $\T_{\f}(a)^{-1}(Z) \subseteq W.$
Below, we display all possible edges of $\T_{\f}$ with target in $Z$ or with source in $W\cup \{v_{2345}\}$.

% https://q.uiver.app/#q=WzAsNyxbMSwxLCJ2X3syMzQ1fSJdLFsxLDAsInZfezI1fSJdLFsyLDIsInZfezQ1fSJdLFswLDIsInZfezM1fSJdLFswLDAsInZfezIzNX0iXSxbMiwwLCJ2X3syNDV9Il0sWzEsMiwidl97MzQ1fSJdLFs0LDFdLFs0LDNdLFs1LDJdLFs1LDFdLFs2LDJdLFs2LDNdLFswLDRdLFswLDVdLFswLDZdXQ==
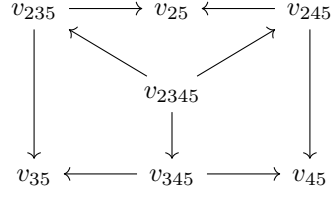
\[\begin{tikzcd}
	{v_{235}} & {v_{25}} & {v_{245}} \\
	& {v_{2345}} \\
	{v_{35}} & {v_{345}} & {v_{45}}
	\arrow[from=1-1, to=1-2]
	\arrow[from=1-1, to=3-1]
	\arrow[from=1-3, to=1-2]
	\arrow[from=1-3, to=3-3]
	\arrow[from=2-2, to=1-1]
	\arrow[from=2-2, to=1-3]
	\arrow[from=2-2, to=3-2]
	\arrow[from=3-2, to=3-1]
	\arrow[from=3-2, to=3-3]
\end{tikzcd}\]

Let $T$ be the matrix given by the restriction of $\T_{\f}(a)$ to columns of $W$. We will show that $T$ has rank at most 2, hence $Z$ contains a cycle that is not boundary. Observe that $w=\T_{\f}(a)(v_{2345})\in \ker(T)$. If $w \neq 0$ then we are done by the rank-nullity theorem. Otherwise it must be that none of the differential edges with source $v_{2345}$ exist, but then by Proposition \ref{proposition taylor edges down/upward closed} none of the edges in the diagram above exist. In this case then $T$ has rank 0.

\noindent \underline{\textbf{Graph 25\ognum{34}}:} We first observe that $v_{236}$ is homotopically isolated. No differential edge can have target $v_{236}$ as the only candidate $\dd_{236,5}$ would require all of the variables witnessing the edge $\{4,5\}$ to be absent. Note the symmetry given by the permutation $(12)(34)$. It follows that both $v_{236}$ and $v_{146}$ are sources with degree at most one. If either of $\dd_{26,3}$ and $\dd_{16,4}$ do not exist we have an isolated vertex, however if both exist we have the following subgraph of $\T_{\f}$. 
% https://q.uiver.app/#q=WzAsNyxbMCwxLCJ2XzYiXSxbMiwxLCJ2X3syNn0iXSxbMiwwLCJ2X3sxNn0iXSxbMiwyLCJ2X3s1Nn0iXSxbNCwyLCJ2X3s1fSJdLFs0LDEsInZfezIzNn0iXSxbNCwwLCJ2X3sxNDZ9Il0sWzAsMywiXFxjaGlfNSIsMV0sWzAsMSwiXFxjaGlfMiIsMV0sWzAsMiwiXFxjaGlfMSIsMV0sWzQsMywiLVxcY2hpXzYiLDFdLFs1LDEsIi0xIiwxXSxbNiwyLCItMSIsMV1d
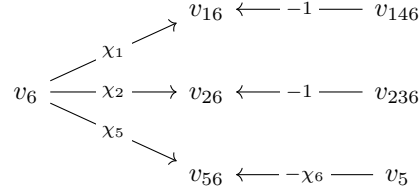
\[\begin{tikzcd}
	&& {v_{16}} && {v_{146}} \\
	{v_6} && {v_{26}} && {v_{236}} \\
	&& {v_{56}} && {v_{5}}
	\arrow["{-1}"{description}, from=1-5, to=1-3]
	\arrow["{\chi_1}"{description}, from=2-1, to=1-3]
	\arrow["{\chi_2}"{description}, from=2-1, to=2-3]
	\arrow["{\chi_5}"{description}, from=2-1, to=3-3]
	\arrow["{-1}"{description}, from=2-5, to=2-3]
	\arrow["{-\chi_6}"{description}, from=3-5, to=3-3]
\end{tikzcd}\]

Note that every edge with a source in $\{v_5,v_6,v_{146},v_{236}\}$ is pictured.
Hence for all $a\in \A_k^{6}$, we have the cycle $z = a_6v_6 + a_2a_6v_{236} + a_5v_5 + a_1a_6v_{146}$. 
Let $U = \A_k^{6}\smallsetminus V(\chi_1\chi_2\chi_5\chi_6)$ and note that for all $a\in U,$ we have that $z$ is a nonzero cycle that is not a boundary (as it has a nonzero $v_{146}$ component).
Hence $U\subseteq \V_{\f}$, and moreover, by \cite[2.5]{BGP} we can assume $k$ is algebraically closed so $U$ is dense in $\A_k^{6}$ and hence $\V_{\f} = \A_k^{6}$.

\noindent \underline{\textbf{Graph 26\ognum{40}}:} We provide a similar argument as the one given for graph 25. Both of $v_{146}$ and $v_{156}$ are sources with with degree at most one. If both $\dd_{46,1}$ and $\dd_{15,6}$ exist then we have the following subgraph of $\T_{\f}$ (which we note satisfies the hypothesis of Lemma \ref{lemma odd length path in taylor graph}). 

% https://q.uiver.app/#q=WzAsNyxbNCwwLCJ2X3sxNDZ9Il0sWzQsMiwidl97MTU2fSJdLFsyLDAsInZfezQ2fSJdLFsyLDIsInZfezE1fSJdLFswLDAsInZfNCJdLFswLDIsInZfNSJdLFsyLDEsInZfezQ1fSJdLFswLDIsIjEiLDFdLFsxLDMsIjEiLDFdLFs0LDIsIi1cXGNoaV82IiwxXSxbNSwzLCJcXGNoaV8xIiwxXSxbNSw2LCJcXGNoaV80IiwxXSxbNCw2LCItXFxjaGlfNSIsMV1d
\[\begin{tikzcd}
	{v_4} && {v_{46}} && {v_{146}} \\
	&& {v_{45}} \\
	{v_5} && {v_{15}} && {v_{156}}
	\arrow["{-\chi_6}"{description}, from=1-1, to=1-3]
	\arrow["{-\chi_5}"{description}, from=1-1, to=2-3]
	\arrow["1"{description}, from=1-5, to=1-3]
	\arrow["{\chi_4}"{description}, from=3-1, to=2-3]
	\arrow["{\chi_1}"{description}, from=3-1, to=3-3]
	\arrow["1"{description}, from=3-5, to=3-3]
\end{tikzcd}\]

Note  $v_4 $ and $v_{5}$ indeed have out-degree 2, hence hence $a_5v_4 +a_4v_5 + a_6a_5v_{146} - a_1a_4v_{156}$ is generically cycle that is not a boundary.  We conclude that $\V_{\f}$ is full. 

\end{proof}

\vspace{5pt}

We now consider graphs that are the GCD graphs of ideals with support varieties that are not full. Whether or not a support variety is realized is determined by the non-presence of various prescribed variables. Hence when displaying graphs of types A, B, and C we use dashed vertices and edges to indicate bona fide vertices and edges of the GCD graph for which the absence of the corresponding variables yields interesting support varieties. For example, ideals with GCD graph 27 below will have $\V_{\f} = V(\chi_4\chi_6)$ exactly when the variables $x_{1}$, $x_{12}$, and $x_{15}$ are not present. The following graphs constitute type B.

\begin{center}
%graph 1
\begin{minipage}{0.19\textwidth}
\centering

\begin{tikzpicture}[scale=.9, every node/.style={circle, inner sep=0pt, minimum size=13pt, draw}]

  % Nodes
  \node[dashed] (v1) at (0,1) {1};
  \node (v2) at (-1,0) {2};
  \node (v3) at (1,0) {3};
  \node (v4) at (-1,-1) {4};
  \node (v5) at (0,-1) {5};
  \node (v6) at (1,-1) {6};

  % Edges
  \draw[dashed] (v1) -- (v2);
  \draw (v1) -- (v3);
  \draw[dashed] (v1) -- (v5);
  \draw (v2) -- (v3);
  \draw (v2) -- (v4);
  \draw (v2) -- (v5);
  \draw (v3) -- (v5);
  \draw (v3) -- (v6);
  \draw (v4) -- (v5);
\end{tikzpicture}
Graph 27
\end{minipage}
%graph 3
\begin{minipage}{0.19\textwidth}
\centering
\begin{tikzpicture}[scale=.9, every node/.style={circle, inner sep=0pt, minimum size=13pt, draw}]

  % Nodes
  \node[dashed] (v1) at (0,1) {1};
  \node (v2) at (-1,0) {2};
  \node (v3) at (1,0) {3};
  \node (v4) at (-1,-1) {4};
  \node (v5) at (0,-1) {5};
  \node (v6) at (1,-1) {6};

  % Edges
  \draw[dashed] (v1) -- (v2);
  \draw (v1) -- (v3);
  \draw[dashed] (v1) -- (v5);
  %\draw (v2) -- (v3);
  \draw (v2) -- (v4);
  \draw (v2) -- (v5);
  \draw (v3) -- (v5);
  \draw (v3) -- (v6);
  \draw (v4) -- (v5);
\end{tikzpicture}
Graph 28
\end{minipage}
%graph 8
\begin{minipage}{0.19\textwidth}
\centering
\begin{tikzpicture}[scale=.9, every node/.style={circle, inner sep=0pt, minimum size=13pt, draw}]

  % Nodes
  \node[dashed] (v1) at (0,1) {1};
  \node (v2) at (-1,0) {2};
  \node (v3) at (1,0) {3};
  \node (v4) at (-1,-1) {4};
  \node (v5) at (0,-1) {5};
  \node (v6) at (1,-1) {6};

  % Edges
  \draw[dashed] (v1) -- (v2);
  \draw (v1) -- (v3);
  \draw[dashed] (v1) -- (v5);
  %\draw (v2) -- (v3);
  \draw (v2) -- (v4);
  \draw (v2) -- (v5);
  %\draw (v3) -- (v5);
  \draw (v3) -- (v6);
  \draw (v4) -- (v5);
\end{tikzpicture}
Graph 29
\end{minipage}
%graph 2
\begin{minipage}{0.19\textwidth}
\centering
\begin{tikzpicture}[scale=.9, every node/.style={circle, inner sep=0pt, minimum size=13pt, draw}]

  % Nodes
  \node[dashed] (v1) at (0,1) {1};
  \node (v2) at (-1,0) {2};
  \node (v3) at (1,0) {3};
  \node (v4) at (-1,-1) {4};
  \node (v5) at (0,-1) {5};
  \node (v6) at (1,-1) {6};

  % Edges
  \draw (v1) -- (v2);
  \draw (v1) -- (v3);
  \draw[dashed] (v1) -- (v5);
  \draw (v2) -- (v3);
  \draw (v2) -- (v4);
  \draw (v2) -- (v5);
  \draw (v3) -- (v5);
  \draw (v3) -- (v6);
\end{tikzpicture}
Graph 30
\end{minipage}
%%%%%%%%%%%%%%%%%%%%%%%%%%%%%%%%%%%%%%%%%%%%%%%%
\begin{minipage}{0.19\textwidth}
\centering
\begin{tikzpicture}[scale=.9, every node/.style={circle, inner sep=0pt, minimum size=13pt, draw}]

  % Nodes
  \node[dashed] (v1) at (0,1) {1};
  \node (v2) at (-1,0) {2};
  \node (v3) at (1,0) {3};
  \node (v4) at (-1,-1) {4};
  \node (v5) at (0,-1) {5};
  \node (v6) at (1,-1) {6};

  % Edges
  \draw (v1) -- (v2);
  \draw (v1) -- (v3);
  \draw[dashed] (v1) -- (v5);
  %\draw (v2) -- (v3);
  \draw (v2) -- (v4);
  \draw (v2) -- (v5);
  \draw (v3) -- (v5);
  \draw (v3) -- (v6);
\end{tikzpicture}
Graph 31
\end{minipage}

%graph 4
\begin{minipage}{0.19\textwidth}
\centering
\begin{tikzpicture}[scale=.9, every node/.style={circle, inner sep=0pt, minimum size=13pt, draw}]

  % Nodes
  \node[dashed] (v1) at (0,1) {1};
  \node (v2) at (-1,0) {2};
  \node (v3) at (1,0) {3};
  \node (v4) at (-1,-1) {4};
  \node (v5) at (0,-1) {5};
  \node (v6) at (1,-1) {6};

  % Edges
  \draw (v1) -- (v2);
  \draw (v1) -- (v3);
  %\draw (v1) -- (v5);
  \draw (v2) -- (v3);
  \draw (v2) -- (v4);
  \draw (v2) -- (v5);
  \draw (v3) -- (v5);
  \draw (v3) -- (v6);
\end{tikzpicture}
Graph 32
\end{minipage}
\begin{minipage}{0.19\textwidth}
\centering
\begin{tikzpicture}[scale=.9, every node/.style={circle, inner sep=0pt, minimum size=13pt, draw}]

  % Nodes
  \node[dashed] (v1) at (0,1) {1};
  \node (v2) at (-1,0) {2};
  \node (v3) at (1,0) {3};
  \node (v4) at (-1,-1) {4};
  \node (v5) at (0,-1) {5};
  \node (v6) at (1,-1) {6};

  % Edges
  \draw (v1) -- (v2);
  \draw (v1) -- (v3);
  %\draw (v1) -- (v5);
  %\draw (v2) -- (v3);
  \draw (v2) -- (v4);
  \draw (v2) -- (v5);
  \draw (v3) -- (v5);
  \draw (v3) -- (v6);
\end{tikzpicture}
Graph 33
\end{minipage}
\begin{minipage}{0.19\textwidth}
\centering
\begin{tikzpicture}[scale=.9, every node/.style={circle, inner sep=0pt, minimum size=13pt, draw}]

  % Nodes
  \node[dashed] (v1) at (0,1) {1};
  \node (v2) at (-1,0) {2};
  \node (v3) at (1,0) {3};
  \node (v4) at (-1,-1) {4};
  \node (v5) at (0,-1) {5};
  \node (v6) at (1,-1) {6};

  % Edges
  \draw (v1) -- (v2);
  \draw (v1) -- (v3);
  \draw[dashed] (v1) -- (v5);
  \draw (v2) -- (v3);
  \draw (v2) -- (v4);
  \draw (v2) -- (v5);
  %\draw (v3) -- (v5);
  \draw (v3) -- (v6);
\end{tikzpicture}
Graph 34
\end{minipage}
\begin{minipage}{0.19\textwidth}
\centering
\begin{tikzpicture}[scale=.9, every node/.style={circle, inner sep=0pt, minimum size=13pt, draw}]

  % Nodes
  \node[dashed] (v1) at (0,1) {1};
  \node (v2) at (-1,0) {2};
  \node (v3) at (1,0) {3};
  \node (v4) at (-1,-1) {4};
  \node (v5) at (0,-1) {5};
  \node (v6) at (1,-1) {6};

  % Edges
  \draw (v1) -- (v2);
  \draw (v1) -- (v3);
  \draw[dashed] (v1) -- (v5);
  %\draw (v2) -- (v3);
  \draw (v2) -- (v4);
  \draw (v2) -- (v5);
  %\draw (v3) -- (v5);
  \draw (v3) -- (v6);
\end{tikzpicture}
Graph 35
\end{minipage}
\\
\vspace{10pt} 
Type B Graphs
\end{center}

\begin{theorem}\label{theorem interesting supports 6 gen}
    If, up to relabeling, $\G_{\f}$ is Type B then $\V_{f} = V(\chi_4\chi_6)$ if all of the dashed variables are not present, otherwise $\V_{\f} = \A_k^{6}.$ 
\end{theorem}
\begin{proof}
We start by showing type B graphs have $\V_{\f} =V(\chi_4\chi_6)$ when the various prescribed variables are indeed absent. We follow the strategy outlined by our running Example \ref{running example taylor graph}, \ref{running example subvarieties}, and \ref{running example perfect matching} of Section \ref{section computing support varieties} en masse. There are a number of cases to consider, which we consolidate in the table below. 

\begin{table}[H]
\begin{tabular}{|c|c|c|c|c|c|}
\hline
Graphs                & $\chi_4$ Sink & $\chi_4$ Source & $\chi_6$ Sink & $\chi_6$ Source & Matching \\ \hline
27 ($x_{45}$ present) & $v_{346}$     &                 &               & $v_{12}$        & $M_1$    \\ \hline
27 ($x_{45}$ absent)  & $v_{346}$     &                 &               & $v_{24}$        & $M_1$    \\ \hline
28, 29                & $v_{346}$     &                 &               & $v_{15}$        & $M_1$    \\ \hline
30, 31, 32, 33        & $v_{346}$     &                 & $v_{246}$     &                 & $M_2$    \\ \hline
34, 35                &               & $v_{13}$        & $v_{246}$     &                 & $M_2$    \\ \hline
\end{tabular}
\end{table}

We first demonstrate $\{f_4, f_6\}$-perfect matchings that are triangular. For graphs 27, 28 and 29 let \[M_1 = \dd_{23,1}\otimes 2^{\{4,5,6\}} \cup \dd_{35,1}\otimes 2^{\{4,6\}}\cup \h_{3,4}\otimes 2^{\{1,6\}}\cup\h_{\emptyset,6}\otimes 2^{\{1,2,4,5\}}.\]
Recall from Lemma \ref{lemma existence of hypercubes} that we need only check that $\dd_{23,1}$, $\dd_{35,1}$, $\h_{136,4}$ and $\h_{1245,6}$ exist in $\T_{\f}$ to show that all of the edges in $M_1$ also exist. Indeed the homotopy edges exist as $f_4$ is not a neighbor of $f_1$, $f_3$, and $f_6$ and the only neighbor of $f_6$ is $f_3$. The differential edges occur precisely because $x_{1}$, $x_{12}$, and $x_{15}$ are not present. We now show $M_1$ is triangular. Lemma \ref{lemma 2 determined matchings are triangular} does not apply in to this matching, however a similar argument holds. 
In particular, $\Aux_{\f}^M$ has a 3 coloring given by $D_1$, $H_4$, and $H_6$. Just as in Lemma \ref{lemma 2 determined matchings are triangular}, no cycle can exist with just two colors. Note that a vertex of $\Aux_{\f}^M$ does not contain $f_3$ if and only if it is in $H_6$.
It follows that edges into $H_6$ must have weight $\h_6^{-1}\dd_3$, whereas edges leaving $H_6$ have weight $\e^{-1} \h_3$. 
Thus no cycle exists in $\Aux_{\f}^M$, otherwise it would traverse $\h_3\h_6^{-1}\dd_3$ which is impossible.

For graphs 30 to 35 let \[M_2 = \dd_{23,1}\otimes 2^{\{4,5,6\}}\cup \h_{3,4}\otimes 2^{\{1,5,6\}}\cup \h_{\emptyset,6}\otimes 2^{\{1,2,4,5\}}.\]
We observe that $\dd_{23,1},$ $\h_{1356,4}$, and $\h_{1245,6}$  each exist as well as the fact that $M_2$ is $\{f_2,f_3\}$-determined. Hence for graphs of type B we have $\chi_4\chi_6\in \I_{\f}$ and thus $\V_{\f} \subseteq V(\chi_4\chi_6)$ when all of the dashed variables are absent.

Recall that a differentially isolated vertex $v_\sigma$ is a homotopy sink if $[n] = \sigma \cup N_{\f}(\sigma)$ and a homotopy source if $\sigma\subseteq N_{\f}(\sigma)$. We claim that $v_{346}$ is a homotopy sink for $\chi_4$ for graphs 27-33, whereas $v_{13}$ is a homotopy source for $\chi_4$ for graphs 34 and 35. Checking that the neighborhood conditions are satisfied in each case is routine, hence it suffices to show these vertices are differentially isolated. We have that $v_{346}$ and $v_{13}$ are not the sources of any differential edges by Proposition \ref{proposition homotopically/differentially isolated}. 
Moreover, Proposition \ref{proposition homotopically/differentially isolated} implies for graphs 29, 31, and 33, the vertex $v_{346}$ cannot be the target of a differential edge. 
Also, that $v_{346}$ can only be the target of a differential edges: $\dd_2$ or $\dd_5$ for graph 27, $\dd_5$ for graph 28, and $\dd_2$ for graphs 30 and 32. 
For each of these graphs, some variable must divide $f_1$ and not $f_3$. 
In each case this variable must be one of $x_{125}$ or $x_{12}$, which accordingly prevents the differential edges of concern. 
We again use Proposition \ref{proposition homotopically/differentially isolated} to argue $v_{13}$ is differentially isolated in graph 35 and that the only possible differential edges using this vertex for graph 34 is $\dd_2$. 
However $x_{24}$ is necessarily present in this case, preventing such an edge. Thus we may conclude $V(\chi_4)\subseteq \V_{\f}$ for each $\f$ under consideration.

We now demonstrate a homotopy source/sink for $\chi_6$ in each case. First consider graph 27. If $x_{45}$ is present then $v_{12}$ is a homotopy source. The only non-routine thing to check is incoming differential edges, but $x_{45}$ prevents a $\dd_5$ edge and $x_{36}$ must be present and prevents a $\dd_3$ edge. On the other hand, if $x_{45}$ is not present then $v_{24}$ is a homotopy source. After routine checks, the only concern is a differential edge $\dd_5$, however some variable must be present that divides $f_5$ but not $f_2$. As $x_{45}$ is not present, it cannot be that $f_5\vert f_{24}$.
For graphs 28 and 29 we have that $v_{15}$ is a homotopy source. The only incoming differential edges to consider are $\dd_{3}$ (prevented by $x_{36}$) and $\dd_2$. For both graphs, a variable dividing $f_2$ but not $f_5$ cannot divide $f_1$, hence there is no $\dd_2$ into $v_{15}$. Finally, for graphs 30 to 35 we have that $v_{246}$ is a homotopy sink. This vertex is differentially isolated in each case as a variable witnessing the edge $\{f_1,f_3\}$ must be present. 
Moreover, these graphs also have $\{f_1,f_3,f_5\}\subset N_{\f}(f_2,f_4,f_6)$. 
Thus we may conclude $V(\chi_6)\subseteq \V_{\f}$ for each $\f$ under consideration. 
We have that for type B graphs, $\V_{\f} = V(\chi_4\chi_6)$ when the dashed variables are not present.

We now consider the cases when a dashed variable is present. Notably, some subtleties occur due to the ``up to relabeling" condition in conjunction with certain graph automorphisms. In particular, for graphs 27 and 29 the permutation swapping 2 and 5 respects dashed variables. This allows us to reduce to the two cases for these graphs: $x_1$ is present or $x_{15}$ is present. On the other hand graphs 30-33 have an automorphism transposing 1 and 5 that does not respect dashed variables. 
Hence for each of these, we do not consider the case when just $x_1$ is present. 
Instead, it suffices to consider when $x_1$ and $x_5$ are both present. 
In total this gives 17 cases to consider. 
In every case we use the present variable in conjunction with Proposition \ref{proposition homotopically/differentially isolated} to argue about certain vertices of $\T_{\f}$. Graphs 30, 32, and 34 encompass 5 of these cases, each of which has $v_{23}$ as an isolated vertex. 
Graphs 31, 33, and 35 encompass 5 more of these cases, each of which has $v_{2346}$ as an isolated vertex.  
The remaining 7 cases explicitly appear in the table below. 
Most of these cases admit an isolated vertex, however we apply Lemma \ref{lemma more sinks than sources} for 2 of them, listing the relevant sinks and (possible) neighbors.
\begin{table}[H]
\begin{tabular}{|c|c|c|c|c|}
\hline
Graph      & Present  & Isolated Vertex & Lemma \ref{lemma more sinks than sources} Sinks & Lemma \ref{lemma more sinks than sources} Neighbors \\ \hline
27         & $x_{1}$  &                 & $v_{23}$, $v_{35}$                              & $v_{235}$                                           \\ \hline
27         & $x_{15}$ & $v_{23}$        &                                                 &                                                     \\ \hline
28        & $x_{1}$  & $v_{35}$        &                                                 &                                                     \\ \hline
28         & $x_{12}$ & $v_{35}$      &                                                 &                                                     \\ \hline
28         & $x_{15}$ & $v_{2346}$        &                                                 &                                                     \\ \hline
29         & $x_{1}$  &                 & $v_{2346}$, $v_{3456}$                          & $v_{23456}$                                         \\ \hline
29         & $x_{15}$ & $v_{2346}$      &                                                 &                                                     \\ \hline
30, 32, 34 & any      & $v_{23}$        &                                                 &                                                     \\ \hline
31, 33, 35 & any      & $v_{2346}$      &                                                 &                                                     \\ \hline
\end{tabular}
\end{table}

\end{proof}

%%%%%%%%%%%%%%%%%%%%%%%%%%%%%%%%%%%%

\section{Some Families of Support Varieties}\label{section families}
    
In this section we consider three families of monomial ideals. 
The first two achieve an arbitrarily large codimension of $\V_{\f}$, provided that the ambient space is large enough. 
The third family yields a support variety that is the union of $n$ hyperplanes inside of $\A_k^{2n}$. 
We recall that if the GCD graph, $\G_{\f},$ is disconnected then the cohomological support variety, $\V_{\f},$ can be computed as the product of varieties corresponding to each connected component of $\G_{\f} $ \cite[Lemma 6.13]{BGP}. Thus one can easily provide examples of support varieties with arbitrarily large codimension, however, our families do not make use of this approach as they have connected GCD graphs. We start by defining these graphs, which depend on two positive parameters $a,b\in \N$.

\begin{definition}
    The \emph{double broom} graph $\DB(a,b)$ is a simple graph with $n=a+b+3$ vertices consisting of a path on three vertices $f_1,f_2,f_3$ along with $a$ vertices $g_1,\ldots,g_a$ each connected to $f_1$ and $b$ vertices $h_1,\ldots,h_b$ each connected to $f_3$. The \emph{whiskered triangle} graph $\WT(a,b)$ is given by taking $\DB(a,b)$ and adding an edge between $f_1$ and $f_3$.
\end{definition}
\begin{example}
These generalize the two GCD graphs with 5 vertices identified by \cite{BGP} to have interesting support. 
The graph $\DB(1,1)$ is the path with 5 vertices i.e. $P_5$ of Example \ref{running example taylor graph}. Similarly, $\WT(1,1)$ is graph A in that example. 
Moreover, $\WT(1,2)$ and $\DB(1,2)$ constitute graphs 36 and 37, the 2 type A graphs of Section \ref{section 6 generated ideals}. Just as before, we use dashed vertices to denote the variable required to be absent for interesting support to occur.

\begin{center}
\begin{minipage}{0.4\textwidth}
\centering
\begin{tikzpicture}[scale=.9, every node/.style={circle, inner sep=0pt, minimum size=15pt, draw}]

  % Nodes
  \node[dashed] (v1) at (0,1) {$f_2$};
  \node (v2) at (-1,0) {$f_1$};
  \node (v3) at (1,0) {$f_3$};
  \node (v4) at (-2.3,0) {$g_1$};
  \node (v5) at (2.3,1) {$h_1$};
  \node (v6) at (2.3,0) {$h_2$};

  % Edges
  \draw (v1) -- (v2);
  \draw (v1) -- (v3);
  %\draw (v1) -- (v5);
  \draw (v2) -- (v3);
  \draw (v2) -- (v4);
  %\draw (v2) -- (v5);
  \draw (v3) -- (v5);
  \draw (v3) -- (v6);
\end{tikzpicture}\\
Graph 36\\ $\WT(1,2)$
\end{minipage}
\begin{minipage}{0.4\textwidth}
\centering
\begin{tikzpicture}[scale=.9, every node/.style={circle, inner sep=0pt, minimum size=15pt, draw}]

  % Nodes
  \node[dashed] (v1) at (0,1) {$f_2$};
  \node (v2) at (-1,0) {$f_1$};
  \node (v3) at (1,0) {$f_3$};
  \node (v4) at (-2.3,0) {$g_1$};
  \node (v5) at (2.3,1) {$h_1$};
  \node (v6) at (2.3,0) {$h_2$};

  % Edges
  \draw (v1) -- (v2);
  \draw (v1) -- (v3);
  %\draw (v1) -- (v5);
  %\draw (v2) -- (v3);
  \draw (v2) -- (v4);
  %\draw (v2) -- (v5);
  \draw (v3) -- (v5);
  \draw (v3) -- (v6);
\end{tikzpicture}\\
Graph 37\\$\DB(1,2)$
\end{minipage}
\\
\vspace{10pt} 
Type A Graphs
\end{center}    
\end{example}
    
For the next part of this section we adopt the notation $[3] = \{f_1,f_2,f_3\}$, $[a] = \{g_1,\ldots,g_a\}$, and $[b] = \{h_1,\ldots,h_b\}$. We set $[n] = [3]\cup[a]\cup [b]$. 
Let us also adopt the following convention for cohomological operators: $\eta_\ell$ corresponds to $f_\ell,$ $\chi_i$ corresponds to $g_i$, and $\gamma_j$ corresponds to $h_j$. Finally, since symbols such as $v_{12}$ and $\dd_{123,4}$ are nonsensical in this context we explicitly name the generators in such subscripts throughout this section. For example $\h_{{[n] \smallsetminus \{f_3,h_j\}},h_j}$ is the homotopy edge from $v_{[n] \smallsetminus \{f_3,h_j\}}$ to $v_{[n] \smallsetminus \{f_3\}}$.

\begin{theorem}\label{theorem DB WT varieties}
    Suppose $\G_{\f}$ is $\DB(a,b)$ or $\WT(a,b)$. If $x_{f_2}$ is not present then 
    \[
    \V_{\f} = V(\chi_1,\ldots,\chi_a) \cup V(\gamma_1,\ldots,\gamma_b),
    \]
    otherwise if $x_{f_2}$ is present then $\V_{\f} = \A_k^{a+b+3}$.
\end{theorem}
\begin{proof}
    We first consider the case that $x_{f_{2}}$ is present. Every vertex of $\G_{\f}$ contained in $[a]\cup[b]$ has degree 1, hence every $x_{g_i}$ and $x_{h_j}$ must be present. This prevents any differential edges of the form $\dd_{g_i}$ and $\dd_{h_j}$.
    
    For the double broom we show $v_\sigma=v_{[n]\smallsetminus\{f_2\}}$ is an isolated vertex in $\T_{\f}$. When $\G_{\f}$ is connected, every vertex of size $n-1$ in $\T_{\f}$ is homotopically isolated. Since $\G_{\f}$ is triangle-free, $x_{f_1f_2}$ and $x_{f_2f_3}$ must be present, preventing $\dd_{\sigma,f_1}$ and  $\dd_{\sigma,f_3}$. 
    It follows that $v_\sigma$ is not the source of a differential edge. 
    The variable $x_{f_2}$ prevents the only possible edge with target $v_\sigma$ hence $v_\sigma$ is isolated. 

    For the whiskered triangle, $v_{f_1f_3}$ is an isolated vertex in $\T_{\f}$.  Every vertex of $\G_{\f}$ outside of $f_1$ and $f_3$ has its corresponding singleton variable present, hence $v_{f_1f_3}$ is not the target of a differential edge, and is thus differentially isolated. It now suffices to observe $[n] = N_{\f}(f_1,f_2)$ i.e. that $v_{f_1f_3}$ is homotopically isolated.

    Now suppose that $x_{f_2}$ is absent. We first claim that $v_{f_2f_3}$ is homotopy source for $\chi_1,\ldots,\chi_a$. Since $N_{\f}(f_2,f_3) = [3]\cup [b]$ it suffices to check that $v_{f_2f_3}$ is differentially isolated. By Proposition \ref{proposition homotopically/differentially isolated} the only possible differential edge using this vertex is $\dd_{f_2f_3,f_1}$ in the case of $\WT(a,b)$. However, the variable $x_{f_1g_1}$ is present and prevents this differential edge. 
    A symmetric argument yields that $v_{f_1f_2}$ is a homotopy source for $\gamma_1,\ldots,\gamma_b$. By Proposition \ref{proposition homotopy soure/sink gives subvariety} this shows $V(\chi_1,\ldots,\chi_a) \cup V(\gamma_1,\ldots,\gamma_b)\subseteq \V_{\f}.$

    The theorem's conclusion in this case is equivalent to $\I_{\f}$ being minimally generated by $\chi_i\gamma_j$ for all $g_i\in[a]$ and $h_j\in [b]$. 
    We fix such an $i$ and $j$ then proceed by producing a triangular $\{g_i,h_j\}$-perfect matching. By Proposition \ref{proposition det by cycle decompositions} this implies $\chi_i\gamma_j\in \I_{\f}$. This completes the proof as it gives $\V_{\f}\subseteq V(\chi_1,\ldots,\chi_a) \cup V(\gamma_1,\ldots,\gamma_b).$
    
    We first note that $\dd_{f_1f_3,f_2}$ is an edge of $\T_{\f}$ because $x_{f_2}$ is not present, so we set $D =\dd_{f_1f_3,f_2}\otimes 2^{[a]\cup[b]}$. For any $\sigma\cap \{f_1,g_i\}=\emptyset$ we have the edge $\h_{\sigma,g_i}$ so let $G_i = \h_{\emptyset,g_i}\otimes 2^{[n] \smallsetminus \{f_1,g_i\}}$. The vertices not contained by $D$ and $G_i$ are exactly those containing $f_1$ but not $f_3$, hence we take $H_j = \h_{f_1,h_j}\otimes 2^{[n] \smallsetminus \{f_1, f_3, h_j\}}$, while noting $\h_{{[n] \smallsetminus \{f_3,h_j\}},h_j}$ is indeed an edge of $\T_{\f}$. Membership of an edge using $v_\sigma$ in $D$, $G_i$, or $H_j$ can be determined by $\sigma\cap\{f_1,f_3\}$ so these sets are disjoint. Moreover, $G_i$ has $2^{a+b+1}$ edges while $D$ and $H_j$ each have $2^{a+b}$ edges, so $M=D\cup G_i\cup H_j$ has $2^{a+b+2}$ edges and is thus a $\{g_i,h_j\}$-perfect matching.
    It suffices to note that $M$ is indeed $\{f_1,f_3\}$-determined and hence triangular by Lemma \ref{lemma 2 determined matchings are triangular}. 
\end{proof}

We now consider a family of ideals inspired by Graph 38 (on page \pageref{graph 38}) along with its requisite absent variables. We think of this graph by starting with the triangle 135 and adding one more triangle along each edge, i.e. a stacked polytope. We copy this construction for simplices with larger dimension.  

\begin{definition}\label{definition stacked polytope}
For $n\geq 3$ let $\Delta(n)$ be the graph with vertex set 
\[
[n]=[f]\cup[g] = \{f_1,\ldots,f_n\}\cup \{g_1,\ldots, g_n\}
\]
and edges $\setbuild{\{f_i,f_j\}}{i\neq j}\cup \setbuild{\{f_i,g_j\}}{i\neq j}$. That is, $\Delta(n)$ is a union of a complete graph on $[f]$ and a complete graph for each $([f]\smallsetminus \{f_i\})\cup\{g_i\}$.    
\end{definition}

We will only demonstrate one ideal with $\G_{\f} = \Delta(n)$ with interesting support, though in general many exist. The fifth author plans to study ideals with this GCD graph as well as 7 and 8 generated monomial ideals in future work. We prepare the following lemma to show that the ideal we will give indeed has $\G_{\f} = \Delta(n).$

\begin{lemma}\label{lemma boundary present gives gcd graph}
Let $K$ be a simplicial complex and $\f$ be given by a set of variables corresponding to faces of $K$. If every facet $\sigma\in K$ has $\dim(\sigma) \geq 2$ and every $f_i\in \sigma$ has that $x_{\sigma\smallsetminus \{f_i\}}$ is present, then $\G_{\f}$ is the 1-skeleton of $K.$
\end{lemma}
\begin{proof}
    It suffices to check that every edge of the 1-skeleton of $K$ is witnessed and that no $f_i \;\vert\; f_j$. By assumption, every edge $\{f_i, f_j\}$ in the 1-skeleton is contained in some larger facet $\sigma$ containing $\{f_i,f_j,f_\ell\}$. Hence $x_{\sigma\smallsetminus \{f_\ell\}}$ witnesses this edge in $\G_{\f}$. Similarly, $f_i\;\vert\; f_j$ could only happen if some facet $\sigma$ contains $\{f_i,f_j,f_\ell\}$. The presence of $x_{\sigma\smallsetminus \{f_j\}}$ prevents $f_i\;\vert\; f_j$. 
\end{proof}

\begin{theorem}\label{theorem union of hyperplanes}
Let $\f$ be given by taking $X_{\f}$ to contain the variables corresponding to faces of the clique complex $K_{\Delta(n)}$ except for the faces: 
\begin{itemize}
    \item $\setbuild{{\{g_1\}\cup\sigma}}{\sigma\subseteq [f]\smallsetminus \{f_1\},\; \lvert\sigma\rvert\leq n-3}$
    \item $\setbuild{{\{g_2\}\cup\theta}}{\theta\subseteq [f]\smallsetminus \{f_1,f_2\},\; \lvert\theta\rvert\leq n-3}$
    \item $\setbuild{{\{g_3\}\cup\varphi}}{\varphi\subseteq [f]\smallsetminus \{f_1,f_2,f_3\},\; \lvert\varphi\rvert\leq n-3}.$
\end{itemize}
Then $\V_{\f} = V(\chi_1\chi_2\cdots\chi_n)$, where $\chi_i$ is associated to $g_i.$
\end{theorem}

\begin{example}\label{example Delta(4)}
   The case when $n=3$ gives GCD graph 38 (on page \pageref{graph 38}) where $[f] = \{1,3,5\}$ and $[g] = \{2,4,6\}$ and the ideal is as pictured there, i.e. the only variables not present correspond to the singletons $\{g_i\}$. This example is too small to understand behavior of $g_i$ when $i>3$. We depict $\Delta(4)$ below where the absent variables for our ideal (at largest, edges) are dashed. Variables corresponding to triangles and tetrahedra are all present. 
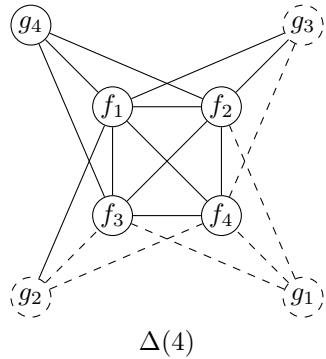
\begin{center}
\begin{minipage}{0.4\textwidth}
\centering
\begin{tikzpicture}[scale=.9, every node/.style={circle, inner sep=0pt, minimum size=15pt, draw}]

  % Nodes
  \node (v1) at (-.8,.8) {$f_1$};
  \node (v2) at (.8,.8) {$f_2$};
  \node (v3) at (-.8,-.8) {$f_3$};
  \node (v4) at (.8,-.8) {$f_4$};

  \node[dashed] (w1) at (2,-2) {$g_1$};
  \node[dashed] (w2) at (-2,-2) {$g_2$};
  \node[dashed] (w3) at (2,2) {$g_3$};
  \node (w4) at (-2,2) {$g_4$};

  % Edges
  \draw (v1) -- (v2);
  \draw (v1) -- (v3);
  \draw (v1) -- (v4);
  \draw (v2) -- (v3);
  \draw (v2) -- (v4);
  \draw (v3) -- (v4);

  \draw[dashed] (w1) -- (v2);
  \draw[dashed] (w1) -- (v3);
  \draw[dashed] (w1) -- (v4);
  \draw (w2) -- (v1);
  \draw[dashed] (w2) -- (v3);
  \draw[dashed] (w2) -- (v4);
  \draw (w3) -- (v1);
  \draw (w3) -- (v2);
  \draw[dashed] (w3) -- (v4);
  \draw (w4) -- (v1);
  \draw (w4) -- (v2);
  \draw (w4) -- (v3);
  
\end{tikzpicture}\\
$\Delta(4)$
\end{minipage}

\end{center}
   
\end{example}

\begin{proof}[Proof of Theorem \ref{theorem union of hyperplanes}]
Note that the facets of the clique complex are $[f]$ and $[f]\oplus \{f_i,g_i\}$ where $\oplus$ is the symmetric difference of sets e.g. $([f]\smallsetminus \{f_i\}) \cup \{g_i\}.$ Thus this ideal satisfies the hypothesis of Lemma \ref{lemma boundary present gives gcd graph}, so $\G_{\f} = \Delta(n)$.
In particular, $N_{\f}(f_i) = [n]\smallsetminus \{g_i\}$ and $N_{\f}(g_i) = [f]\smallsetminus \{f_i\}$. 
It follows that homotopy edges of $\T_{\f}$ are enumerated by $\h_{\emptyset,f_i} \otimes 2^{\{g_i\}}$ and $\h_{\emptyset,g_i} \otimes 2^{[g] \oplus \{g_i, f_i\}}$. Note that every vertex $v_\sigma$ used by a homotopy edge has $\lvert[f]\cap \sigma\rvert<2$.

    If there is any hope for $\V_{\f}$ to be not full, every $v_\sigma$ with  $\lvert[f]\cap \sigma\rvert\geq 2$ must be involved with some differential edge, otherwise we would have an isolated vertex. 
    Moreover, the only possible differential edges are $\dd_{g_1}$, $ \dd_{g_2}$, and $ \dd_{g_3}$ as every other generator enjoys a corresponding singleton variable. By Proposition \ref{proposition homotopically/differentially isolated}, $\lvert[f]\cap \sigma\rvert\geq 2$ is a necessary condition for $v_\sigma$ to belong to one of these differential edges. 
    In particular, we have that no homotopy edge shares a vertex with a differential edge. 
    We observe for each $i$, that each $v_{[g]\cup\{f_i\}}$ is a homotopy sink for $\chi_i$. 
    It follows that $ V(\chi_1\cdots\chi_n)\subseteq \V_{\f}$.
    
    Set $H_{g_1} = \h_{\emptyset,g_1}\otimes 2^{[g]\oplus \{g_1,f_1\}}$ and for $i>1$, $H_{g_i} = \h_{f_i,g_i}\otimes 2^{[g]\smallsetminus \{g_i\}}.$ 
    Further, let \[D_1 = \bigcup_{1<i<j}\dd_{f_if_j,g_1}\otimes 2^{[n]\smallsetminus\{f_i,f_j,g_1\}} \qquad D_2 = \bigcup_{2<i}\dd_{f_1f_i,g_2}\otimes 2^{[g]\smallsetminus \{g_2\}}\qquad D_3 = \dd_{f_1f_2,g_3}\otimes 2^{[g]\smallsetminus\{g_3\}}.\]

    The existence of the edges in $D_i$ can be checked using Lemma \ref{lemma existence of hypercubes}, so it suffices to show $\dd_{f_if_j,g_1}$, $\dd_{f_1f_i,g_2}$, and $\dd_{f_1f_2,g_3}$ exist. For $D_1$, note that $g_1$ is the product of the $x_{\{g_1\}\cup\sigma}$ subject to $\sigma\subseteq [f]\smallsetminus\{f_1\}$ and $\lvert[f]\cap \sigma\rvert\geq n-2$. Hence for each $f_i$ with $i>1$ the only variable dividing $g_1$ and not $f_i$ is  $x_{[f]\oplus\{g_1,f_1,f_i\}}$. In particular, for $1<i<j$ we indeed have $g_1\vert \lcm(f_i,f_j)$.
    For $D_2$, we note that every $x_{\{g_2\}\cup \sigma}$ with $f_1\in \sigma\subseteq [f]\smallsetminus \{f_2\}$ is present. Further, the only $x_{\{g_2\}\cup \theta}$ present with $f_1\notin\theta$ has $\theta = [f]\smallsetminus\{f_1,f_2\}$. Hence every $i>2$ has $g_2\vert \lcm(f_1,f_i)$. For $D_3$, note that the variables $x_{\{g_3\}\cup \varphi}$ not present are exactly those with $f_1,f_2\notin\varphi$. 
    Thus, we have $g_3\vert \lcm(f_1,f_2)$.

    We thus have a $[g]$-perfect matching $M = H_{g_1}\cup \cdots\cup H_{g_n}\cup D_1\cup D_2\cup D_3 $. We argue $M$ is $[f]$-determined. That is, for any $\e =\e_{\sigma,g_i}\in M$ we can determine which piece of $M$ the edge $\e$ belongs from the data of $\psi = [f]\cap \sigma.$ 

    \begin{itemize}
    \item   If $\psi = \emptyset$, then $\e\in H_{g_1}$.
    \item If $\lvert \psi\rvert = 1$, then $\e\in H_{g_i}$.
    \item If $\lvert \psi\rvert >2$ or $\lvert \psi\rvert=2$ and $f_1\notin \psi$ then $\e\in D_1$.
    \item If $\psi = \{f_1,f_i\}$ for some $i>2$ then $\e\in D_2$.
     \item If $\psi = \{f_1,f_2\}$ then $e\in D_3$.
    \end{itemize}
   We thus have that $M$ is $[f]$-determined. 
    
    The facts that differential edges and homotopy edges use disjoint sets of vertices and that $M$ is $[f]$-determined give that the only possible weights of edges in the auxiliary graph  $\Aux_{\f}^M$ are $\dd_{g_i}^{-1}\dd_{f_j}$ and $\h_{g_i}^{-1}\h_{f_j}$. Clearly no such differential weight may exist as there are no $\dd_{f_i}$. The homotopy weight can exist, but must have a source in $H_{g_1}$. It follows then there are no cycles in $\Aux_{\f}^M$ and hence by Proposition \ref{proposition det by cycle decompositions} $M$ is triangular. 
    We conclude $\chi_1\cdots\chi_n\in \I_{\f}$ and hence $\V_{\f} = V(\chi_1\cdots\chi_n)$.

\end{proof}

%%%%%%%%%%%%%%%%%%%%%%%%%%%%%%%%%%%%

\section{Support Varieties for Cyclic GCD Graphs}\label{section cycle varieties}
 The main objects of study in this section are the edge ideals of cycles and their cohomological support varieties. 
 Recall that \cite{BGP} and \cite{gintz} give examples of these ideals achieving support varieties that are non-linear hypersurfaces.
 In fact, these are the only known support varieties of rings defined by monomial ideals that are nonlinear hypersurfaces.
 In general, the edge ideal construction associates to any simple graph a square-free monomial ideal with degree 2 generators. 
 In particular--and in contrast to GCD graphs--the vertices of the simple graph correspond to variables of a polynomial ring and the generators of the edge ideal are the products of each of the variables corresponding to edges. 
 For $n\geq 3$ let $C_n$ denote the $n$-cycle, the simple graph with $n$ vertices and $n$ edges given by a regular $n$-gon. These graphs have the remarkable property that they are the GCD graphs of their own edge ideals. In particular, ideals with GCD graph $C_n$ forms a family of ideals which contains the edge ideal of $C_n$.
 
 \begin{remark}\label{remark cycle edge ideals and taylro subgraphs}
     Whenever $n>3$, we have that $C_n$ is triangle-free and hence by Corollary \ref{corollary triangle free fiber} there are $2^n$ monomial ideals generated by $\f$ that have $\G_{\f} =C_n$ to consider. Moreover, these ideals are in bijection with the power set on singleton variables, where the empty set corresponds to the edge ideal of $C_n$. Hence by Proposition \ref{proposition Taylor graphs are order respecting}, every $\f$ with $\G_{\f} = C_n$ has that $\T_{\f}$ is a subgraph of the Taylor graph of the edge ideal of $C_n$. Hence the prime directive of this section is understanding the Taylor graphs of edge ideals of cycles. 
 \end{remark}

\begin{remark}\label{remark cycle edges}
    A useful representation of vertices in the Taylor graph will be with binary strings of length $n$, that is, the string $b_1b_2\cdots b_n$ represents the vertex containing exactly the $f_i$ with $b_i=1$. As we are considering the edge ideals of cycles, by definitions and Proposition \ref{proposition homotopically/differentially isolated}, we perfectly understand how differential and homotopy edges behave in $\T_{\f}$. 
    Using appropriate modular arithmetic, note that any $\dd_i^s$ must have $b_{i-1} = b_i = b_{i+1} = 1$, $\h_i^s$ must have $b_{i-1} = b_i = b_{i+1} = 0$, whereas the target vertices have  $b_{i-1} = b_{i+1} = 1-b_i$. For convenience below, we will use $\odot$ to represent concatenation for our binary strings, e.g. ${\tt101}\odot{\tt011=101011}$. 
    Similarly, $b^{\odot n}$ is the concatenation of $b$ with itself $n$ many times.
\end{remark}

\begin{lemma}\label{lemma odd length path in taylor graph}
    Suppose $s>2$ is an odd integer and $\T_{\f}$ contains either of the following balanced walks $\upsilon$ or $\omega$ with $s-1$ steps.
    First, $v_s = \upsilon(v_1)$ where $v_1$ and $v_s$ are sinks with in-degree 1 and $v_j$ has in-degree 2 for every odd $1<j<s$. Alternatively,  $w_s = \omega(w_1)$ such that $w_1$ and $w_s$ are sources with out-degree 1 and $w_j$ has out-degree 2 for every odd $1<j<s$. In both cases, $\V_{\f}$ is full.
\end{lemma}
% https://q.uiver.app/#q=WzAsMTQsWzAsMCwidl8xIl0sWzYsMCwidl9zIl0sWzEsMCwidl8yIl0sWzIsMCwidl8zIl0sWzMsMCwidl80Il0sWzUsMCwidl97cy0xfSJdLFs0LDAsIlxcY2RvdHMiXSxbMCwxLCJ3XzEiXSxbMSwxLCJ3XzIiXSxbMiwxLCJ3XzMiXSxbMywxLCJ3XzQiXSxbNCwxLCJcXGNkb3RzIl0sWzUsMSwid197cy0xfSJdLFs2LDEsIndfcyJdLFsyLDBdLFsyLDNdLFs0LDNdLFs1LDFdLFs1LDYsIiIsMSx7InN0eWxlIjp7ImJvZHkiOnsibmFtZSI6ImRvdHRlZCJ9fX1dLFs0LDYsIiIsMSx7InN0eWxlIjp7ImJvZHkiOnsibmFtZSI6ImRvdHRlZCJ9fX1dLFs3LDhdLFs5LDhdLFs5LDEwXSxbMTEsMTAsIiIsMSx7InN0eWxlIjp7ImJvZHkiOnsibmFtZSI6ImRvdHRlZCJ9fX1dLFsxMywxMl0sWzExLDEyLCIiLDEseyJzdHlsZSI6eyJib2R5Ijp7Im5hbWUiOiJkb3R0ZWQifX19XV0=
\[\begin{tikzcd}
	{v_1} & {v_2} & {v_3} & {v_4} & \cdots & {v_{s-1}} & {v_s} \\
	{w_1} & {w_2} & {w_3} & {w_4} & \cdots & {w_{s-1}} & {w_s}
	\arrow[from=1-2, to=1-1]
	\arrow[from=1-2, to=1-3]
	\arrow[from=1-4, to=1-3]
	\arrow[dotted, from=1-4, to=1-5]
	\arrow[dotted, from=1-6, to=1-5]
	\arrow[from=1-6, to=1-7]
	\arrow[from=2-1, to=2-2]
	\arrow[from=2-3, to=2-2]
	\arrow[from=2-3, to=2-4]
	\arrow[dotted, from=2-5, to=2-4]
	\arrow[dotted, from=2-5, to=2-6]
	\arrow[from=2-7, to=2-6]
\end{tikzcd}\]
\begin{proof}
It suffices to show that $\A_k^{n}\smallsetminus V(\chi_1\chi_2\cdots\chi_n)\subseteq\V_{\f}$, as by \cite[2.5]{BGP} we can assume $k$ is algebraically closed so $\A_k^{n}\smallsetminus V(\chi_1\chi_2\cdots\chi_n)$ is an open dense subset of $\A_k^{n}$. Thus, we assume that $a\in (k^{*})^n$.

In the first setting, note that $v_1\in\ker(\T_{\f}(a))$. 
By assumption the only columns with entries in the rows corresponding to $v_1, v_3, \ldots, v_s$ are the columns corresponding to $v_2,v_4,\ldots, v_{s-1}$.  
Hence, if $v_1\in \im(\T_{\f}(a))$ it must be the image of some $z = c_2v_2 + c_4v_4 + \cdots + c_{s-1}v_{s-1}$ with $c_2\neq 0$. 
If $\alpha_{i,j}$ is the weight of the edge $v_i\rightarrow v_j$ (which is nonzero by assumption), we have that the projection of $\T_{\f}(a)(z)$ to the span of $v_1, v_3, \ldots, v_s$ is $\alpha_{2,1}c_2v_1 + (\alpha_{2,3}c_2 + \alpha_{4,3}c_4)v_3 + \cdots + \alpha_{s-1,s}c_{s-1}v_s.$ By assumption this sum is $v_1$ so it must be that $c_{s-1} = 0$, which in turn implies that $c_{s-3} = 0$ and so on. It follows that $z=0$, a contradiction. Hence $v_1$ is a cycle that is not a boundary and $\V_{\f} = \A_k^n$ in the first setting.

In the second setting we have that for odd $i$, $\T_{\f}(a)(w_i)$
is in the span of $w_{2},w_4,\ldots, w_{s-1}$. Since each image is nonzero and there are more odd vertices than even, there must be a linear dependence among the images. That is, there exists some nonzero cycle $c_1w_1 + c_3w_3+\cdots+c_sw_s$. By a similar argument as the first setting, it must be that every $c_i\neq 0$. In particular this cycle is not a boundary as $w_1$ is a source, but the cycle has a nonzero $w_1$ component.

\end{proof}

\begin{lemma} \label{lemma not 2 mod 4 edge ideals}
    If $n \not\equiv 2 \pmod {4}$ and $\f$ generates the edge ideal of $C_n$, then $\V_{\f} = \A_k^{n}.$
\end{lemma}
\begin{proof}
    Let $n=4m$ and note that none of ${\tt000},$ ${\tt111}$, ${\tt101},$ and ${\tt010}$ appear as a substring in $v={\tt1100}^{\odot m}$. Hence by Remark \ref{remark cycle edges}, $v$ is isolated in $\T_{\f}$.

    When $n$ is odd there are no isolated vertices in the Taylor graph, however we will demonstrate sets of vertices in $\T_{\f}$ satisfying the hypothesis of Lemma \ref{lemma odd length path in taylor graph}. Consider the following procedure. 
    For $n = 4m+1$, take a walk consisting of $4m$ steps on $\T_{\f}$ starting at $w_1 =({\tt0011})^{\odot m}\odot 0$ and ending at $w_{4m+1} = ({\tt1100})^{\odot m}\odot 0$. 
    For $n = 4m+3$, take a walk consisting of $4m+2$ steps on $\T_{\f}$ starting at $v_1=({\tt0110})^m\odot {\tt011}$ and ending at $v_{4m+3} = ({\tt1001})^m\odot {\tt101}$. 
    In both cases, the $i$-th step consists of changing the $i$-th bit. Hence in total the procedure changes the bit at every index except for the last one.
    
    It is easiest to see that these walks satisfy the hypothesis of Lemma \ref{lemma odd length path in taylor graph} by working out the two cases for $m=2$, but we provide an argument nonetheless. Note that $w_1$ and $w_{4m+1}$ are indeed sources with out-degree 1 and $v_1$ and $v_{4m+3}$ are sinks with in-degree 1. For $n=4m+1$ it suffices to note that for $1\leq i\leq m$, vertices of the form $({\tt1100})^{\odot m-i}  \odot ({\tt0011})^{\odot i}\odot 0$ and $({\tt1100})^{\odot m-i}  \odot {\tt1111}\odot  ({\tt0011})^{\odot i-1}\odot 0$
    have out-degree 2 when they are not $w_1$ and $w_{4m+1}$. Similarly for $n=4m+3,$ we check that vertices of the forms $({\tt1001})^{\odot m-i}  \odot ({\tt0110})^{\odot i}\odot {\tt011}$ and $({\tt1001})^{\odot m-i} \odot {\tt1010}\odot  ({\tt0110})^{\odot i-1}\odot {\tt011}$ 
    have in-degree 2 when they are not $v_1$. 
\end{proof}

\begin{corollary}\label{corollary cycle gcds n not 2 mod 4}
If $n\not \equiv 2 \pmod {4}$ and $\G_{\f} = C_n$, then $\V_{\f} = \A_k^{n}.$ 
\end{corollary}
\begin{proof}
    If $n=4m$ we already have an isolated vertex in $\T_{\f}$. If $n$ is odd, note that in the proof of Lemma \ref{lemma not 2 mod 4 edge ideals} $v_{4m+3}$ and $w_1$ are degree 1 vertices using a differential edge in $\T_{\f}$. If a variable corresponding to a singleton is present, then up to relabeling, this differential edge does not exist, giving an isolated vertex.
\end{proof}

    We now consider edge ideals of cycles with $n=4m+2$ vertices. 
    Set $E = \{f_2, f_4, \ldots, f_n\}$ and $O=\{f_1, f_3, \ldots, f_{n-1}\}$. First, we provide a $E$-perfect matching $M$ that is $O$-determined. 
    Intuitively, we construct $M$ by taking edges using the first appearance of ${\tt000}$, ${\tt010}$, ${\tt101}$, or ${\tt111}$ where the middle bit has an even index.
    This matching will fail to be triangular, so our goal is to understand the cycles that appear in $\Aux_{\f}^M.$ 
    Let $b = b_1b_3\cdots b_{n-1}$ be a binary string of length $2m+1$ corresponding to only the odd indexed generators. 
    In particular $b$ will correspond to an $O$-coordinate.
    Set $i$ to be the smallest index such that $b_i = b_{i+2}$ (if no such repetition occurs, it must be that $b_1 = b_{n-1}$, so we take $i=n-1$). 
    Take $\sigma_b = \setbuild{f_j}{b_j=1}\subseteq O$, and if $b_i=0$ we set $M_b = \h_{\sigma_b,i+1}\otimes2^{E\smallsetminus \{f_{i+1}\}}$, otherwise $M_b = \dd_{\sigma_b,i+1}\otimes2^{E\smallsetminus \{f_{i+1}\}}$. 
    Note that $M_b$ contains $2^{2m}$ edges. 
    Hence $M=\bigcup_{b\in \{0,1\}^{2m+1}}M_b$ consists of $2^{2m+2m+1}$ edges and is thus an $E$-perfect matching that is $O$-determined. 
    Note that each $H_j$ and $D_j$ is empty when $j$ is odd, and is comprised of $\max\{1,2^{2m-j/2}\}$ many $M_b$ when $j$ is even.

\begin{example}\label{example hexagon}
In this example we demonstrate our argument for Theorem \ref{theorem cycle edge ideal supports} in the simplest case, the edge ideal of the hexagon. 
This will demonstrate why Lemma \ref{lemma 2 determined matchings are triangular} cannot in general be extended to a matching determined by a 3 element set. 
We depict the coordinate coloring of $\Aux_{\f}^M$ below, decorating each color with its type color in the superscript. 
Moreover, the label of an edge in $\Aux_{\f}^M$ is a function of the coordinate color of its source and target, hence we label the edges accordingly.

\begin{center}
\begin{minipage}{\textwidth}
\centering
\begin{tikzpicture}[scale=1, fatnode/.style={circle, inner sep=0pt, minimum size=20pt, draw}]

  % Nodes
\node[fatnode] (v000) at (-2,0) {$M_{{\tt000}}^{H_2}$};

\node[fatnode] (v001) at (2,-3) {$M_{{\tt001}}^{H_2}$};
\node[fatnode] (v010) at (-5.5,0) {$M_{{\tt010}}^{H_6}$};
\node[fatnode] (v100) at (2,3) {$M_{{\tt100}}^{H_4}$};

\node[fatnode] (v011) at (-2,-3) {$M_{{\tt011}}^{D_4}$};
\node[fatnode] (v101) at (5.5,0) {$M_{{\tt101}}^{D_6}$};
\node[fatnode] (v110) at (-2,3) {$M_{{\tt110}}^{D_2}$};

\node[fatnode] (v111) at (2,0) {$M_{{\tt111}}^{D_2}$};

  % Edges
\draw[-{Stealth}, thick] (v110) edge[bend left = 7] node[midway, right, scale = .9, minimum size = 40pt]{$\h_6^{-1}\dd_1$}  (v010);
\draw[-{Stealth}, thick] (v010) edge[bend left = 7] node[midway, left, scale = .9,minimum size = 40pt]{$\dd_2^{-1}\h_1$}  (v110);

\draw[-{Stealth}, thick] (v000) edge[bend left = 7] node[midway, below, scale = .9]{$\h_6^{-1}\h_3$}  (v010);
\draw[-{Stealth}, thick] (v010) edge[bend left = 7] node[midway, above, scale = .9]{$\h_2^{-1}\dd_3$}  (v000);

\draw[-{Stealth}, thick] (v100) edge[bend left = 7] node[midway, right, scale = .9,minimum size = 40pt]{$\h_2^{-1}\dd_1$}  (v000);
\draw[-{Stealth}, thick] (v000) edge[bend left = 7] node[midway, left, scale = .9,minimum size = 40pt]{$\h_4^{-1}\h_1$}  (v100);

\draw[-{Stealth}, thick] (v100) edge[bend left = 7] node[midway, below, scale = .9]{$\dd_2^{-1}\h_3$}  (v110);
\draw[-{Stealth}, thick] (v110) edge[bend left = 7] node[midway, above, scale = .9]{$\h_4^{-1}\dd_3$}  (v100);

\draw[-{Stealth}, thick] (v011) edge[bend left = 7] node[midway, left, scale = .9, minimum size = 40pt]{$\h_6^{-1}\dd_5$}  (v010);
\draw[-{Stealth}, thick] (v010) edge[bend left = 7] node[midway, right, scale = .9,minimum size = 40pt]{$\dd_4^{-1}\h_5$}  (v011);

\draw[-{Stealth}, thick] (v100) edge[bend left = 7] node[midway, right, scale = .9, minimum size = 40pt]{$\dd_6^{-1}\h_5$}  (v101);
\draw[-{Stealth}, thick] (v101) edge[bend left = 7] node[midway, left, scale = .9,minimum size = 40pt]{$\h_4^{-1}\dd_5$}  (v100);

\draw[-{Stealth}, thick] (v001) edge[bend left = 7] node[midway, below, scale = .9]{$\dd_4^{-1}\h_3$}  (v011);
\draw[-{Stealth}, thick] (v011) edge[bend left = 7] node[midway, above, scale = .9]{$\h_2^{-1}\dd_3$}  (v001);

\draw[-{Stealth}, thick] (v111) edge[bend left = 7] node[midway, right, scale = .9, minimum size = 40pt]{$\dd_4^{-1}\dd_1$}  (v011);
\draw[-{Stealth}, thick] (v011) edge[bend left = 7] node[midway, left, scale = .9,minimum size = 40pt]{$\dd_2^{-1}\h_1$}  (v111);

\draw[-{Stealth}, thick] (v101) edge[bend left = 7] node[midway, below, scale = .9]{$\dd_2^{-1}\h_3$}  (v111);
\draw[-{Stealth}, thick] (v111) edge[bend left = 7] node[midway, above, scale = .9]{$\dd_6^{-1}\dd_3$}  (v101);

\draw[-{Stealth}, thick] (v101) edge[bend left = 7] node[midway, right, scale = .9, minimum size = 40pt]{$\h_2^{-1}\dd_1$}  (v001);
\draw[-{Stealth}, thick] (v001) edge[bend left = 7] node[midway, left, scale = .9,minimum size = 40pt]{$\dd_6^{-1}\h_1$}  (v101);
\end{tikzpicture}\\
\vspace{10pt}
Coordinate Coloring and Edge Labels of $\Aux_{\f}^M$
\end{minipage}

\end{center}

Any cycle in $\Aux_{\f}^M$ will yield a cycle in the graph above. Note that each color class consists of 4 vertices of $\T_{\f}$, which can be specified by the even index entries of a length 6 binary string. Since 2-cycles are impossible, it must be that any cycle traverses $\dd_1$ and $\h_1$ at least once. 
These require $b_6=1$ and $b_6=0$ respectively. This implies that the cycle passes through $M_{{\tt010}}^{H_6}$ and $M_{{\tt101}}^{D_6}$ as entering these colors is the only way to change $b_6$. Let us then consider paths from $M_{{\tt010}}^{H_6}$ to $M_{{\tt101}}^{D_6}$. Every vertex in $M_{{\tt010}}^{H_6}$ has $b_6=0$, so the starting vertex is determined by $b_2$ and $b_4.$ 
If $b_4= 1$, it is impossible to traverse $\dd_{4}^{-1}$ so the top half of the graph must be used. On the other hand, if $b_{4} = 0$, the top half cannot be traversed because that would require applying $\h_4^{-1}$ so the bottom half must be used. 
Similarly, the value of $b_2$ determines whether or not $D_2$ or $H_2$ is visited along the way to $M_{{\tt101}}^{D_6}$. 
Thus there is a bijection between paths from $M_{{\tt010}}^{H_6}$ to $M_{{\tt101}}^{D_6}$ and vertices in $M_{{\tt010}}^{H_6}$. 
Notably, traversing these paths results in landing at the vertex in $M_{{\tt101}}^{D_6}$ corresponding to the complement of the starting point of the path. 
The argument for paths from right to left is entirely symmetric and each such path is given by taking the complement of the corresponding left to right path at each step. 
Hence there are exactly 4 cycles in $\Aux_{\f}^M$ and we note they are all disjoint. 

With the cycles in $\Aux_{\f}^M$ in hand, we use Proposition \ref{proposition det by cycle decompositions} to compute $\det(\T_{\f}^M)$. 
Each cycle involves 12 edges of $\T_{\f}$ and uses every $\dd_i$ and $\h_i$ exactly once. 
Due to the complementary nature of these cycles, each $\dd_i^s$ is the complement of the $\h_i^s$ in the same cycle, and likewise for the targets. 
Notably if $i$ is odd, these have the same sign and if $i$ is even these have opposite signs. 
In total there are $2^4=16$ cycle decompositions. Each cycle is a permutation of 6 elements, hence has sign $-1$. The effect of a cycle on the product of the diagonal is that it replaces the product of $\dd_2,\dd_4,\dd_6,\h_2,\h_4,$ and $\h_6$ with the product $\dd_1,\dd_3,\dd_5,\h_1,\h_3,$ and $\h_5$. Note that the even product involves exactly 3 negative signs, while the odd product has an even number of negatives. Hence when computing the determinant (with the sign of the permutation in mind) each cycle replaces $-\chi_2\chi_4\chi_6$ with $-\chi_1\chi_3\chi_5$ in the product of the diagonal of $\T_{\f}^M$. Since the product of the actual diagonal is, up to a sign, $\chi_{2}^8\chi_4^4\chi_6^4$ we have that $\det(\T_{\f}^M) = \pm \chi_2^4(\chi_1\chi_3\chi_5+\chi_2\chi_4\chi_6)^4$. In particular $\chi_2\chi_4\chi_6(\chi_1\chi_3\chi_5+\chi_2\chi_4\chi_6)\in \I_{\f}$, but then by symmetry so is $\chi_1\chi_3\chi_5(\chi_1\chi_3\chi_5+\chi_2\chi_4\chi_6)$. Their sum is the square of $\chi_1\chi_3\chi_5+\chi_2\chi_4\chi_6$, hence $\V_{\f}\subseteq V(\chi_1\chi_3\chi_5+\chi_2\chi_4\chi_6)$.

We demonstrate a cycle that is not a boundary of $\T_{\f}(a)$ for $a\in V(\chi_1\chi_3\chi_5+\chi_2\chi_4\chi_6)$, for the other containment. 
Let us consider the cycle of $\Aux_{\f}^M$ starting at ${\tt001100}$ in $M_{{\tt010}}^{H_6}$. 

% https://q.uiver.app/#q=WzAsMTIsWzAsMiwiXFx0dHswMDExMDB9Il0sWzAsMCwiXFx0dHsxMDExMDB9Il0sWzIsMCwiXFx0dHsxMTExMDB9Il0sWzQsMCwiXFx0dHsxMTAxMDB9Il0sWzYsMCwiXFx0dHsxMTAwMDB9Il0sWzgsMCwiXFx0dHsxMTAwMTB9Il0sWzgsMiwiXFx0dHsxMTAwMTF9Il0sWzgsNCwiXFx0dHswMTAwMTF9Il0sWzYsNCwiXFx0dHswMDAwMTF9Il0sWzQsNCwiXFx0dHswMDEwMTF9Il0sWzIsNCwiXFx0dHswMDExMTF9Il0sWzAsNCwiXFx0dHswMDExMDF9Il0sWzAsMSwiXFxoXzEiXSxbMiwxLCItXFxkZF8yIiwyXSxbMiwzLCJcXGRkXzMiXSxbNCwzLCJcXGhfNCIsMl0sWzQsNSwiXFxoXzUiXSxbNiw1LCItXFxkZF82IiwyXSxbNiw3LCJcXGRkXzEiXSxbOCw3LCJcXGhfMiIsMl0sWzgsOSwiXFxoXzMiXSxbMTAsMTEsIlxcZGRfNSJdLFsxMCw5LCItXFxkZF80IiwyXSxbMCwxMSwiXFxoXzYiLDJdXQ==
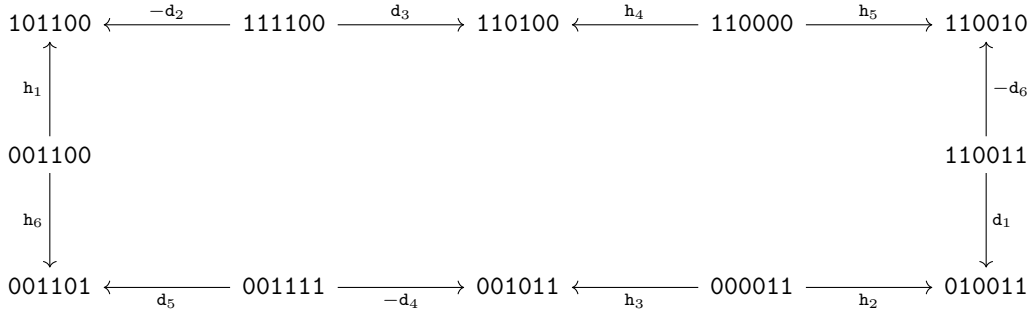
\[\begin{tikzcd}
	{{\tt101100}} && {{\tt111100}} && {{\tt110100}} && {{\tt110000}} && {{\tt110010}} \\
	\\
	{{\tt001100}} &&&&&&&& {{\tt110011}} \\
	\\
	{{\tt001101}} && {{\tt001111}} && {{\tt001011}} && {{\tt000011}} && {{\tt010011}}
	\arrow["{-\dd_2}"', from=1-3, to=1-1]
	\arrow["{\dd_3}", from=1-3, to=1-5]
	\arrow["{\h_4}"', from=1-7, to=1-5]
	\arrow["{\h_5}", from=1-7, to=1-9]
	\arrow["{\h_1}", from=3-1, to=1-1]
	\arrow["{\h_6}"', from=3-1, to=5-1]
	\arrow["{-\dd_6}"', from=3-9, to=1-9]
	\arrow["{\dd_1}", from=3-9, to=5-9]
	\arrow["{\dd_5}", from=5-3, to=5-1]
	\arrow["{-\dd_4}"', from=5-3, to=5-5]
	\arrow["{\h_3}", from=5-7, to=5-5]
	\arrow["{\h_2}"', from=5-7, to=5-9]
\end{tikzcd}\]
Note that every vertex in this cycle is either either a source or sink and in both cases, degree 2. Hence, we have identified a connected component of $\T_{\f}$. In particular $\T_{\f}(a)$ admits a block diagonal structure with one block given by the following matrix.
\[ \begin{pmatrix}
    a_1 & -1 & 0   & 0  & 0   & 0  \\ 
    0   & 1  & a_4 & 0  & 0   & 0  \\ 
    0   & 0  & a_5 & -1 & 0   & 0  \\
    0   & 0  & 0   & 1  & a_2 & 0  \\
    0   & 0  & 0   & 0  & a_3 & -1 \\
    a_6   & 0  & 0   & 0  & 0   & 1 
\end{pmatrix}\]
Notably, the determinant of this matrix is $a_1a_3a_5 + a_2a_4a_6$, hence when $a\in V(\chi_1\chi_3\chi_5+\chi_2\chi_4\chi_6)$ we have a cycle in the span of the corresponding columns. Note these columns correspond to sources in $\T_{\f}$ so this cycle is not a boundary and hence $a\in \V_{\f}$. We conclude that $\V_{\f}= V(\chi_1\chi_3\chi_5+\chi_2\chi_4\chi_6)$.
\end{example}

This brings us to the 4 type C graphs of Section \ref{section 6 generated ideals} depicted below. Just as before, we use dashed vertices to denote the variable required
to be absent for interesting support to occur.

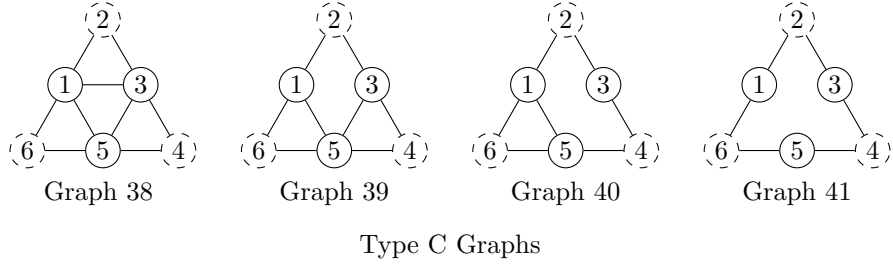
\begin{center}

\begin{minipage}{0.19\textwidth}
\centering
\begin{tikzpicture}[scale=1, every node/.style={circle, inner sep=0pt, minimum size=13pt, draw}]

  % Nodes
  \node (v1) at (-.5,.87) {1};
  \node (v2) at (0,0) {5};
  \node[dashed] (v3) at (0,1.73) {2};
  \node[dashed] (v4) at (1,0) {4};
  \node[dashed] (v5) at (-1,0) {6};
  \node (v6) at (.5,.87) {3};

  % Edges
  \draw (v1) -- (v2);
  \draw (v1) -- (v3);
  \draw (v1) -- (v5);
  \draw (v1) -- (v6);
  \draw (v2) -- (v4);
  \draw (v2) -- (v5);
  \draw (v2) -- (v6);
  \draw (v3) -- (v6);
  \draw (v4) -- (v6);
\end{tikzpicture}
Graph 38 \label{graph 38}
\end{minipage}
\begin{minipage}{0.19\textwidth}
\centering
\begin{tikzpicture}[scale=1, every node/.style={circle, inner sep=0pt, minimum size=13pt, draw}]

  % Nodes
  \node (v1) at (-.5,.87) {1};
  \node (v2) at (0,0) {5};
  \node[dashed] (v3) at (0,1.73) {2};
  \node[dashed] (v4) at (1,0) {4};
  \node[dashed] (v5) at (-1,0) {6};
  \node (v6) at (.5,.87) {3};

  % Edges
  \draw (v1) -- (v2);
  \draw (v1) -- (v3);
  \draw (v1) -- (v5);
  %\draw (v1) -- (v6);
  \draw (v2) -- (v4);
  \draw (v2) -- (v5);
  \draw (v2) -- (v6);
  \draw (v3) -- (v6);
  \draw (v4) -- (v6);
\end{tikzpicture}
Graph 39
\end{minipage}
\begin{minipage}{0.19\textwidth}
\centering
\begin{tikzpicture}[scale=1, every node/.style={circle, inner sep=0pt, minimum size=13pt, draw}]

  % Nodes
  \node (v1) at (-.5,.87) {1};
  \node (v2) at (0,0) {5};
  \node[dashed] (v3) at (0,1.73) {2};
  \node[dashed] (v4) at (1,0) {4};
  \node[dashed] (v5) at (-1,0) {6};
  \node (v6) at (.5,.87) {3};
  % Edges
  \draw (v1) -- (v2);
  \draw (v1) -- (v3);
  \draw (v1) -- (v5);
  %\draw (v1) -- (v6);
  \draw (v2) -- (v4);
  \draw (v2) -- (v5);
  %\draw (v2) -- (v6);
  \draw (v3) -- (v6);
  \draw (v4) -- (v6);
\end{tikzpicture}
Graph 40
\end{minipage}
\begin{minipage}{0.19\textwidth}
\centering
\begin{tikzpicture}[scale=1, every node/.style={circle, inner sep=0pt, minimum size=13pt, draw}]

  % Nodes
  \node (v1) at (-.5,.87) {1};
  \node (v2) at (0,0) {5};
  \node[dashed] (v3) at (0,1.73) {2};
  \node[dashed] (v4) at (1,0) {4};
  \node[dashed] (v5) at (-1,0) {6};
  \node (v6) at (.5,.87) {3};

  % Edges
  %\draw (v1) -- (v2);
  \draw (v1) -- (v3);
  \draw (v1) -- (v5);
  %\draw (v1) -- (v6);
  \draw (v2) -- (v4);
  \draw (v2) -- (v5);
  %\draw (v2) -- (v6);
  \draw (v3) -- (v6);
  \draw (v4) -- (v6) ;
\end{tikzpicture}
Graph 41
\end{minipage}
\\
\vspace{10pt} 
Type C Graphs
\end{center}
\begin{corollary}\label{corollary gcd fiber over hexagon}
    If, up to relabeling, $\G_{\f}$ is Type C and $\f$ does not generate the edge ideal of the hexagon, then $\V_{\f} = V(\chi_2\chi_4\chi_6)$ if the dashed variables are not present, otherwise $\V_{\f} = \A_k^{6}.$
\end{corollary}

\begin{proof}
    We first consider the cases when the dashed variables are absent. A combination of Lemmas \ref{lemma leafs and edges} and \ref{lemma variables present for degree 2 vertex} give that each variable corresponding to an edge of graph 41 is present. In particular, then by Proposition \ref{proposition Taylor graphs are order respecting} the Taylor graphs in each of these cases form an interval with maximal element given by the edge ideal of the hexagon and minimal element given by the ideal with GCD graph $\Delta(3)$ (i.e. graph 38) in Theorem \ref{theorem union of hyperplanes}. We observe that the $O$-determined $E$-perfect matching of Example \ref{example hexagon} coincides with the $[f]$-determined $[g]$-perfect matching of Theorem \ref{theorem union of hyperplanes} when we set vertex 2 of graph 38 to be $g_1$ of $\Delta(3)$. Call this matching $M$ and note that the containment $\V_{f}\subseteq V(\chi_2\chi_4\chi_6)$ can be verified by showing there are no cycles in $\Aux_{\f}^M.$ Up to symmetry, there are four variables that could be added to the edge ideal of the hexagon: $x_1,$ $x_{15},$ $x_{135}$, and $x_{156}$. If $x_1$ is present, then no $\dd_1$ is present, breaking each cycle. If any of the other three are present, then $f_1$ and $f_5$ are neighbors in $\G_{\f}$. Because of this, the $\dd_{6}^{-1}\h_5$ and $\dd_{6}^{-1}\h_1$ edges entering $M_{{\tt101}}^{D_6}$ do not exist. The only cycle that avoids these edges traverses $\dd_{2}^{-1}\h_1$ from $M_{{\tt011}}^{D_4}$ to $M_{{\tt111}}^{D_2}$, which now also does not exist. 
    It follows that $M$ is triangular and hence 
$\V_{\f}\subseteq V(\chi_2\chi_4\chi_6)$.

    It now suffices to show $V(\chi_2\chi_4\chi_6)\subseteq \V_{\f}$. 
    We do so with Proposition \ref{proposition homotopy soure/sink gives subvariety} by supplying a homotopy source/sink for each of $\chi_2$, $\chi_4$, and $\chi_6$.

% \usepackage{multirow}
\begin{table}[H]
\begin{tabular}{c|c|c|c|c|c|}
\cline{2-6}
    & $x_\sigma$ Present & $\chi_2$ Source                & $\chi_2$ Sink     & $\chi_4$ Source                & $\chi_6$ Sink \\ \cline{2-6} 
    & $x_1$              &                                & ${\tt010011}$     & ${\tt110001}$      & ${\tt011001}$ \\ \hline
\multicolumn{1}{|c|}{\multirow{3}{*}{$f_1$, $f_5$ neighbors}} & $x_{15}$           & \multirow{3}{*}{${\tt000110}$} & \multirow{3}{*}{} & \multirow{3}{*}{${\tt110000}$} & ${\tt001101}$ \\ \cline{2-2} \cline{6-6} 
\multicolumn{1}{|c|}{}                                        & $x_{135}$          &                                &                   &                                & ${\tt001001}$ \\ \cline{2-2} \cline{6-6} 
\multicolumn{1}{|c|}{}                                        & $x_{156}$          &                                &                   &                                & ${\tt001101}$ \\ \hline
\end{tabular}
\end{table}

Finally, we consider when (up to relabeling) an even singleton variable is present. 
In particular, for graph 41 this means that there exists a pair of singleton variables of different parity $x_i$, $x_j$ that are present. 
It suffices to consider the pairs $x_1, x_2$ and $x_1,x_4$ being present. 
Whence ${\tt111001}$ and ${\tt011011}$ are isolated vertices of $\T_{\f}$ respectively. For graphs 38-40, $f_{1}$ and $f_{5}$ are neighbors and it suffices to consider when $x_2$ is present or when $x_6$ is present. We note ${\tt111000}$ and ${\tt100110}$ are isolated in the respective cases. Hence, if $\G_{\f}$ is type C but not as prescribed in the Corollary statement, then $\V_{\f} = \A_k^{6}.$
\end{proof}

\begin{remark}\label{remark equigenerated}
    Recall that \cite{gintz} considers equigenerated monomial ideals. In our framework, these can be formulated by assigning a degree in $\N$ to each variable, taking 0 if the variable is not present and a positive degree otherwise. The condition that $\f$ generates an equigenerated ideal is equivalent to $\sum_{i\in \sigma} \deg(x_\sigma)$ being a constant function of $i$. 
\end{remark}

\begin{corollary}\label{corllary no equigenerated 3 hyperplanes}
    No monomial ideal with 6 minimal generators $\f$ with a support variety $\V_{\f}$ that is the union of 3 hyperplanes is equigenerated.
\end{corollary}
\begin{proof}
    These ideals only occur for type C GCD graphs where up to relabeling, at least one of $x_1$, $x_{15}$, $x_{135},$ or $x_{156}$ is present and, $x_2$, $x_4,$ and $x_6$ are absent. We note that this implies that the $x_\sigma$ with positive degree have at most as many even entries of $\sigma$ as odd entries, and at least one $x_\sigma$ with positive degree has more odd entries. 
    It follows that the mean degree of an odd vertex is strictly lager than the mean degree of even vertices, hence the ideal is not equigenerated.
\end{proof}

\begin{remark}\label{remark Delta n and C_2n not generally an interval}
The fact that the Taylor graphs realizing interesting support formed an interval between ideals with GCD graphs $\Delta(3)$ and $C_6$ was extremely useful in proving Corollary \ref{corollary gcd fiber over hexagon}. One might hope for similar a phenomenon to occur between our ideals on $\Delta(n)$ and $C_{2n}$ for $n>3$. 
However, $g_1$ in $\Delta(n)$ cannot be the generator of an edge ideal, as every present $x_\sigma$ dividing $g_1$ has $\lvert\sigma\rvert>2$. Since $g_1$ must be divisible by at least 2 variables, the degree of $g_1$ is at least 6 for any Taylor graph larger than the one given by our ideal for $\Delta(n)$. 
\end{remark}

We return to the case of that $\f$ generates the edge ideal of $C_n$ for $n=4m+2$. 
With our $O$-determined $E$-perfect matching $M$ in hand, let us recursively construct an undirected version of the graph depicted in Example \ref{example hexagon} with vertex set $\setbuild{M_b}{b\in \{{\tt 0,1}\}^{2m+1}}$. 
Each vertex of this graph is a coordinate color class of $\Aux_{\f}^M$, so the edges of this graph will represent the possible edges between these color classes. 

\begin{definition} \label{definition G(p,s)}
Let $p\geq -1$ be an integer called the \emph{prefix length} and $s$ a binary string with even length called the \emph{suffix}. 
The graph $G(p,s)$ is defined recursively as follows. 
If $p = -1$ then  $G(p,s)$ has two vertices $M_{{\tt0}\odot s},$ $M_{{\tt1}\odot s}$ and the edge between them. 
Otherwise let $r = {\tt 0\odot(10)}^{\odot p}$ and $\tilde{r} = {\tt 1\odot(01)}^{\odot p}$ and define $G(p,s)$ as the graph below. 
We call $M_{r\odot{\tt10}\odot s}$ and $M_{\tilde r\odot{\tt01}\odot s}$ the left and right endpoints, respectively. 
Each dashed edge represents a subgraph with a smaller prefix length whose left and right endpoints are exactly the endpoints of the dashed edge. 
There are $2\cdot 4^{p+1}$ vertices in $G(p,-)$ which becomes quite large, even for small $p$. 
We depict $G(1,\emptyset)$ in Example \ref{example G(p,s)}.
\vspace{3pt}
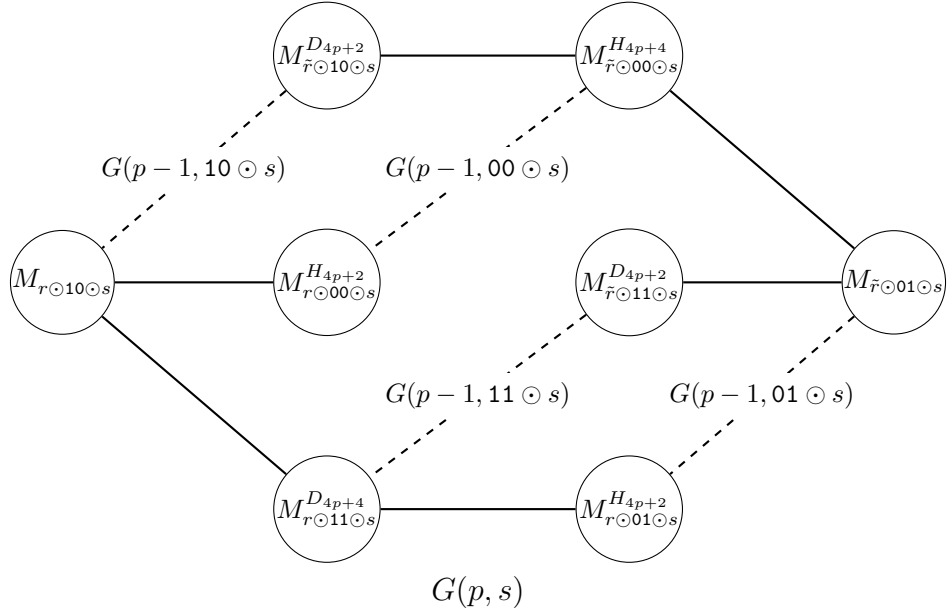
\begin{center}
\begin{minipage}{0.9\textwidth}
\centering
\begin{tikzpicture}[scale=1, fatnode/.style={circle, inner sep=0pt, minimum size=20pt, draw}]

  % Nodes
\node[fatnode] (v000) at (-2,0) {$M_{r\odot {\tt00}\odot s}^{H_{4p+2}}$};

\node[fatnode] (v001) at (2,-3) {$M_{r\odot {\tt01}\odot s}^{H_{4p+2}}$};
\node[fatnode] (v010) at (-5.5,0) {$M_{r\odot {\tt10}\odot s}^{}$};
\node[fatnode] (v100) at (2,3) {$M_{\tilde{r}\odot {\tt00} \odot s}^{H_{4p+4}}$};

\node[fatnode] (v011) at (-2,-3) {$M_{r\odot {\tt11}\odot s}^{D_{4p+4}}$};
\node[fatnode] (v101) at (5.5,0) {$M_{\tilde{r}\odot {\tt01}\odot s}$};
\node[fatnode] (v110) at (-2,3) {$M_{\tilde{r}\odot {\tt10}\odot s}^{D_{4p+2}}$};

\node[fatnode] (v111) at (2,0) {$M_{\tilde{r}\odot {\tt11}\odot s}^{D_{4p+2}}$};

  % Edges
\draw[dashed, thick] (v110) edge node[midway, fill=white ]{$G(p-1,{\tt10}\odot s)$}  (v010);

\draw[dashed, thick] (v100) edge node[midway, fill=white ]{$G(p-1,{\tt00}\odot s)$}  (v000);

\draw[dashed, thick] (v111) edge node[midway, fill=white ]{$G(p-1,{\tt11} \odot s)$}  (v011);
\draw[dashed, thick] (v101) edge node[midway, fill=white ]{$G(p-1,{\tt01}\odot s)$}  (v001);

\draw[thick] (v000) --  (v010);

\draw[thick] (v100) --  (v110);

\draw[thick] (v011) --  (v010);

\draw[thick] (v100) --  (v101);

\draw[thick] (v001) edge (v011);

\draw[thick] (v101) --  (v111);

\end{tikzpicture}\\

\Large{$ G(p,s)$}
\end{minipage}

\end{center}
\vspace{3pt}
\end{definition}

\begin{example} \label{example G(p,s)}
    The second easiest case to consider is $G(0,\emptyset),$ which recovers an undirected version of the graph in Example \ref{example hexagon}. Below we have $G(1,\emptyset)$ which will correspond to the edge ideal of the decagon.
    Each vertex is decorated with a superscript denoting the type color, $H_i$ or $D_j$, it has in $\Aux_{\f}^M$.
    The circles indicate the vertices given by the top level of recursion, which then form the left and right endpoints of some $G(0,s)$ in the second level of recursion. 
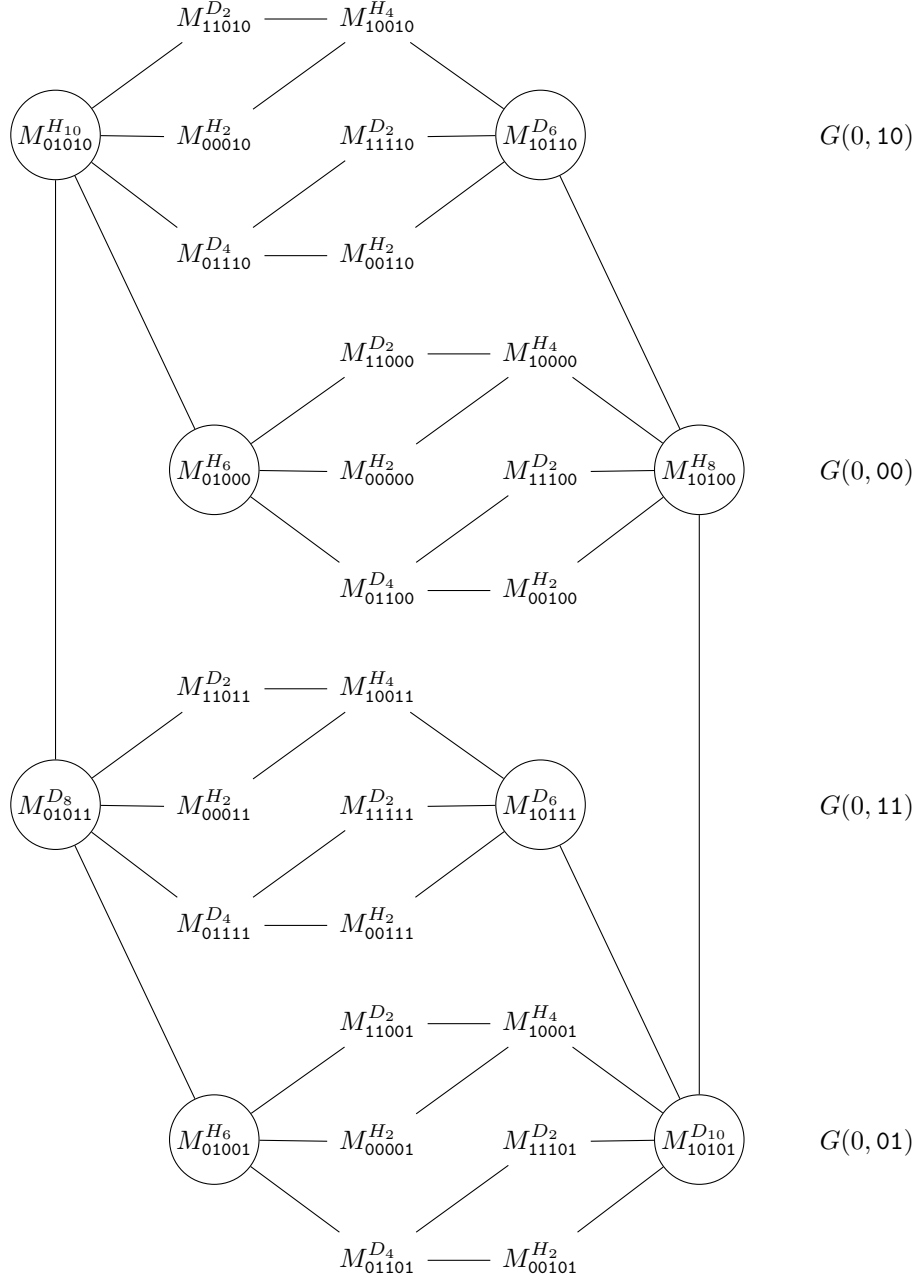
\[\begin{tikzcd}
	& {M_{{\tt11010}}^{D_2}} & {M_{{\tt10010}}^{H_4}} &&& \\
	|[draw, circle, inner sep=.1em]|{M_{{\tt01010}}^{H_{10}}} & {M_{{\tt00010}}^{H_2}} & {M_{{\tt11110}}^{D_2}} & |[draw, circle,inner sep=.1em]|{M_{{\tt10110}}^{D_6}} && {G(0,{\tt10})} \\
	& {M_{{\tt01110}}^{D_4}} & {M_{{\tt00110}}^{H_2}} \\
	&& {M_{{\tt11000}}^{D_2}} & {M_{{\tt10000}}^{H_4}} \\
	& |[draw, circle,inner sep=.1em]|{M_{{\tt01000}}^{H_6}} & {M_{{\tt00000}}^{H_2}} & {M_{{\tt11100}}^{D_2}} & |[draw, circle,inner sep=.1em]|{M_{{\tt10100}}^{H_8}} & {G(0,{\tt00})} \\
	&& {M_{{\tt01100}}^{D_4}} & {M_{{\tt00100}}^{H_2}} \\
	& {M_{{\tt11011}}^{D_2}} & {M_{{\tt10011}}^{H_4}} \\
	|[draw, circle,inner sep=.1em]|{M_{{\tt01011}}^{D_8}} & {M_{{\tt00011}}^{H_2}} & {M_{{\tt11111}}^{D_2}} & |[draw, circle,inner sep=.1em]| {M_{{\tt10111}}^{D_6}} && {G(0,{\tt11})} \\
	& {M_{{\tt01111}}^{D_4}} & {M_{{\tt00111}}^{H_2}} \\
	&& {M_{{\tt11001}}^{D_2}} & {M_{{\tt10001}}^{H_4}} \\
	& |[draw, circle,inner sep=.1em]|{M_{{\tt01001}}^{H_6}} & {M_{{\tt00001}}^{H_2}} & {M_{{\tt11101}}^{D_2}} & |[draw, circle,inner sep=.1em]|{M_{{\tt10101}}^{D_{10}}} & {G(0,{\tt01})} \\
	&& {M_{{\tt01101}}^{D_4}} & {M_{{\tt00101}}^{H_2}}
	\arrow[no head, from=1-2, to=1-3]
	\arrow[no head, from=1-2, to=2-1]
	\arrow[no head, from=1-3, to=2-2]
	\arrow[no head, from=1-3, to=2-4]
	\arrow[no head, from=2-1, to=2-2]
	\arrow[no head, from=2-1, to=3-2]
	\arrow[no head, from=2-1, to=8-1]
	\arrow[no head, from=2-3, to=2-4]
	\arrow[no head, from=2-3, to=3-2]
	\arrow[no head, from=2-4, to=3-3]
	\arrow[no head, from=3-2, to=3-3]
	\arrow[no head, from=4-3, to=4-4]
	\arrow[no head, from=4-3, to=5-2]
	\arrow[no head, from=4-4, to=5-3]
	\arrow[no head, from=4-4, to=5-5]
	\arrow[no head, from=5-2, to=2-1]
	\arrow[no head, from=5-2, to=5-3]
	\arrow[no head, from=5-2, to=6-3]
	\arrow[no head, from=5-4, to=5-5]
	\arrow[no head, from=5-4, to=6-3]
	\arrow[no head, from=5-5, to=2-4]
	\arrow[no head, from=5-5, to=6-4]
	\arrow[no head, from=5-5, to=11-5]
	\arrow[no head, from=6-3, to=6-4]
	\arrow[no head, from=7-2, to=7-3]
	\arrow[no head, from=7-2, to=8-1]
	\arrow[no head, from=7-3, to=8-2]
	\arrow[no head, from=7-3, to=8-4]
	\arrow[no head, from=8-1, to=8-2]
	\arrow[no head, from=8-1, to=9-2]
	\arrow[no head, from=8-3, to=8-4]
	\arrow[no head, from=8-3, to=9-2]
	\arrow[no head, from=8-4, to=9-3]
	\arrow[no head, from=9-2, to=9-3]
	\arrow[no head, from=10-3, to=10-4]
	\arrow[no head, from=10-3, to=11-2]
	\arrow[no head, from=10-4, to=11-3]
	\arrow[no head, from=10-4, to=11-5]
	\arrow[no head, from=11-2, to=8-1]
	\arrow[no head, from=11-2, to=11-3]
	\arrow[no head, from=11-2, to=12-3]
	\arrow[no head, from=11-4, to=11-5]
	\arrow[no head, from=11-4, to=12-3]
	\arrow[no head, from=11-5, to=8-4]
	\arrow[no head, from=11-5, to=12-4]
	\arrow[no head, from=12-3, to=12-4]
\end{tikzcd}\]
\end{example}
\begin{lemma}
    If $\f$ generates the edge ideal of a $(4m+2)$-cycle, then each vertex of $G(m-1,\emptyset)$ is a coordinate color class of $\Aux_{\f}^M$ and moreover, for every edge of $\Aux_{\f}^M$ there is an edge between the color classes of its vertices in $G(m-1,\emptyset)$.
\end{lemma}
\begin{proof}
Since $M$ is an $O$-determined, $E$-perfect matching, we have that every edge in $\Aux_{\f}^M$ has a weight $\e_i^{-1}\e_j$ where $i$ is even and $j$ is odd. The containment of $M_b$ inside $H_\ell$ or $D_\ell$ is determined by the location of the first $\tt00$ or $\tt11$ in appearing in $b$. 
It follows that any edge with weight $\e_i^{-1}\e_j$ in $\Aux_{\f}^M$ and source in $H_\ell$ or $D_\ell$ must have that $j\leq \ell+1$ in order to change the type color.
Moreover, edges in $\Aux_{\f}^M$ between coordinate colors $M_b$ and $M_{b'}$ must have that $b$ and $b'$ differ in exactly one entry.
It suffices to show that the edges of $G(m-1,\emptyset)$ are exactly those satisfying the above criteria. 

By construction it is clear that every edge of $G(m-1,\emptyset)$ is between color classes differing by one coordinate.
It is also easy to see that edges of $G(m-1,\emptyset)$ involve changing a coordinate at or before the first repetition. 
The result follows by showing every vertex $M_b$ with type color $H_i$ or $D_i$ has degree $i/2+1$.
We observe this is true for the initial left and right endpoints $M_{{\tt1}\odot ({\tt01})^{\odot2m}}$ and $M_{{\tt0}\odot ({\tt10})^{\odot2m}}$ as each recursive call adds two edges with the exception of the base case adding just one edge.
For other vertices, consider the first appearance (largest prefix length $p$) of a given vertex. 
This vertex must have type color in $\{H_{4p+2},D_{4p+2}, H_{4p+4}, D_{4p+4}\}$ and be an internal vertex (i.e. not a left/right endpoint) in the depiction of $G(p,s)$, whereas in all subsequent recursive calls, the vertex is a left or right endpoint. 
The subsequent recursive calls result in adding $2p+1$ edges. 
Notably vertices with type color with subscript $4p+2$ have 1 edge coming from the first appearance, and type color with subscript $4p+4$ have 2 such edges.
In both cases, this gives the desired degrees.
\end{proof}

\begin{theorem}\label{theorem cycle edge ideal supports}
    Let $R=k[x_1,\ldots,x_n]/(x_1x_2,x_2x_3,\ldots,x_n x_1)$, the ring defined by the edge ideal of an n-cycle. Then 
    \[\V_{\f} = \begin{cases}
    V(\chi_1\chi_3\ldots\chi_{n-1} + \chi_2\chi_4\ldots\chi_{n}) & \text{if} \ n \equiv {2 \pmod {4}} \\
    \A^n_k & otherwise.
\end{cases}\]
\end{theorem}
\begin{proof}
By Lemma \ref{lemma not 2 mod 4 edge ideals}, it suffices to just consider the case when $n = 4m+2$ for $m \geq 1$. 
We closely follow the same strategy outlined in Example \ref{example hexagon}. 
Our first goal is to classify the $M$-walks on the Taylor graph that act as the identity. 
If $\h_i^{-1}$ appears in such an $M$-walk, it must be that $\dd_i^{-1}$ also appears. 
Note that this forces $\h_1$ and $\dd_1$ to be traversed on the walk as the coordinate colors contained in $H_i$ and $D_i$ differ in the first coordinate.
In turn, $\h_{n}$ and $\dd_{n}$ must also be traversed.
Notably both $H_n$ and $D_n$ consist of exactly one coordinate color each given by $M_{ {\tt0\odot (10)}^{\text{\tiny$\odot2m$}}}$ and $M_{ {\tt1\odot (01)}^{\text{\tiny$\odot2m$}}}$ respectively.
We have argued that any cycle of $\Aux_{\f}^M$ must pass through these two color classes, hence we consider paths between them.

For any $G(p,s)$, we wish to show that each vertex of $\Aux_{\f}^M$ with coordinate color $M_{r\odot {\tt10} \odot s}$ (the left endpoint) has exactly one vertex of color $M_{\tilde{r}\odot {\tt01} \odot s}$ (the right endpoint) reachable by a (unique) path contained in $G(p,s)$.
The argument is by induction on $p$, and the claim is obvious for $p=-1$.
Note that $r$ and $\tilde{r}$ do not contain ${\tt00}$ or ${\tt11}$ as substrings and hence every vertex drawn has type color $H_i$ or $D_i$ with $i\geq 4p+2.$ 
Moreover, the type color of the endpoints $M_{r\odot {\tt10} \odot s}$ and $M_{\tilde{r}\odot {\tt01} \odot s}$ depends on $s$ but the index of this type color is larger than $4p+4.$
Hence, the type colors depicted in Definition \ref{definition G(p,s)} are correct. 
It follows that if $f_{4p+2}$ is not in a vertex in the left endpoint of $G(p,s)$ then $G(p-1,{\tt00}\odot s)$ cannot be traversed. 
Similarly,
if $f_{4p+2}$ is in a vertex in the left endpoint of $G(p,s)$ then $G(p-1,{\tt10}\odot s)$ cannot be traversed.  
Furthermore, whether or not $f_{4p+2}$ is in a vertex in $M_{r\odot {\tt11}\odot s}^{D_{4p+4}}$ excludes one of $G(p-1,{\tt11}\odot s)$ or $G(p-1,{\tt01}\odot s)$ from being traversed. 
Moreover, if a vertex in the left endpoint does not contain $f_{4p+4}$, then no path contained in $G(p,s)$ starting from this vertex can land at $M_{\tilde{r}\odot {\tt00}\odot s}^{H_{4p+4}}$.
Hence, if there is an edge from a vertex in $\Aux_{\f}^M$ with coordinate color given by the left endpoint to a vertex with coordinate color $M_{r\odot {\tt11}\odot s}^{D_{4p+4}}$, it must be traversed. 
This gives uniqueness of a path, should one exist.
We note that in all four cases of $f_{4p+2}$ and $f_{4p+4}$ being present or not in a vertex of the left endpoint, a path to a left endpoint of some $G(p-1,-)$ exists. 
Applying the inductive hypothesis yields our desired path.
Note that this $M$-walk results in taking the complement of the prefix.

Applying this to $G(m-1,\emptyset)$ gives that every vertex of $D_{n}$ can be reached by a (unique) path from a unique vertex in $H_n$ and that these vertices are complements. 
By symmetry, the same is true when swapping the roles of $H_n$ and $D_n$.
Hence we have shown there are exactly $\lvert H_n\rvert = 2^{2m}$ cycles in $\Aux_{\f}^M$ and they are all disjoint.
Further, each cycle traverses each even $\dd_i^{-1}$ and $\h_i^{-1}$ along with each odd $\dd_j$ and $\h_j$ exactly once. 
The argument that each cycle used in a cycle decomposition multiplies the product of the diagonal of $\T_{\f}^M$ by ${\chi^O}/{\chi^E}$ is the exact same as in Example \ref{example hexagon}.
It follows that $\chi^E(\chi^O+\chi^E) \in \I_{\f}$. 
Hence by symmetry $\V_{\f}\subseteq V(\chi^O + \chi^E)$.

For the other containment, we note that the $M$-walk \[({\tt0011})^{\odot m}\odot {\tt00} = \h_{n}^{-1}\dd_{n-1}\cdots\dd_4^{-1}\h_3\h_2^{-1}\dd_1 \dd_{n}^{-1}\h_{n-1}\cdots\h_4^{-1}\dd_3\dd_2^{-1}\h_1\left(({\tt0011})^{\odot m}\odot {\tt00}\right)\] 
traverses every edge of a connected component of the Taylor graph. When $a\in V(\chi^O+\chi^E)$ the corresponding block of $\T_{\f}(a)$ admits a cycle that is not a boundary.
Consequently, we conclude that $\V_{\f} =  V(\chi^O+~\chi^E).$
\end{proof}

\begin{corollary}\label{corollary gcd graph a cycle}
    If $\G_{\f} = C_n$ and $\V_{\f}\neq  \A_k^{n}$, then  $n=4m+2$ (for some $m$) and, up to relabeling, $\V_{\f}$ is either $V(\chi^O +\chi^E)$ or $V(\chi^E)$.
\end{corollary}

\begin{proof}
    By Corollary \ref{corollary cycle gcds n not 2 mod 4} and Theorem \ref{theorem cycle edge ideal supports} we need only consider the case $n=4m+2$ and at least one singleton variable is present.
    Recall that by Proposition \ref{proposition Taylor graphs are order respecting} we have that $\T_{\f}$ is a subgraph of the Taylor graph for the edge ideal of $C_n$.
    We first consider when $x_i$ and $x_j$ are present with different parity. 
    Note that this removes differential edges of different parity in the connected component of $\T_{\f}$ containing $({\tt0011})^{\odot m}\odot {\tt00}$. 
    This results in the existence of at least one walk satisfying the hypothesis of Lemma \ref{lemma odd length path in taylor graph}, hence $\V_{\f} = \A_k^n$. Thus, we may assume that all $x_i$ present have the same parity, and by relabeling, we may further assume that they are all odd.
    Without loss of generality, we may assume that $x_1$ is present, so no $\dd_1$ exists.
    Thus there are no cycles in the auxiliary graph.
    We remark that the $E$-perfect matching $M$ remains and is now triangular, hence by Proposition \ref{proposition det by cycle decompositions}, $V_{\f}\subseteq V(\chi^E)$.
    For the other containment, we have that for every $0\leq s \leq m$, the vertex $ \tt (0110)^{\odot s}\odot 01 \odot (0011)^{\odot m-s}$ is a homotopy sink for $\chi_{4s+2}$ and that $ \tt(1100)^{\odot s} \odot 01\odot(1001)^{\odot m-s}$ is a homotopy source for $\chi_{4s}$.
    Thus every even indexed $V(\chi_i)$ is contained in $\V_{\f}$ and hence we have $\V_{\f} = V(\chi^E)$.
\end{proof}

%%%%%%%%%%%%%%%%%%%%%%%%%%%%%
\section{Open Questions}
We pose the following questions in light of the work presented in this paper.

\begin{question}
    Can the cohomological support varieties of 7 and 8 generated monomial ideals be classified by our approach, at least computationally?
\end{question}

\begin{question} Are the edge ideals of $(4m+2)$-gons the only monomial ideals that have a cohomological support variety that is not the union of coordinate subspaces?
\end{question}

\begin{question}
   Suppose a monomial ideal has a connected GCD graph. Let $n$ be the number of minimal generators of the monomial ideal, $c$ the codimension of $\V_{\f}$, and $d$ the number of irreducible components of $\V_{\f}$. Do the inequalities $d\leq \frac{n-3}{2}$ and $c\leq \frac{n}{2}$ hold?
\end{question}

%%%%%%%%%%%%%%%%%%%%%%%%%%%%%
%%%%%%%%%%%%%%%%%%%%%%%%%%%%%
\section*{Acknowledgments}

We are grateful to Brian Harbourne, Eloísa Grifo, Mark Walker, and Michael Gintz for several valuable discussions. 
We would like to thank Tyler Dang, 
Brian Lopez Medina,
Andy Nguyen, 
Henry Nguyen, and
Gabriela Portales for their work in the 2025 UNL FGFY REU which inspired this project. 
The authors KF, JF, BK, SS, and RW were supported by NSF grant DMS-2342256. 
Additionally, KMS was supported by NSF Grant DMS-2236983.

%%%%%%%%%%%%%%%%%%%%%%%%%%%%%

\bibliographystyle{alpha}
\bibliography{references.bib}

\end{document}